%% file: HansenMuller2_revision.tex
\documentclass[11pt]{article}

  \usepackage{a4wide}
  \usepackage{amsmath}
  \usepackage{amssymb}
  \usepackage{latexsym}
  \usepackage[pdftex]{graphicx}
  \usepackage[final]{pdfpages}
  \usepackage{subcaption}
  \usepackage{color}
  \usepackage[mathscr]{euscript}
  \usepackage{float}
  \usepackage{url}
  \usepackage{tikz}
  \usepackage{ wasysym }
  \usepackage{dsfont}

  \newenvironment{proof}{\vspace{1ex}\noindent{\bf Proof.}}{\hspace*{\fill}$\blacksquare$\vspace{1ex}}
  \newenvironment{proofof}[1]{\vspace{1ex}\noindent{\bf Proof of #1.}}{\hspace*{\fill}$\blacksquare$\vspace{1ex}}

  \newtheorem{theorem}{Theorem} 
  \newtheorem{lemma} [theorem] {Lemma}
  \newtheorem{corollary} [theorem] {Corollary}
  \newtheorem{proposition} [theorem] {Proposition}
  \newtheorem{definition} [theorem] {Definition}
  \newtheorem{conjecture} [theorem] {Conjecture}
  \newtheorem{problem} [theorem] {Problem}
  \newtheorem{fact} [theorem] {Fact}
 
\newcommand{\Ecal}[0]{\ensuremath{{\mathcal E}}}

\newcommand{\Ucal}[0]{\ensuremath{{\mathcal U}}}

\newcommand{\Xcal}[0]{\ensuremath{{\mathcal X}}}

\newcommand{\Zcal}[0]{\ensuremath{{\mathcal Z}}}
\newcommand{\eR}[0]{\ensuremath{\mathbb R}}
\newcommand{\Haa}[0]{\ensuremath{\mathbb H}}

\newcommand{\eN}[0]{\ensuremath{ \mathbb N}}
\newcommand{\Zed}[0]{\ensuremath{ \mathbb Z}}

\newcommand{\Dee}[0]{\ensuremath{\mathbb D}}

\newcommand{\Xtil}[0]{\tilde{X}}

\newcommand{\Escr}[0]{\ensuremath{{\mathscr E}}}

\newcommand{\Lscr}[0]{\ensuremath{{\mathscr L}}}

\newcommand{\Tscr}[0]{\ensuremath{{\mathscr T}}}
\newcommand{\Uscr}[0]{\ensuremath{{\mathscr U}}}

\newcommand{\norm}[1]{\ensuremath{\|#1\|}}

\newcommand{\Pee}[0]{\ensuremath{{\mathbb P}}}

\newcommand{\Ee}[0]{\ensuremath{{\mathbb E}}}

\newcommand{\isd}[0]{\hspace{.2ex} \raisebox{-.1ex}{$=$} \hspace{-1.5ex} 
\raisebox{1ex}{{$\scriptstyle d$}} \hspace{.8ex} }

 \newcommand{\eps}{\varepsilon}

\DeclareMathOperator{\dist}{dist}

\DeclareMathOperator{\diam}{diam}

\DeclareMathOperator{\area}{area}

\DeclareMathOperator{\dd}{d}
\definecolor{orange}{RGB}{255,127,0}
\definecolor{pink}{RGB}{255,150,150}

\DeclareMathOperator{\sector}{sect}

\newcommand{\ballR}[0]{\ensuremath{B_{\eR^2}}}
\newcommand{\ballH}[0]{\ensuremath{B_{\Haa^2}}}

\newcommand{\distH}[0]{\ensuremath{\dist_{\Haa^2}}}
\newcommand{\areaH}[0]{\ensuremath{\area_{\Haa^2}}}

\newcommand{\Zcalb}[0]{\ensuremath{\Zcal_{\text{b}}}}
\newcommand{\Zcalw}[0]{\ensuremath{\Zcal_{\text{w}}}}

\newcommand{\diamH}[0]{\ensuremath{\diam_{\Haa^2}}}

\newcommand{\sect}[3]{\sector\left(#1,#2,#3\right)}

\let\inf\relax \DeclareMathOperator*\inf{\vphantom{p}inf}
\DeclareMathOperator{\ahd}{ahd}
\DeclareMathOperator{\arcsinh}{arcsinh}

\newcommand{\BGab}{\ensuremath{B_{\text{Gab}}}}

\newcommand{\DD}{\ensuremath{DD}}

\DeclareMathOperator{\cert}{cert}

\begin{document}

\title{Poisson-Voronoi percolation in the hyperbolic plane \\ with small intensities}

\author{%
Benjamin T.~Hansen\thanks{Bernoulli Institute, 
Groningen University, The Netherlands. 
E-mail: {\tt b.t.hansen@rug.nl}.  Supported by the Netherlands Organisation for Scientific Research (NWO) 
under project no. 639.032.529.}
\and 
Tobias M\"uller\thanks{Bernoulli Institute, 
Groningen University, The Netherlands. 
E-mail: {\tt tobias.muller@rug.nl}. Supported in part by the Netherlands Organisation for Scientific Research (NWO) 
under project nos 612.001.409 and 639.032.529.}%
}

\date{\today}

\maketitle

\begin{abstract}
We consider percolation on the Voronoi tessellation generated by a homogeneous Poisson point process on the hyperbolic plane.  
We show that the critical probability for the existence of an infinite cluster is asymptotically equal to
$\pi \lambda/3$ as $\lambda\to0.$ 
This answers a question of Benjamini and Schramm~\cite{benjamini2001percolation}. 
\end{abstract}

\section{Introduction and statement of main result}

We will study percolation on the Voronoi tessellation generated by a homogeneous Poisson point process 
on the hyperbolic plane $\Haa^2$. That is, with each point of a constant intensity Poisson process on $\Haa^2$ we associate 
its Voronoi cell -- which is the set of all points of the hyperbolic plane that are closer to it than to any other point 
of the Poisson process -- and we colour each cell black with probability $p$ and white with probability $1-p$, independently 
of the colours of all other cells. 
See Figure~\ref{fig:simulations} for a depiction of computer simulations of this process.

\begin{figure}[h!]
 \begin{center}
  \includegraphics[width=.45\textwidth]{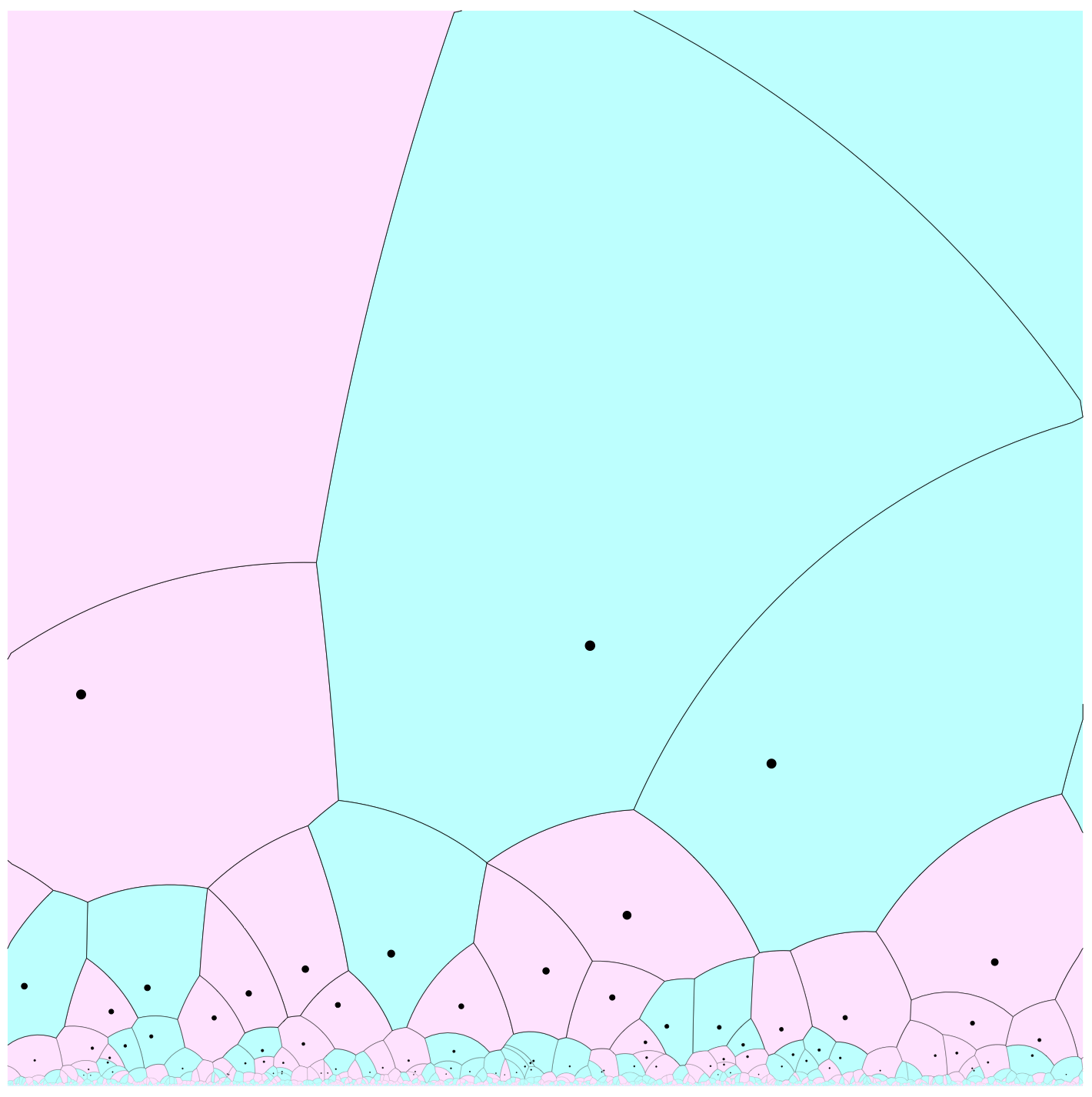}%
  \hspace{3ex}%
  \includegraphics[width=.45\textwidth]{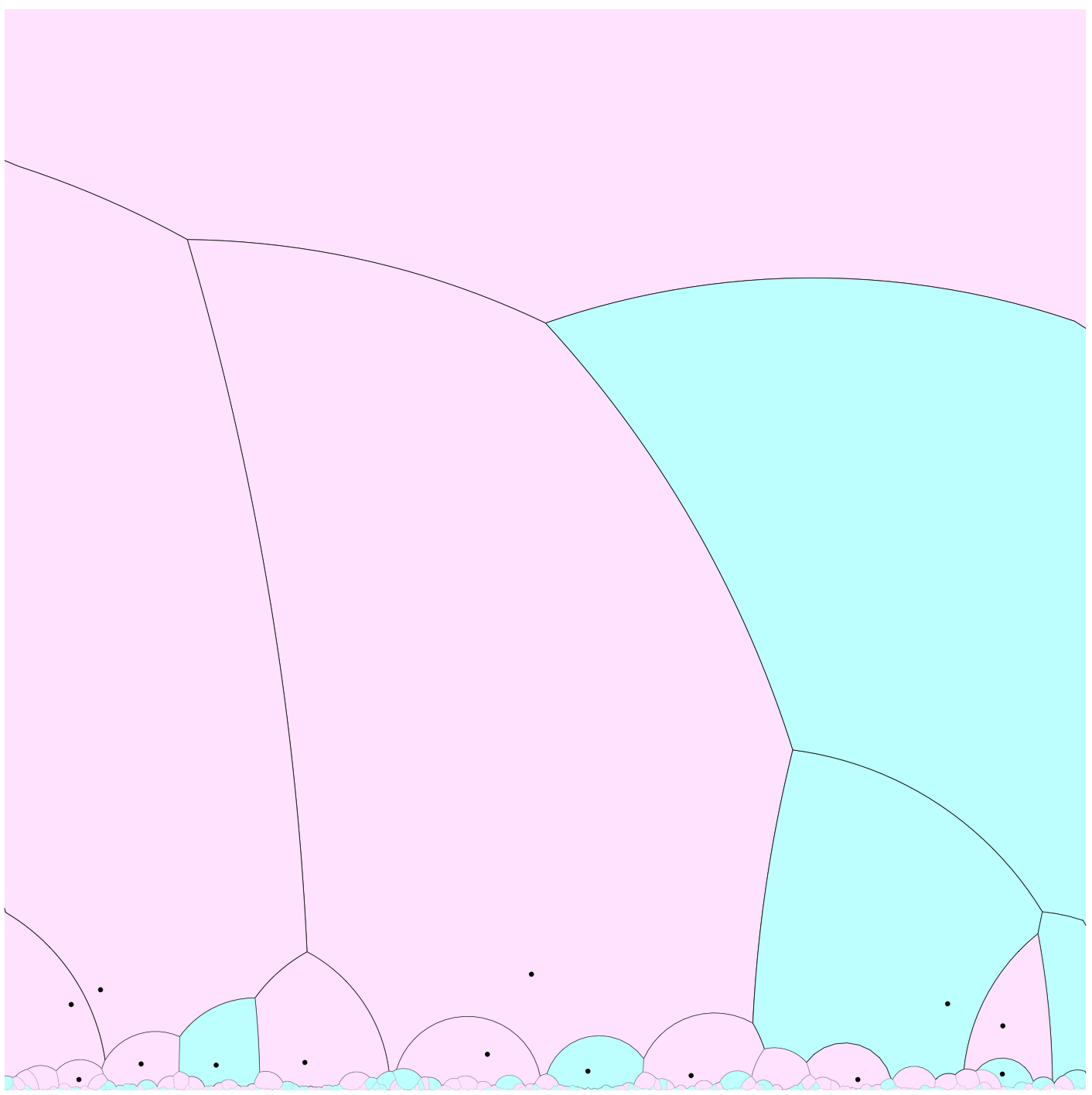}
  
  \vspace{3ex}
  
  \includegraphics[width=.45\textwidth]{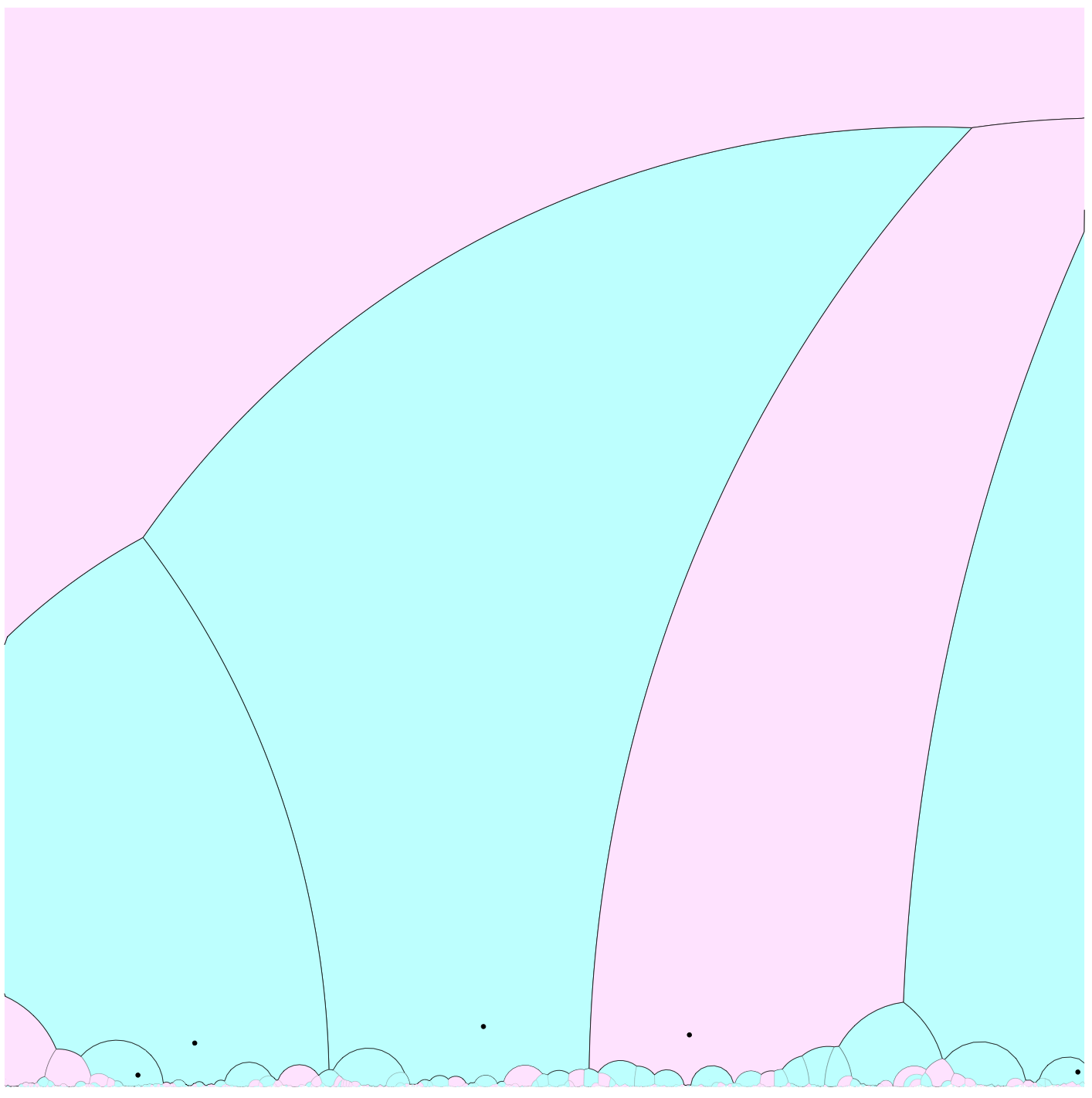}%
  \hspace{3ex}%
  \includegraphics[width=.45\textwidth]{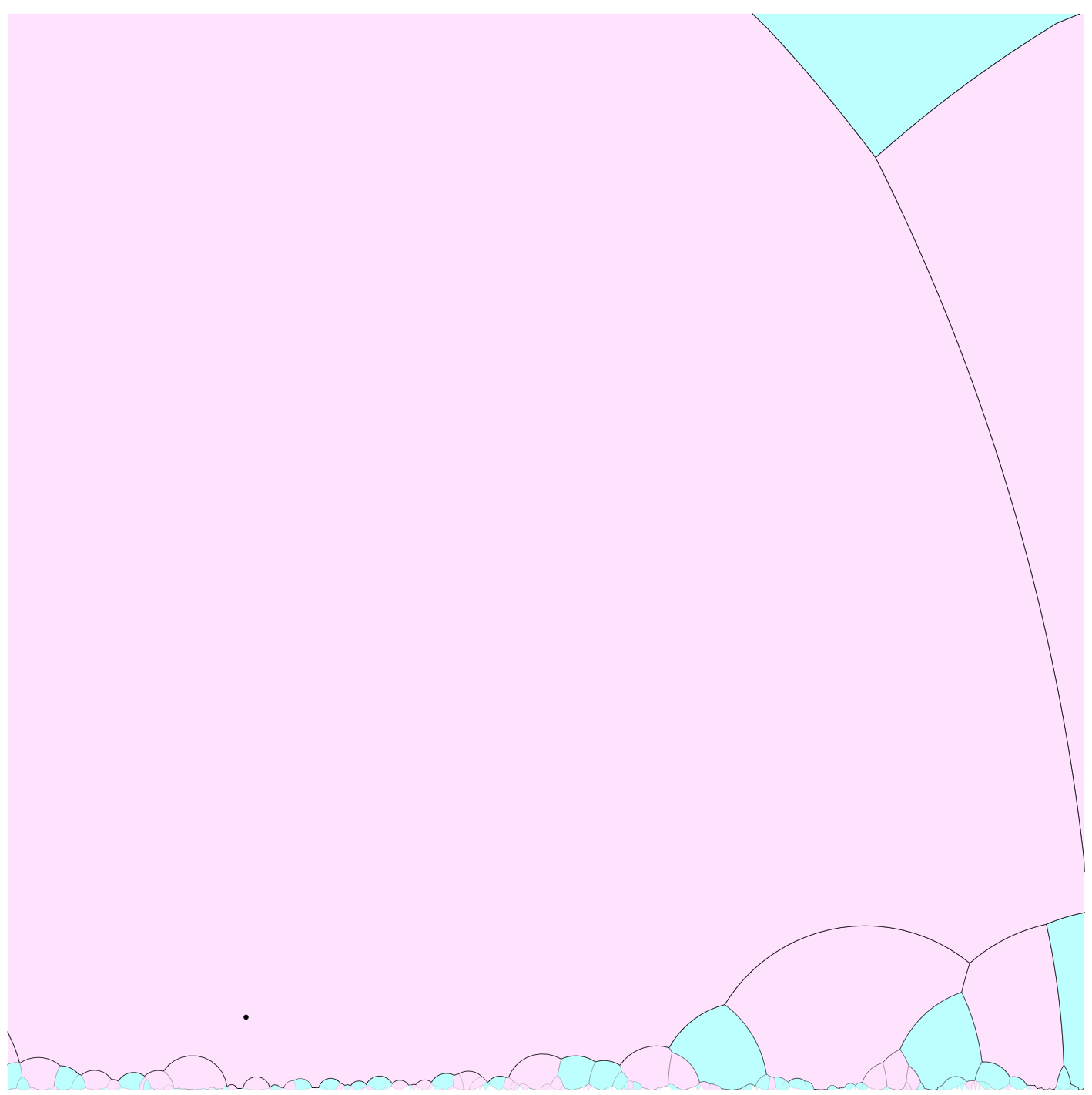}
  
  \caption{Simulations of hyperbolic Poisson-Voronoi percolation, depicted in the half-plane model.
  Top left: $\lambda=1$, top right: $\lambda=\frac{1}{10}$, bottom left: $\lambda=\frac{1}{25}$, 
  bottom right: $\lambda=\frac{1}{50}$; and $p=\frac12$ in all cases.
  \label{fig:simulations}}
 \end{center}
\end{figure}

%
%

We say that \emph{percolation} occurs if there is an infinite connected cluster of black cells. 
For each intensity $\lambda>0$ of the underlying Poisson process, the \emph{critical probability} is defined as 

$$ p_c(\lambda):=\inf\{p:\Pee_{\lambda,p}(\text{percolation})>0\}. $$

To the best of our knowledge, hyperbolic Poisson-Voronoi percolation was first 
studied by Benjamini and Schramm in the influential paper~\cite{benjamini2001percolation}.
Amongst other things, they showed that $0<p_c(\lambda)<1/2$ for all $\lambda>0$; that $p_c(\lambda)$ is 
a continuous function of $\lambda$; and that $p_c(\lambda)\to 0$ as $\lambda\searrow 0$. 

They also asked for the asymptotics of $p_c(\lambda)$ as $\lambda\searrow 0$, and specifically for 
``the derivative at $0$''. Here we answer this question, by showing:

\begin{theorem}\label{thm:main} 
$p_c(\lambda)=(\pi/3)\cdot\lambda + o(\lambda)$ as $\lambda\searrow 0.$ 
\end{theorem}

For comparison, Benjamini and Schramm gave an upper bound $p_c(\lambda) \leq \frac12 - \frac{1}{4\pi\lambda+2}$, whose asymptotics
are $\pi \lambda + o(\lambda)$ as $\lambda\searrow 0$.

The results of Benjamini and Schramm highlight striking differences between Poisson-Voronoi percolation in 
the hyperbolic plane and Poisson-Voronoi percolation in the ordinary, Euclidean plane.
For starters, in the latter case it is known~\cite{Zvavitch,bollobas2006critical} that the critical probability equals
$1/2$ for all values of the intensity parameter $\lambda$.
More strikingly perhaps is the difference in the behaviour of the number of infinite, black clusters.
In the Euclidean case there are no infinite black clusters when $p\leq 1/2$ and 
precisely one infinite black cluster otherwise (almost surely). For the hyperbolic case, Benjamini and Schramm showed that
if $p\leq p_c(\lambda)$ all black clusters are finite; if $p\geq 1-p_c(\lambda)$ then there is precisely one 
infinite black cluster; but if $p_c(\lambda)<p<1-p_c(\lambda)$ then there are infinitely many, 
distinct, infinite, black clusters (almost surely).

\vspace{1ex}

{\bf Related work.} Percolation theory is an active area of modern probability theory with a considerable history, going back 
to the work of Broadbent and Hammersley~\cite{BroadbentHammersley57} in the late fifties. 
By now there is an immense amount of research articles, mostly centered on percolation on lattices. 
We direct the reader to the monographs~\cite{bollobas2006percolation,Grimmettboek} for an overview. 

Poisson-Voronoi tessellations (mostly in $d$-dimensional, Euclidean space) are one of the central models
studied in stochastic geometry. They are studied in connection with many different applications and have 
a long history going back at least to the work of Meijering~\cite{Meijering} in the early fifties. 
For an overview, see the monographs~\cite{SchneiderWeil,StoyanKendallMecke87} and the references therein.

In the early nineties, Vahidi-Asl and Wierman~\cite{VahidiaslWierman90} introduced
first passage percolation (a notion related to, but distinct from percolation as
we have defined it above) on planar, Euclidean Poisson-Voronoi tessellations.
A few years later, Zvavitch~\cite{Zvavitch} proved that in the Euclidean plane almost surely 
all black clusters are finite when $p\leq 1/2$ (and $\lambda>0$ arbitrary). 
About a decade after that Bollob\'as and Riordan were able to complement this result by showing that, 
almost surely, there exists an infinite black cluster when $p>1/2$ (and $\lambda>0$ arbitrary) -- thus establishing 
$p_c=1/2$ for planar, Euclidean Poisson-Voronoi percolation.
Since then planar, Euclidean Poisson-Voronoi percolation, especially ``at criticality'', has received a fair amount of 
additional attention. See e.g.~\cite{AhlbergBaldasso18,AhlbergEtAl16,Tassion16,Vanneuville19}.
Poisson-Voronoi percolation on the projective plane was studied by Freedman~\cite{Freedman97}, 
and Poisson-Voronoi percolation on more general two and three dimensional manifolds 
was studied by Benjamini and Schramm in~\cite{BSconformal}. Poisson-Voronoi percolation on 
higher dimensional Euclidean space has been studied as well, in~\cite{BalisterBollobasQuas05,BalisterBollobas10,Duminil19}.  

Fuchsian groups can be seen as the analogue in the hyperbolic plane of lattices in the Euclidean plane.
Lalley~\cite{Lalley98,Lalley01} studied percolation on Fuchsian groups and amongst other things established 
that the critical probabilities for ``existence'' and ``uniqueness'' of infinite clusters are distinct.
Works on continuum percolation models over Poisson processes in the hyperbolic plane 
include~\cite{visibility1,FlammantArxiv,Thale14,Tykesson07,visibility2}.
Aspects of hyperbolic Poisson-Voronoi tessellations besides percolation that have been studied include 
the (expected) combinatorial structure of their cells, random walks on them and anchored expansions -- see for 
example~\cite{Elliot2, Elliot, CalkaChapronEnriquez, Zakhar, Isokawa3d, Isokawa00}.
Percolation on hyperbolic Poisson-Voronoi tessellations was first studied specifically by Benjamini and Schramm 
in~\cite{benjamini2001percolation}. 
In the recent paper~\cite{HansenMuller1} the current authors showed that $p_c(\lambda)\to 1/2$ as $\lambda\to\infty$
for Poisson-Voronoi percolation on the hyperbolic plane, proving a conjecture from~\cite{benjamini2001percolation}.

Comparing Theorem~\ref{thm:main} with Isokawa's formula (stated as Theorem~\ref{thm:Isokawa} below), we
see that our result can be rephrased as : as $\lambda\searrow 0$, the 
critical probability is asymptotic to the reciprocal of the ``typical degree'' 
(defined precisely in Section~\ref{sec:VorDel} below).
A similar phenomenon happens in several Euclidean percolation models when one sends the dimension to infinity.
The most well-known result in this direction is probably that for both site and bond percolation 
on $\Zed^d$, we have that $p_c = (1+o_d(1)) \cdot (2d)^{-1}$, as was shown concurrently and independently
by Gordon~\cite{Gordon91}, Hara and Slade~\cite{HaraSlade90} and Kesten~\cite{Kesten90}. 
Prior to that there were non-rigorous derivations of the result in the physics literature and 
Cox and Durret~\cite{CoxDurrett83} had proved the analogous result for oriented percolation (which is technically easier to analyze).
Penrose~\cite{Penrose96} proved an analogous result for the Gilbert model on $d$-dimensional Euclidean space
as the dimension tends to infinity, and Meester, Penrose and Sarkar~\cite{MeesterPenroseSarkar97} extended this result
to the random connection model.
The analogous phenomenon was shown by Penrose~\cite{Penrose93} for spread out percolation in fixed dimension when the 
connections get more and more spread out.

\vspace{1ex}

{\bf Sketch of the main ideas used in the proof.} The intuition guiding the proof is that when $\lambda$ is small and $p$ 
is of the same order as $\lambda$ then black clusters are ``locally tree like'', in the sense that while there will be 
some short cycles in the black subgraph of the Delaunay graph, but these will be ``rare''.
(The Delaunay graph is the abstract combinatorial graph whose vertices are 
the Poisson points and whose edges are precisely those pairs of points whose Voronoi cells meet.)
This is also the intuition behind the results on high-dimensional and spread-out percolation mentioned above.
In fact, in several of the works cited it is in fact shown that if we scale $p$ as a constant $\mu$ divided by the degree
(so that the origin has $\mu$ black neighbours in expectation) then the cluster of the origin behaves more and more like a 
Galton-Watson tree with a Poisson$(\mu)$ offspring distribution as the dimension grows.

Before going further, it is instructive to informally discuss in a bit more detail the situation 
for the high-dimensional Gilbert model analyzed by Penrose in~\cite{Penrose96}. 
In that model, we build a random graph on a constant intensity Poisson process on $\eR^d$ by joining any pair 
of Poisson points at distance $<1$ by an edge.
We seek the {\em critical intensity} $\lambda_c$ such that there is a.s.~no percolation for $\lambda<\lambda_c$ and 
there is a positive probability of percolation when $\lambda > \lambda_c$.
We consider the scenario where $d$ is large and the intensity is $\lambda = \mu\cdot \pi_d^{-1}$ with 
$\pi_d$ the volume of the $d$-dimensional unit ball and $\mu>0$ a fixed constant. 
We add the origin $o$ to the Poisson process (note that this a.s.~does not change whether or not there is an infinite
component), and think of  ``exploring'' the cluster of the origin. 
We do this in an iterative fashion : we first add the neighbours of the origin to the cluster, then we consider each of these 
neighbours in turn and add their neighbours to the cluster, and so on.
Of course the neighbours of the origin are precisely those Poisson points that fall inside the $d$-dimensional unit ball $B$.
In particular their number follows a Poisson distribution with mean $\mu$.
Once we have already added some points to the cluster of the origin and we consider the neighbours of a previously added point $u$, we 
add those Poisson points that fall inside the ball of radius one around $u$ from which we have removed the union of all 
radius one balls around points we have processed already. 
Because of the high-dimensional geometry, the volume of this set difference is typically very close to the volume 
of a unit radius ball with nothing removed -- at least during the initial stages of the exploration. 
In other words, at least during the first few exploration steps, the number of new points that gets added in each exploration 
step is approximately distributed like a Poisson random variable with mean $\mu$. This naturally 
leads to the aforementioned connection with Galton-Watson trees with Poisson$(\mu)$ offspring distribution. 
Essentially, the geometric reason why this works out is the ``concentration of measure'' for high dimensional balls 
(see for instance~\cite{Matousekboek}, Chapter 13) : in high-dimensional Euclidean space, the volume of the unit ball is 
concentrated near its boundary. From this one can derive that, in a sense that can be made precise, 
most pairs of points connected by an edge in the Gilbert graph will have distance close to one. 
Moreover, the mass of a $d$-dimensional ball is also concentrated near its equator, from which one can 
derive that -- in a sense that can be made precise -- for most pairs of points with a common neighbour, their 
distance will be very close to $\sqrt{2}$ and, more generally, most pairs with graph distance $k$ will have Euclidean 
distance close to $\sqrt{k}$ (for $k$ fixed). 

A similar phenomenon, but maybe even more extreme in a sense, happens in the hyperbolic plane for disks 
of large radius $r$ : The area of a hyperbolic disk $B$ of large radius $r$ is concentrated near its boundary; and if we take
two points at random from $B$ then typically their distance is close to $2r$, the maximal possible distance 
between any two points in $B$.
In the current paper, we consider a Poisson point process $\Zcal$ on the hyperbolic plane with small intensity $\lambda>0$.
We again include the origin and try to analyze the black cluster of the Voronoi cell of the origin
in the Voronoi tessellation for $\Zcal\cup\{o\}$, where the cell of the origin is coloured black and all other cells are each coloured
black with probability $p$ and white with probability $1-p$.
By Isokawa's formula (stated as Theorem~\ref{thm:Isokawa} below) the expected number 
of cells that are adjacent to the cell of the origin is asymptotic to $\frac{3}{\pi\lambda}$ as $\lambda$ tends to zero.
Moreover, as can be seen either by looking more carefully at Isokawa's computations or by reading some 
of our arguments below, it can be shown that most neighbours of the origin are at distance close to 
$r := 2\ln(1/\lambda)$. Note that $r$ goes to infinity as $\lambda\searrow 0$.
Now suppose we take $p = \mu \cdot (\pi/3) \cdot \lambda$ with $\mu$ a constant and $\lambda$ small, so that 
the expected number of black cells neighbouring the cell of the origin is close to $\mu$.
If we follow an exploration process analogous to the one described above for the Euclidean Gilbert model, the 
black cluster will have the property that most pairs of points at graph distance $k$ have distance close 
to $2kr$ in the hyperbolic metric.

For the upper bound, we will show that if $p = (1+\eps)\cdot(\pi/3)\cdot \lambda$ and $\lambda$ is sufficiently small
(and $\eps>0$ is a fixed constant) then the size of the black cluster of the origin stochastically dominates
the size of a supercritical Galton-Watson tree. We will use an exploration procedure similar to what we described above for the 
high-dimensional Gilbert model. During the exploration, when we consider some point $z$ that has already 
been added to the tree, we make sure to only add those black neighbours of $z$ whose distance to $z$ is within some large constant of 
$r$. We also make sure to select a subset $\{z_1,\dots,z_k\}$ of these neighbours with the property that 
all angles $\angle z_izz_j$ are larger than some small constant, and moreover for each $z_i$ there is a ball
of radius some large (but fixed) constant that contains no other Poisson points. 
Essentially because of the geometric phenomena described earlier, it will turn out that in this version of 
the exploration procedure, for sufficiently small $\lambda$, the 
subgraph of the cluster of the origin we obtain follows the distribution of a certain supercritical Galton-Watson
process {\em exactly}. 
In contrast, in the high-dimensional Gilbert model the correspondence between the exploration process and a 
supercritical Galton-Watson 
process will eventually break down, after a (large) number of exploration steps that depends on the dimension, so that 
additional techniques and arguments were needed by Penrose~\cite{Penrose96}.

For hyperbolic Poisson-Voronoi percolation the lower bound is technically more involved than the upper bound. 
This is probably the most novel contribution in our paper, and we believe it might inspire similar arguments 
applicable to other problems in percolation, notably for high-dimensional models.
For percolation on $\Zed^d$ and the Gilbert model, a trivial (and ``sharp'' up to lower order corrections 
when $d\to\infty$) lower bound
is given by a comparison to branching processes : in the former model, the cluster of the origin is stochastically 
dominated by a Galton-Watson distribution with mean offspring $p\cdot(2d-1)$ and in the latter with mean offspring 
$\mu\cdot \pi_d$. So the critical probability satisfies $p_c \geq 1/(2d-1)$, respectively 
the critical intensity satisfies $\lambda_c \geq 1/\pi_d$.
(For percolation on $\Zed^d$ this was in fact already observed by Broadbent and Hammersley~\cite{BroadbentHammersley57}
and for the Gilbert model by Gilbert~\cite{Gilbert61}.)
In our case 
a similar argument does not seem feasible.
In the Gilbert model, imagine we have explored part of the cluster of the origin and in doing so have revealed 
the status of the Poisson process in some region. We now wish to add those neighbours that are not yet part of the 
cluster of a point that we have previously added to in the cluster. 
The number of new points added is stochastically dominated by the number of new points we add at the very start 
of the exploration, when add the neighbours of the origin. 
In the Poisson-Voronoi percolation model there is no obvious monotonicity of this kind. 
Once we have revealed the status of the Poisson point process in some region this can make 
new edges both more and less likely.
As a side remark, let us mention that using the methods in our paper it ought to be possible to show that 
when $p = (1-\eps)\cdot(\pi/3)\cdot\lambda$ then, as $\lambda\to 0$, the size of black cluster of the origin will 
converge in distribution to the size of a Galton-Watson tree with 
Poisson$(1-\eps)$ offspring distribution. We do not pursue this here however as it does not appear useful
for bounding $p_c(\lambda)$ : it will not exclude the possibility that -- for any fixed, small $\lambda > 0$ -- there is 
an extremely small, but positive probability that the origin is in an infinite component.

A naive approach that one might try is to compute the expected number of black paths of length $k$ starting at the origin
(for $\lambda$ small and $p=(1-\eps)\cdot(\pi/3)\cdot\lambda$), and hope to show that this expectation tends to 
zero as $k\to\infty$. Unfortunately this approach does not seem feasible either. Long paths might revisit the same
area many times which introduces dependencies that are difficult to deal with. 

In order to circumvent this issue, we introduce what we call {\em linked sequences of chunks}.
As we wish to show percolation does not occur, it suffices to show no infinite black cluster exists 
with a more generous notion of adjacency, in the form of what we call {\em pseudo-edges}, that makes the analysis simpler.
With each pseudo-edge we associate a ``certificate'', being the region of the hyperbolic plane that needs to 
be examined to verify it is indeed a pseudo-edge.
We'll say a pseudo-edge on a path $P$ is {\em good} if it has length close to $r$ and does not make a small 
angle with the previous pseudo-edge. Otherwise it is {\em bad}.
A linked sequence of chunks is a sequence of paths $P_1,\dots,P_n$ such that: {\bf 1)} on each path, every pseudo-edge except the 
last is good and the last is bad, and {\bf 2)} for each $i\geq 2$, either the first point of $P_i$ is close 
to the certificate of some pseudo-edge of $P_{i-1}$ or the certificate of the first pseudo-edge of $P_i$ 
intersects one of the certificates of $P_{i-1}$, and for every other pseudo-edge of $P_i$ its certificate
is disjoint from all certificates of all pseudo-edges on $P_1,\dots,P_{i-1}$.
We'll first show (Proposition~\ref{prop:lbprep} below) that if the cluster of the origin is infinite, then 
either there exists an infinite path all of whose edges are good, or there exists an infinite linked 
sequence of chunks starting from the origin. 
Unlike paths in general, it is technically feasible to give decent bounds on the expected number of good paths, respectively 
the number of linked sequences of chunks, of some given length $n$. We are able to obtain bounds that tend to zero 
with $n$, which implies no percolation occurs a.s.

\vspace{1ex}

{\bf Structure of the paper.} In the next section, we collect some notations, definitions, facts and tools that 
we will need in our proofs.
Section~\ref{sec:trees} contains some preparatory geometric observations needed in later sections.
Section~\ref{sec:ub} contains the proof that $p_c(\lambda) \leq (1+\eps)\cdot(\pi/3)\cdot\lambda$ for small 
enough $\lambda$ (and $\eps>0$ fixed), and Section~\ref{sec:lb} contains the proof that 
$p_c(\lambda) \geq (1-\eps)\cdot(\pi/3)\cdot\lambda$ for small enough $\lambda$.
We end the paper by suggesting some directions for further work in Section~\ref{sec:discussion}.

\section{Notation and preliminaries\label{sec:preliminaries}}

\subsection{Ingredients from hyperbolic geometry\label{sec:hyper_prelim}}

The hyperbolic plane $\Haa^2$ is a two dimensional surface with  constant Gaussian curvature $-1$. 
There are many models, i.e.~coordinate charts, for $\Haa^2$ including
the Poincar\'e disk model, the Poincar\'e half-plane model, and the Klein disk model.
A gentle introduction to Gaussian curvature, hyperbolic geometry and these representations 
of the hyperbolic plane can be found in~\cite{stillwell2012geometry}.

Even though we used the Poincar\'e half-plane model for the visualizations in Figure~\ref{fig:simulations}, 
from now on we will exclusively use the Poincar\'{e} disk model.  
(A computer simulation of Poisson-Voronoi percolation depicted in the Poincar\'e disk model can 
be found on the second page of our earlier paper~\cite{HansenMuller1}.)
We briefly recollect its definition 
and some of the main facts we shall be using in our arguments below, and refer the reader to~\cite{stillwell2012geometry}
for proofs and more information. 

The Poincar\'{e} disk model is constructed by equipping the open unit disk $\Dee \subseteq \eR^2$ with an appropriate metric
and area functional.  For points $u,v \in \Dee,$ the hyperbolic distance can be given explicitly by 

$$\dist_{\Haa^2}(u,v)=2\arcsinh\left(\frac{\norm{u-v}}{\sqrt{(1-\norm{u}^2)(1-\norm{v}^2)}}\right)$$ 

\noindent
where $\norm{\cdot}$ denotes the Euclidean norm. 
Straightforward calculations show that in particular, for $z\in \Dee$, the Euclidean and hyperbolic 
distance to the origin are related via: 

\begin{equation}\label{eq:distHtanh} 
\norm{z}=\tanh\left( \distH(o,z)/2 \right). 
\end{equation}

We will use the notations

$$ \ballH(p,r) := \{ u\in\Dee : \distH(p,u)<r \}, \quad \ballR(p,r) := \{ u\in\eR^2 : \norm{p-u}<r\}, $$

\noindent
for hyperbolic, respectively Euclidean, disks.
A standard fact that we will rely on in the sequel is: 

\begin{fact}\label{lem:conformal1} 
Every hyperbolic disk is also a Euclidean disk; and every Euclidean disk contained in the open unit disk $\Dee$  
is also a hyperbolic disk.
\end{fact}

\noindent
(But, the centre and radius of a disk with respect to the hyperbolic metric do not coincide with the
centre and radius with respect to the Euclidean metric.)

For any measurable subset $A \subseteq \Dee,$ its {\em hyperbolic area} is given by

$$\text{area}_{\Haa^2}(A) :=\int_{A} f(z)\dd z,$$

\noindent 
where

\begin{equation}\label{eq:fdef} 
f(u) := \frac{4}{(1-\norm{u}^2)^2}.
\end{equation}

From the formulas for hyperbolic distance and hyperbolic area one can derive that:
\begin{equation}\label{eq:ball_area}
\areaH\left(\ballH(p,r)\right)=2\pi\cdot (\cosh r -1).
\end{equation}
The {\em hyperbolic polar coordinates} $(\alpha,\rho)$ corresponding to a point $z\in\Dee$ 
are $\rho:=\distH(o,z)$ and $\alpha\in[0,2\pi)$ is the (counterclockwise) angle between the positive $x$-axis and the line segment $oz$.
Put differently, $z\in\Dee$ and $\rho\in[0,\infty)$ and $\alpha\in[0,2\pi)$ are such that

\begin{equation}\label{eq:hyppolcoordsub} 
z = \left(\tanh(\rho/2)\cdot\cos\alpha,\tanh(\rho/2)\cdot\sin\alpha\right). 
\end{equation}

In several computations in the paper we'll use the {\em change of variables to hyperbolic polar coordinates}.
By this we of course just mean applying the substitution~\eqref{eq:hyppolcoordsub}.
We will always apply it to integrals of the form $\int_\Dee g(z)f(z)\dd z$ with $f$ as given by~\eqref{eq:fdef} above, 
to obtain:

\begin{equation}\label{eq:hyppolcoord} 
\int_{\Dee} g(z)f(z)\dd z=
\int_{0}^\infty\int_0^{2\pi} g\left(\tanh(\rho/2)\cdot\cos\alpha,\tanh(\rho/2)\cdot\sin\alpha\right)
\sinh\rho \dd\alpha\dd\rho.
\end{equation}

A {\em hyperbolic geodesic} or {\em hyperbolic line segment} between two points $a, b \in \Dee$ is the shortest path between $a$ and $b$
with respect to the hyperbolic metric.
If there is a (Euclidean) line passing through $a,b$ and the origin $o$, then the hyperbolic geodesic between $a$ and $b$ is just 
the (ordinary) line segment between them. 
Otherwise, the hyperbolic geodesic between $a, b$ can be constructed geometrically as follows. 
We let $C \subseteq \eR^2$ be the unique (Euclidean) circle with $a, b \in C$ and such that 
it hits the boundary $\partial \Dee$ of the unit disk at right angles. The points $a,b$ divide $C\setminus\{a,b\}$ into 
two circle segments. The hyperbolic geodesic between $a$ and $b$ is the one contained in $\Dee$.

A {\em hyperbolic line} $\ell\subseteq \Dee$ is either the intersection of an Euclidean line through the origin with the 
open unit disk $\Dee$, or the intersection of $\Dee$ with a circle $C$ hitting $\partial \Dee$ at right angles.

If $a,b,c\in\Dee$ then we use $\angle abc$ to denote the angle the hyperbolic line segment between $a$ and $b$
and the hyperbolic line segment between $b$ and $c$ make in the common point $b$. (In general, when $a,b,c,\in\Dee$ do not lie on 
a hyperbolic line, there will be two possible interpretations of this angle, one in $(0,\pi)$ and one in $(\pi,2\pi)$. 
As is usual, we shall always take the smaller of the two.)
A {\em hyperbolic triangle} is the set of the three hyperbolic line segments defined by three (non-collinear) points 
$a,b,c\in\Dee$.
A critical tool is \emph{hyperbolic law of cosines}. A proof of this result can  
for instance be found in~\cite{thurston1997three} (page 81).

\begin{lemma}\label{lem:law_cosines}
For a hyperbolic triangle, with sides of length $a,b,c$ and respective opposite angles 
$\alpha,\beta,\gamma$ (see Figure~\ref{fig:hyper_triangle}):

$$\cosh(c)=\cosh(a)\cosh(b)-\sinh(a)\sinh(b)\cos(\gamma).$$

\end{lemma}

\begin{figure}[htb!]
\begin{centering}
\includegraphics[width=.4\textwidth]{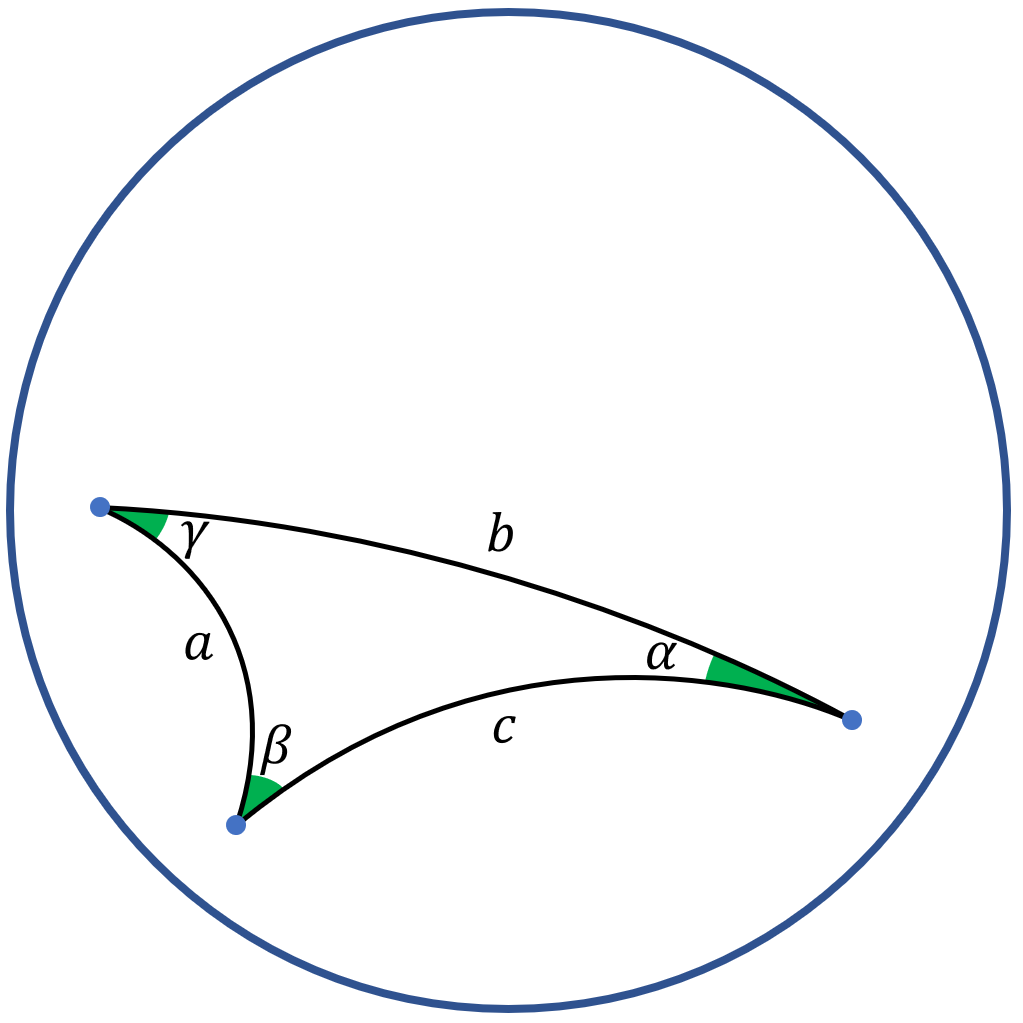}
\caption{A hyperbolic triangle with sides of length $a,b,$ and $c$ and respective opposite angles $\alpha,\beta,$ and $\gamma.$ 
The blue circle is the boundary of the Poincar\'e disk.\label{fig:hyper_triangle}}
\end{centering}
\end{figure}

A hyperbolic ray is defined analogously to a ray in Euclidean geometry. 
That is, if $\ell$ is a hyperbolic line then any $p \in \ell$ divides $\ell \setminus \{p\}$ into 
two connected components. Each of these is a {\em ray emanating from $p$}.
For distinct $p,s\in\Dee$ and $\vartheta\in(0,\pi)$ we let  $\sect{p}{s}{\vartheta}$ denote the 
set of all hyperbolic rays emanating from $p$ that make an angle of no more than $\vartheta$ with the 
ray emanating from $p$ through $s$.
In particular, in the Poincar\'e disk model, if $p = o$ is the origin then $\sect{p}{s}{\vartheta}$ looks 
like a (Euclidean) disk sector 
of opening angle $2\vartheta$ with bisector the ray emanating from $p$ through $s$. See Figure~\ref{fig:sector}.
For any other $p$, the set $\sect{p}{s}{\vartheta}$ can be obtained by applying a suitable hyperbolic 
isometry to such a disk sector.

\begin{figure}[htb]
\begin{center}
\includegraphics[width=.8\textwidth]{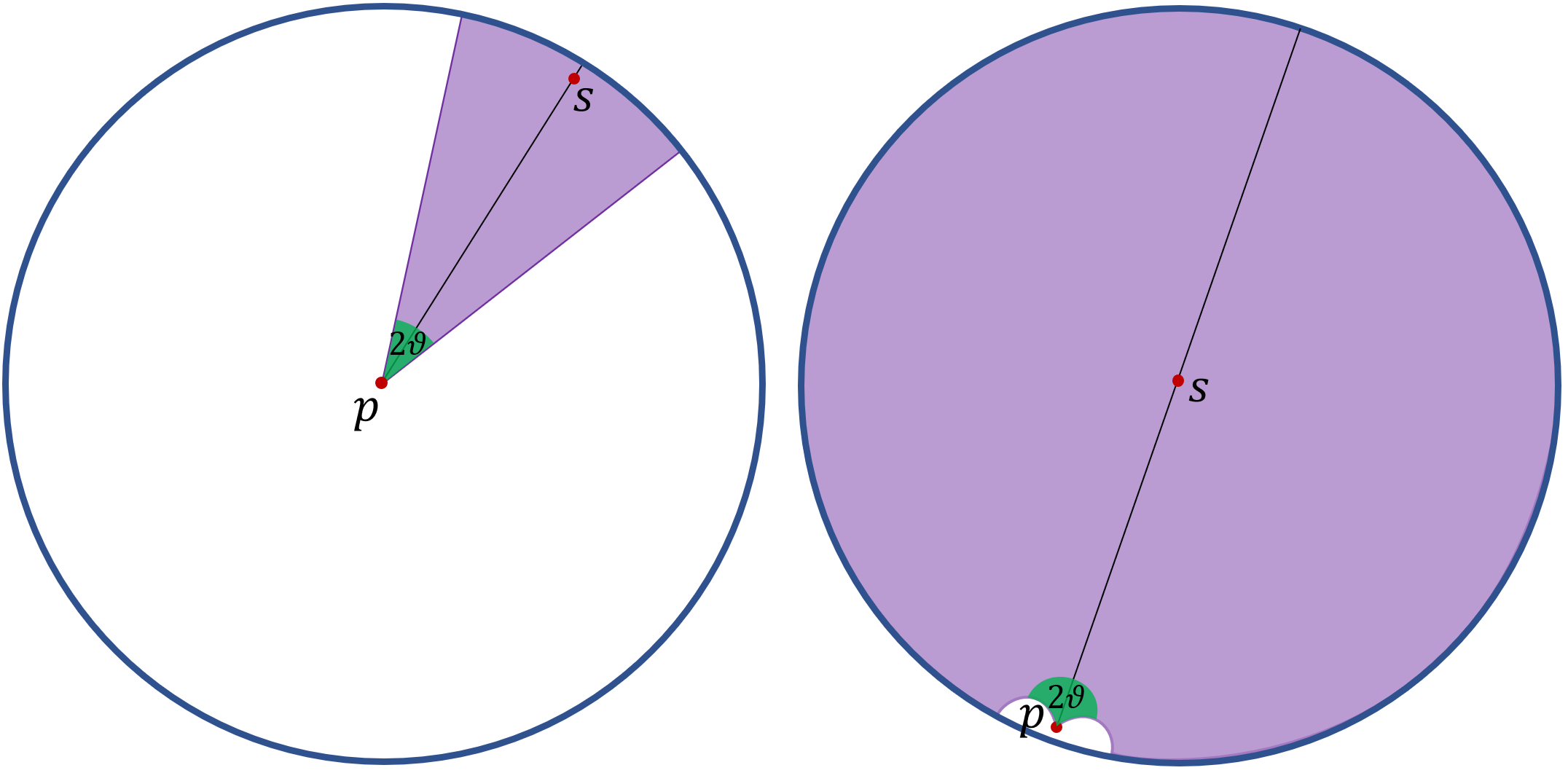}
\caption{The black segment is the ray emanating from $p$ through $s$ and the purple region is the set $\sect{p}{s}{\vartheta}.$ On the left, $p= o$.  
On the right, the sector after applying an isometry that maps $p$ away from the origin and $s$ to the origin.  
The blue circle is the boundary of the Poincar\'e disk.\label{fig:sector}} 
\end{center}
\end{figure}

A {\em hyperbolic isometry} is a bijection $\varphi : \Dee\to\Dee$ that preserves hyperbolic distance, i.e.~$\distH(u,v)
= \distH(\varphi(u),\varphi(v))$ for all $u,v \in \Dee$.
For any two points in $x,y\in \Dee$, there exists a unique hyperbolic line $\ell$ through $x$ and $y$, and
there exists a hyperbolic isometry $\varphi$ such that $\varphi[\ell]$ is an open interval given by the line segment between 
$(-1,0)$ and $(1,0)$, $\varphi(x)=o$ is the origin and $\varphi(y)$ is on the positive x-axis. 
In addition to distance, hyperbolic isometries preserve angles, in the sense that if $c_1,c_2 \subseteq \Dee$ are curves 
(e.g.~hyperbolic line segments, circles, ...) that meet in 
the point $p$ at angle $\alpha$, then the curves $\varphi[c_1],\varphi[c_2]$ meet in the point $\varphi(p)$ at the same
angle $\alpha$.
Hyperbolic isometries also preserve hyperbolic area. That is, 

\begin{equation}\label{eq:areapres} 
\areaH( \varphi[A] ) = \areaH( A ), 
\end{equation}

\noindent
for all (measurable) $A \subseteq \Dee$ and every hyperbolic isometry $\varphi$.
What is more, if $\varphi : \Dee \to \Dee$ is a hyperbolic isometry then, for any (integrable) $g: \Dee \to \eR$ we have

\begin{equation}\label{eq:isosubst} 
\int_{\Dee} g(z) f(z) \dd z = \int_{\Dee} g(\varphi(u)) f(u) \dd u, 
\end{equation}

\noindent
with $f$ as defined by~\eqref{eq:fdef}.
(An easy way to see this is to first note it follows trivially from~\eqref{eq:areapres} when $g$ is the indicator function 
of some measurable $A \subseteq \Dee$. From this it easily follows for step-functions $g = \sum_{i=1}^n a_i 1_{A_i}$. 
Then it follows for an arbitrary measurable function, by approximating it arbitrarily closely by step functions.)

\subsection{Hyperbolic Poisson point processes}

In the rest of this paper $\Zcal$ will  denote a homogeneous Poisson point process (PPP) on the hyperbolic plane. 
Analogously to homogeneous Poisson point processes on the ordinary, Euclidean plane, a homogeneous Poisson process $\Zcal$ of intensity 
$\lambda$ on the hyperbolic plane is characterized completely by the properties that
{\bf a)} for each (measurable) set $A\subseteq \Dee$ the random variable $|\Zcal \cap A|$ is Poisson distributed with mean 
$\lambda \cdot \areaH(A)$, and {\bf b)}
if $A_1, \dots, A_m \subseteq \Dee$ are (measurable and) disjoint then the random variables $|\Zcal\cap A_1|, \dots, |\Zcal\cap A_m|$ are 
independent.
In the light of the formula for $\areaH(.)$ above, we can alternatively view $\Zcal$ as an {\em inhomogeneous} Poisson point 
process on the ordinary, Euclidean plane $\eR^2$ with intensity function

$$ u \mapsto \lambda \cdot 1_{\Dee}(u) \cdot f(u), $$

\noindent
with $f$ given by~\eqref{eq:fdef} above.

Throughout the remainder,  we attach to each point of $\Zcal$ a randomly and independently chosen
colour. (Black with probability $p$ and white with probability $1-p$.)
We let $\Zcalb$ denote the black points and $\Zcalw$ the white points of $\Zcal$.
In the language of for instance~\cite{last2017lectures}, we can view $\Zcal$ as a {\em marked} Poisson point process,
the marks corresponding to the colours.

We will rely heavily on a specific case of the Slivniak-Mecke formula, which is our weapon of choice
for counting tuples of points $z_1,\dots,z_k \in \Zcal$ satisfying a given property.
Before stating it, we remind the reader that formally speaking a Poisson process on $\eR^2$ 
is a random variable that takes values in the space $\Omega_{\text{PPP}}$ of locally finite subsets of $\eR^2$, equipped 
with the sigma algebra 
generated by the family events of the form :  a given Borel $B$ set 
contains precisely $k$ points.
%
%
%

\begin{theorem}[Slivniak-Mecke formula]\label{thm:SlivMeck}
Let $\Zcal$ be a homogeneous hyperbolic Poisson point process of intensity $\lambda$, and let 
Let $g : \Dee^k \times \Omega_{\text{PPP}} \to [0,\infty)$ be a nonnegative, measurable function.
Then

$$ \begin{array}{c} 
\displaystyle
\Ee\left[\sum_{\substack{z_1,...,z_k\in\Zcal\\ \text{distinct}}} g(z_1,...,z_k,\Zcal)\right] \\
= \\
\displaystyle 
\lambda^k \int_\Dee\dots\int_\Dee \Ee\left[g(x_1,...,x_k,\Zcal\cup \{x_1,...,x_k\})\right] f(x_1)\dots f(x_k) \dd x_k\dots\dd x_1,
\end{array} $$

\noindent 
with $f$ given by~\eqref{eq:fdef}.
\end{theorem}

\noindent
As mentioned above, the version we state here is a specific case of a  more general result.
The general version can for instance be found in \cite{SchneiderWeil}, as Corollary 3.2.3. 
(The version we present here is the special case of Corollary 3.2.3 in~\cite{SchneiderWeil} when 
the ambient space $E=\eR^2$ and the intensity measure has density $\lambda \cdot 1_\Dee \cdot f$ with $f$ as in~\eqref{eq:fdef}.)
The Slivniak-Mecke formula is sometimes also called Mecke formula or Campbell-Mecke formula in the literature.

We shall be applying the following consequence of the Slivniak-Mecke formula, that is
tailored to our situation where $\Zcal = \Zcalb \cup \Zcalw$ and membership in $\Zcalb$ is determined via independent, $p$-biased
coin flips.

\begin{corollary}\label{cor:SlivMeck2col}
For $\lambda>0$ and $0\leq p\leq 1$, let $\Zcal = \Zcalb \cup \Zcalw$ be as above and let
$g : \Dee^k \times \Omega_{\text{PPP}} \to [0,\infty)$ be a nonnegative, measurable function.
Then

$$ \begin{array}{c} 
\displaystyle
\Ee\left[\sum_{\substack{z_1,...,z_k\in\Zcalb\\ \text{distinct}}} g(z_1,...,z_k,\Zcal)\right] \\
= \\
\displaystyle 
\left(p\lambda\right)^k \int_{\Dee}\dots\int_{\Dee} \Ee\left[g(u_1,...,u_k,\Zcal\cup \{u_1,...,u_k\})\right] 
f(u_1)\dots f(u_k) \dd{u_1}\dots\dd{u_k},
\end{array} $$

\noindent 
with $f$ given by~\eqref{eq:fdef}.
\end{corollary}

\begin{proof}
Let us write 

$$ S := \sum_{\substack{z_1,...,z_k\in\Zcal\\ \text{distinct}}} g(z_1,...,z_k,\Zcal), \quad
S_{\text{b}} := \sum_{\substack{z_1,...,z_k\in\Zcalb\\ \text{distinct}}} g(z_1,...,z_k,\Zcal). $$

We imagine first ``revealing'' the locations of the Poisson points $\Zcal$, 
but not yet the colours/coin flips.
For any fixed, locally finite $\Ucal \subseteq \eR^d$ we have 

$$ \Ee\left( S | \Zcal = \Ucal \right) = \sum_{{u_1,\dots,u_k \in \Ucal,}\atop\text{distinct}} f(u_1,\dots,u_k, \Ucal), $$

\noindent
while 

$$ \begin{array}{rcl} \Ee\left( S_{\text{b}} | \Zcal = \Ucal \right) 
& = & \displaystyle 
\sum_{u_1,\dots,u_k \in \Ucal,\atop\text{distinct}} f(u_1,\dots,u_k, \Ucal) \cdot 
\Pee( u_1,\dots,u_k \text{ are coloured black} ) \\[4ex]
& = & \displaystyle 
p^k \cdot \left( \sum_{u_1,\dots,u_k \in \Ucal,\atop\text{distinct}} f(u_1,\dots,u_k, \Ucal) \right)
=
p^k \cdot \Ee\left(S|\Zcal=\Ucal\right). 
\end{array} $$

\noindent
This holds for every locally finite $\Ucal$, which implies 

$$\Ee S_{\text{b}} = p^k \cdot \Ee S. $$

\noindent
The result now follows by applying the Slivniak-Mecke formula to $\Ee S$.
\end{proof}

\subsection{Hyperbolic Poisson-Voronoi tessellations\label{sec:VorDel}}

For $\Ucal \subseteq \Dee$ a countable point set and $u \in \Ucal$ we will denote the corresponding 
{\em hyperbolic Voronoi cell} of $u$ by:

$$ \begin{array}{c} 
\displaystyle C(u;\Ucal) := \{ v \in \Dee : \distH(u,v) \leq \distH(u',v) \text{ for all $u'\in\Ucal$}\}.
\end{array} $$

\noindent
We will usually suppress the second argument, and just write $C(u)$ for the 
Voronoi cell of $u$; and $\Ucal$ will usually be either a homogeneous hyperbolic Poisson process $\Zcal$ 
or $\Zcal \cup \{o\}$, such a Poisson process with the origin added in.

The {\em hyperbolic Delaunay graph} for $\Ucal$ is the abstract combinatorial graph with vertex set $\Ucal$ and 
an edge $uu'$ if the Voronoi cells $C(u), C(u')$ meet.
So we can  alternatively view hyperbolic Poisson-Voronoi percolation as site-percolation on the
hyperbolic Poisson-Delaunay graph.
We say that  $u, u' \in \Ucal$ are {\em adjacent} if they share an edge in 
the Poisson-Delaunay graph (in other words, the Voronoi cells $C(u), C(u')$ have at least one point in common).  
In this case we will also say that $u,u'$ are {\em neighbours}.

We point out that if $v \in C(u)\cap C(u')$ then it must hold that $\distH(v,u)=\distH(v,u')$ 
and $\ballH(v,\distH(v,u)) \cap \Ucal = \emptyset$.
In other words $u,u'$ are adjacent if and only if there is 
hyperbolic disk $B$ such that $u,u'\in\partial B$ and $B\cap\Ucal=\emptyset$.
Although we shall not be using this in the present paper, it may be instructive to the reader 
to point out the following. By this last observation, Fact~\ref{lem:conformal1} and 
some relatively straightforward probabilistic considerations : 
if $\Zcal$ is a homogeneous, hyperbolic Poisson process then, almost surely, the 
Voronoi tessellations for $\Zcal$ with respect to the ordinary, Euclidean metric and the Voronoi tessellation 
with respect to the hyperbolic metric as defined above have the same combinatorial structure, in the sense 
that $z,z'\in\Zcal$ are adjacent in the Euclidean tessellation if and only if they are adjacent in the 
hyperbolic tessellation. See Lemma 5 in our previous paper~\cite{HansenMuller1}.

In many of our arguments below we will be considering the Voronoi tessellation for $\Zcal\cup\{o\}$, a
homogeneous, hyperbolic Poisson process with the origin added in.
We refer to the origin as the {\em typical point} and to its Voronoi cell $C(o)$ as the {\em typical cell}.
See Figure~\ref{fig:typdeg} for computer simulations of the typical cell, shown in the
Poincar\'e disk model, for various choices of $\lambda$.
%

\begin{figure}[h!]
 \begin{center}
  \includegraphics[width=.3\textwidth]{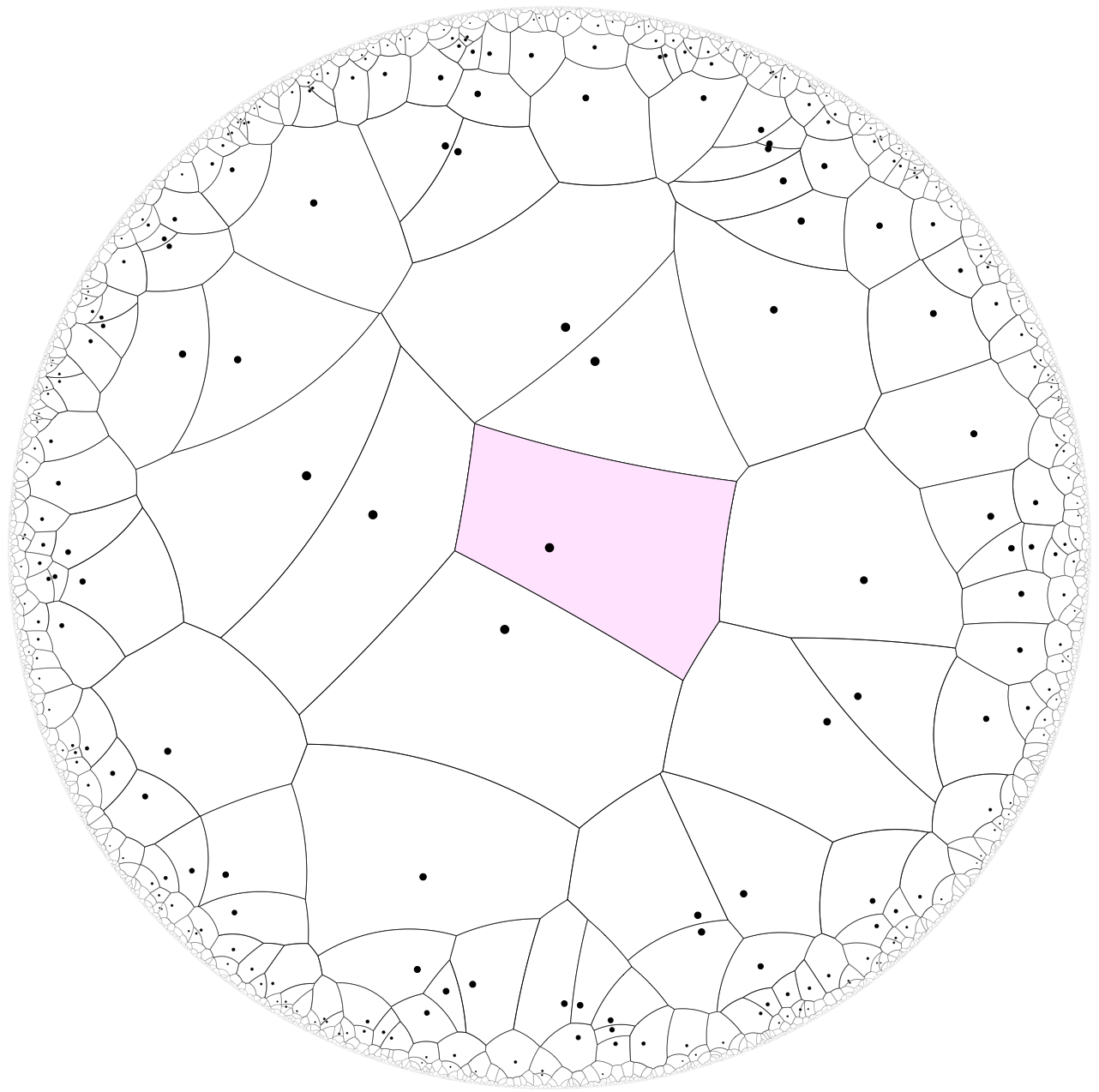}%
  \hspace{3ex}%
  \includegraphics[width=.3\textwidth]{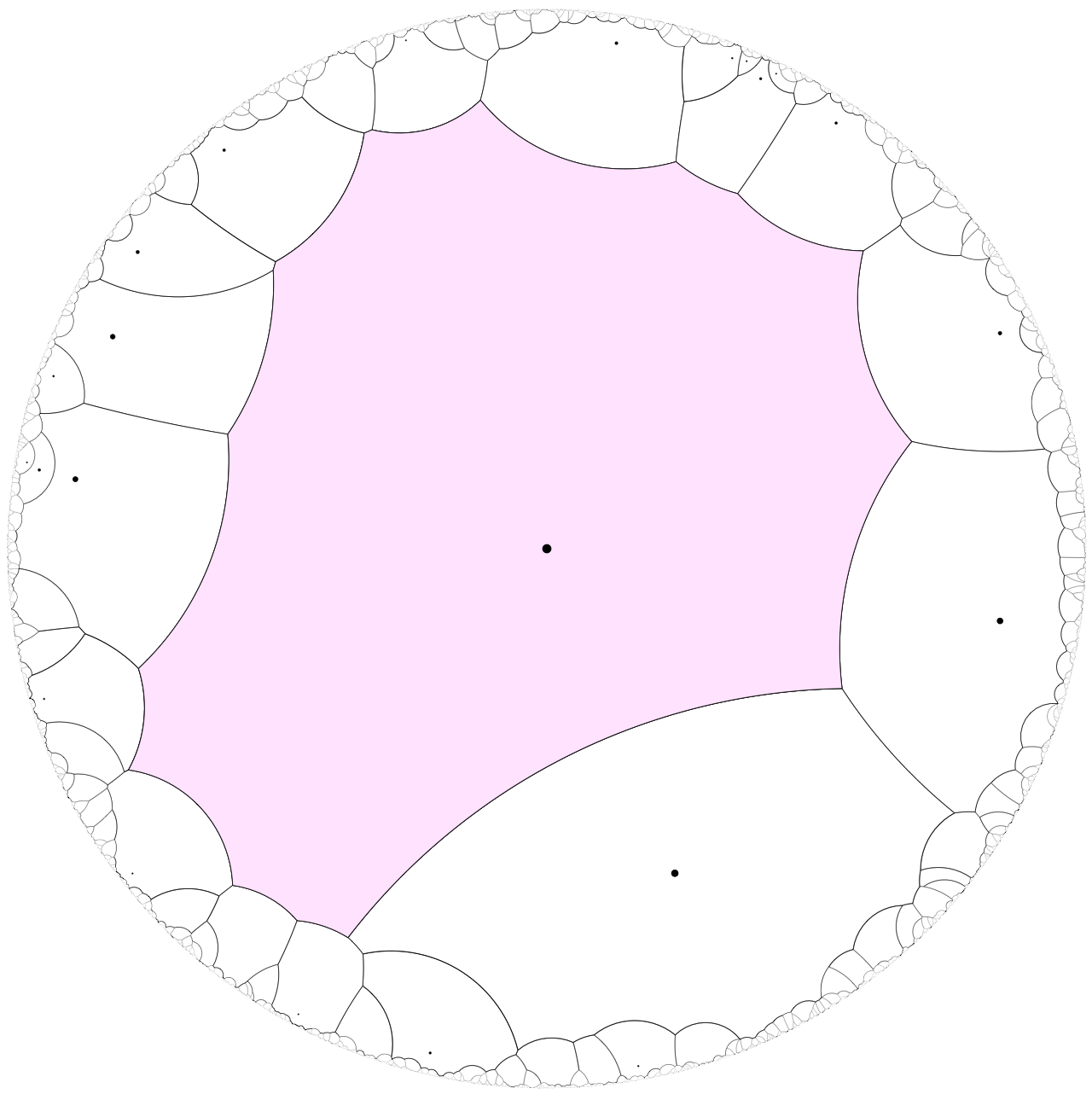}%
  \hspace{3ex}%
  \includegraphics[width=.3\textwidth]{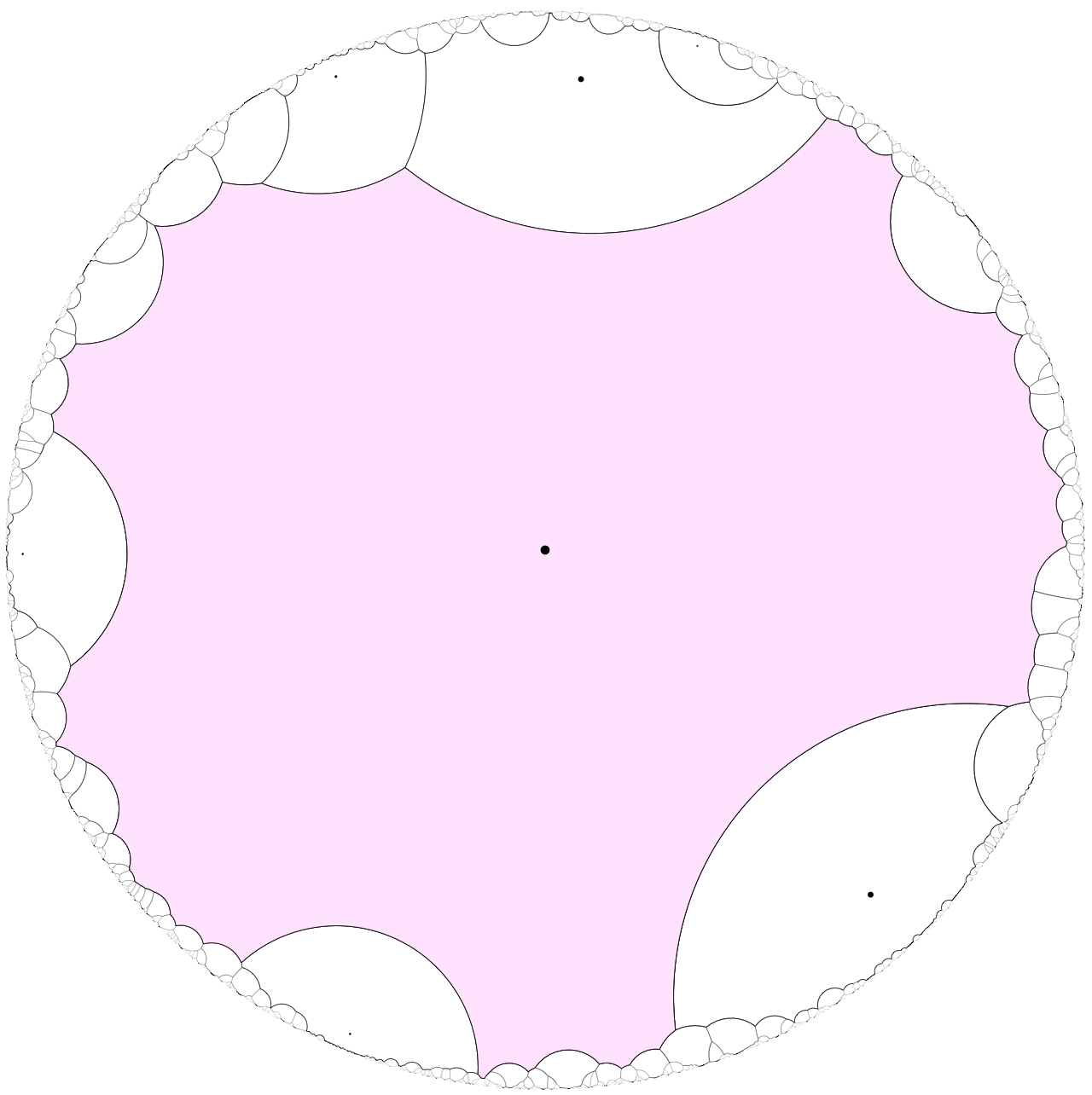}
  
  \caption{Computer simulations of the typical cell (highlighted), shown in the Poincar\'e disk model. Left: $\lambda=1$, 
  middle : $\lambda = \frac{1}{10}$, right: $\lambda=\frac{1}{50}$. \label{fig:typdeg}}
 \end{center}
\end{figure}

We define the {\em typical degree} $D$ as 

\begin{equation}\label{eq:typdegdef} 
D := \left|\left\{ z \in \Zcal : \text{$o,z$ are neighbours in the Voronoi tessellation for $\Zcal\cup\{o\}$}\right\}\right|. 
\end{equation}

The typical degree for homogeneous, hyperbolic Poisson processes has previously been studied by Isokawa~\cite{Isokawa00}. 
She gave the following exact formula for its expectation, valid for all values of the intensity parameter $\lambda>0$.

\begin{theorem}[Isokawa's formula]\label{thm:Isokawa} 
If $D$ is as given by~\eqref{eq:typdegdef} then

$$ \Ee D = 6 + \frac{3}{\pi\lambda}. $$

\end{theorem}

\noindent
What is important for us is the asymptotics $\Ee D \sim \frac{3}{\pi\lambda}$ as $\lambda\searrow 0$.
An independent, alternative derivation of these asymptotics, using conceptually simple arguments that 
are similar to the ones we will use in the present paper, is given in Chapter 3 of the doctoral thesis
of the first author~\cite{Benthesis}.

It may be instructive for the reader (but does not enter into the arguments in the present paper) 
to point out that as $\lambda\to\infty$ the expected typical degree tends to 6, 
the value for the expected typical degree in planar, Euclidean Poisson-Voronoi tessellations~\cite{Meijering}.

It may also be instructive to have another look at the simulations in Figure~\ref{fig:typdeg} (and~\ref{fig:simulations}), 
which indeed suggest as $\lambda$ decreases the typical cell tends to have more adjacencies.

\section{Proofs}

\subsection{Drawing trees in the hyperbolic plane\label{sec:trees}}

We will find it useful to consider graphs embedded in the hyperbolic plane. 
From now on every graph $T$ in the rest of the paper, has vertices $V(T) \subseteq \Dee$ that are points in the Poincar\'e 
disk, and we identify each edge $uv \in E(T)$ with the (hyperbolic) geodesic line segment between $u$ and $v$.

\begin{definition}
For $\rho, w, \vartheta > 0$ we say that $T$ is a {\em $(\rho,w,\vartheta)$-tree} if 
it is connected and
\begin{enumerate}
 \item For every edge $uv \in E(T)$ we have $\rho-w < \distH(u,v) < \rho+w$;
 \item If the edges $uv, uw \in E(T)$ have a common endpoint $u$, then 
 the angle $\angle vuw$ they make at $u$ is at least $\vartheta$.
\end{enumerate}
\end{definition}

\noindent
We will speak of a {\em $(\rho,w,\vartheta)$-path} if the $(\rho,w,\vartheta)$-tree $T$ is a path.

The following proposition verifies that we are justified in using the word ``tree'' :
a $(\rho,w,\vartheta)$-tree is also a tree in the graph theoretical sense, at least when the $\rho$ 
parameter is sufficiently large.

\begin{proposition}\label{prop:treeistree}
For every $w,\vartheta>0$ there exists $\rho_0 = \rho_0(w,\vartheta)$ such that, for all $\rho \geq \rho_0$, every 
$(\rho,w,\vartheta)$-tree is acyclic.
\end{proposition}

We have to postpone the proof until some necessary preparations are out of the way.
If $T$ is a $(\rho,w,\vartheta)$-tree then clearly 
the hyperbolic distance $\distH(u,v)$ between any two vertices $u,v \in V(T)$ is upper bounded 
by $\dist_T(u,v) \cdot (\rho+w)$. 
Here and in the rest of the paper, $\dist_T$ denotes the {\em graph distance}, i.e.~the number 
of edges on the (unique) $u,v$-path in $T$.

The following proposition shows that for $\rho$ sufficiently large this trivial upper bound is close to being tight.

\begin{proposition}\label{prop:distHdistT}
For every $w,\vartheta>0$ there exists $\rho_0 = \rho_0(w,\vartheta)$ and $K = K(w,\vartheta)$ such that,
for all $\rho\geq \rho_0$, every $(\rho,w,\vartheta)$-tree and for every two 
vertices $u,v \in V(T)$:

$$ \distH(u,v) \geq \dist_T(u,v) \cdot (\rho - K). $$

\end{proposition}

\noindent
Again we postpone the proof until we have made some additional observations.

If $G$ is a graph and $uv \in E(G)$ then we denote by $G_{u\setminus v}$ the 
connected component of $G\setminus \{v\}$ that contains $u$.
In other words, $G_{u\setminus v}$ is the subgraph induced on all nodes other than $v$ that can be reached from $u$ using a path that 
avoids $v$. 

The following observation is an important ingredient in the proof of our main result.

\begin{proposition}\label{prop:treedisksector}
For all $w,\vartheta_1,\vartheta_2>0$ there exist $\rho_0 = \rho_0(w,\vartheta_1,\vartheta_2)$ and 
$h = h(w,\vartheta_1,\vartheta_2)$ such that, for all $\rho \geq \rho_0,$ 
every $(\rho,w,\vartheta_1)$-tree $T$ and every edge $uv \in E(T)$:

$$ \bigcup_{x\in V(T_{u\setminus v})} \ballH( x, \rho+w ) \subseteq \ballH(v,h) \cup \sect{v}{u}{\vartheta_2}. $$

\end{proposition}

\noindent
Again we have to postpone the proof until we have made some more preparatory observations.
The next few Lemma's in this section are technical (hyperbolic) geometric considerations that are 
intermediate steps in the proof of the three propositions above. When reading the paper for the 
first time, the reader may wish to only read the lemma statements and skip over the proofs.

We start with a relatively straightforward, but for us useful, consequence of the hyperbolic cosine rule.

\begin{lemma}\label{lem:cosines}
For every $\gamma_0>0,$ there exists $K=K(\gamma_0)$ such that the following holds for every hyperbolic triangle.
Let the lengths of the sides be $a,b,$ and $c$ and the respective opposing angles $\alpha,\beta,\gamma$ 
as in Figure~\ref{fig:hyper_triangle}. 
If $\gamma\geq \gamma_0$ then 

$$ c\geq a+b-K. $$

\end{lemma}

\begin{proof} 
We can assume without loss of generality that $\gamma_0 \leq \pi/2$.
For any hyperbolic triangle satisfying the hypothesis of the lemma, we have

\begin{align*}
e^c&\geq\cosh(c)\\
  &= \cosh(a)\cosh(b)-\sinh(a)\sinh(b)\cos(\gamma)\\
  &\geq \cosh(a)\cosh(b)-\sinh(a)\sinh(b)\cos(\gamma_0)\\
  &\geq\frac{e^{a+b}}{4}(1-\cos(\gamma_0)).
\end{align*}

The second line is the hyperbolic law of cosines (Lemma \ref{lem:law_cosines}).  
The third line follows as $0<\gamma<\pi$ and $\cos(.)$ is decreasing on $[0,\pi)$.
The fourth line uses that $\cosh(a)\cosh(b)\geq \frac{e^{a+b}}{4}$ and $\sinh(a)\sinh(b)\leq \frac{e^{a+b}}{4}$ and 
$\gamma_0 \leq \pi/2$ (by assumption).
Taking logs, we find

$$ c\geq a+b + \ln\left(\frac{1-\cos \gamma_0}{4}\right) =: a+b-K.$$

\end{proof}

Next, we show that if $u,v$ are at least $r+d_0$ apart with $d_0$ a (large) constant and $r$ arbitrary then, from 
the point of view of $v$ (i.e.~if we isometrically map $v$ to the origin of the Poincar\'e disk), a ball of 
large constant radius $r$ around $u$ will be contained in a sector of small opening angle.

\begin{lemma}\label{lem:sec_subset}
For every $\vartheta > 0$ there exists a $d_0=d_0(\vartheta)>0$ such that, for all $u,v \in \Haa^2$ and $r>0$,
if $\distH(u,v) > r+d_0$ then 

$$ \ballH(u,r)\subseteq \sect{v}{u}{\vartheta}.$$

(See Figure~\ref{fig:sector_cover}.)
\end{lemma}

\begin{figure}[htb!]
\begin{center}
\includegraphics[width=.45\textwidth]{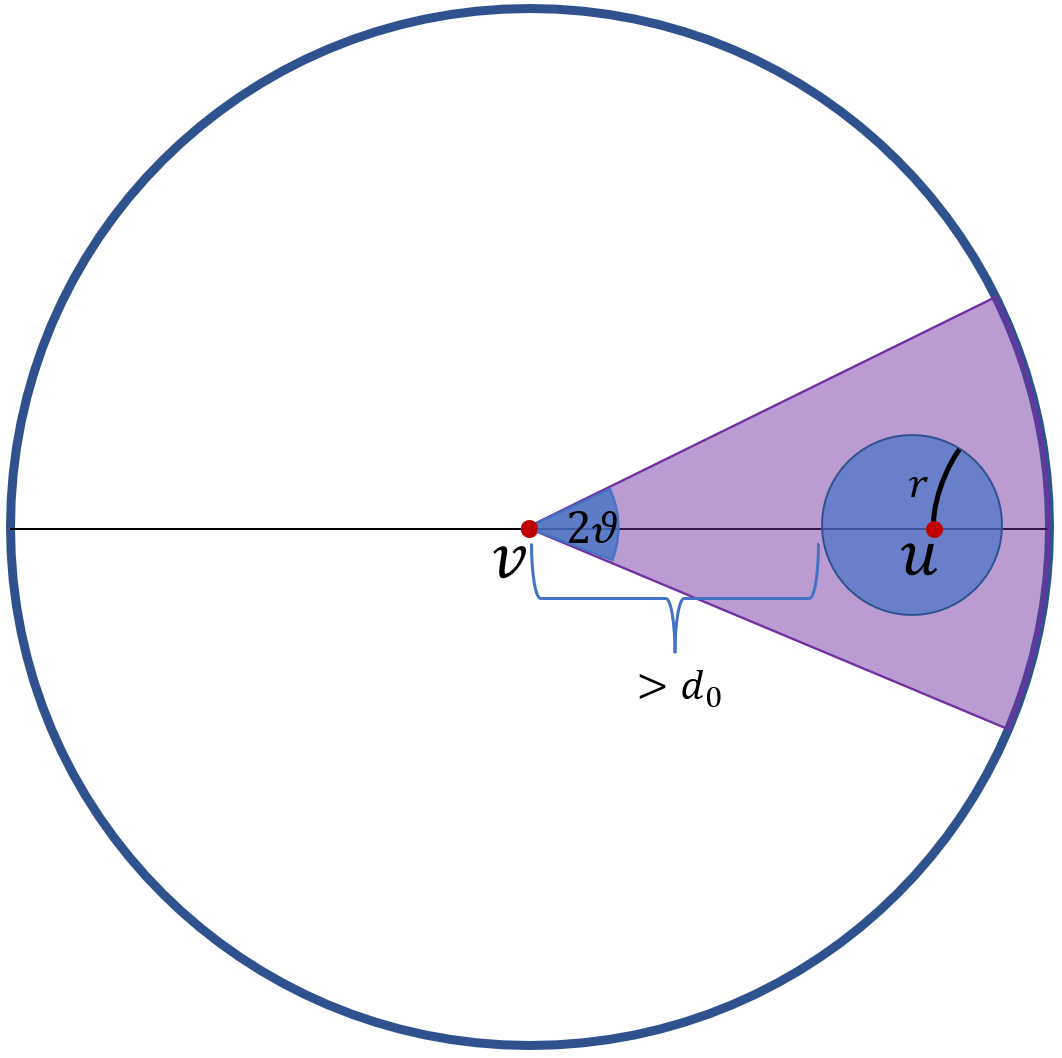}
\caption{An example of $\ballH(u,r)\subseteq \sect{v}{u}{\vartheta}$.\label{fig:sector_cover}}
\end{center}
\end{figure}

\begin{proof}
We let $d_0$ be a large constant, to be determined in the course of the proof.
Applying a suitable isometry if needed, we can assume without loss of generality that $v=o$ is the origin and 
$u$ lies on the positive $x$-axis.
We recall that $B := \ballH(u,r)$ is also a Euclidean disk. Its Euclidean center must lie on the $x$-axis. 
(This is for instance easily seen by noting that the reflection in the $x$-axis 
is a hyperbolic isometry that fixes $u$, and hence leaves $B$ invariant.)

By the triangle inequality and the assumption that $\distH(u,v) > r+d_0$, we have

$$ \distH(v,z) > d_0 \text{ for all $z\in B$. } $$

By~\eqref{eq:distHtanh}, for each $z\in B$ we have 

$$ 1 > \norm{z} > \tanh(d_0/2) \geq 1 - \delta, $$

\noindent
where $\delta$ is a small constant to be determined shortly, and the final inequality 
holds provided we have chosen $d_0$ sufficiently large.
Since $B$ is contained in the annulus  $\Dee \setminus \ballR(o,1-\delta)$, its Euclidean radius is no more than $\delta/2$.

As its Euclidean centre lies on the $x$-axis, it is clear that $\delta=\delta(\vartheta)$ 
can be chosen so that $D \subseteq \sect{v}{u}{\vartheta}$.
\end{proof}

Our next observation is another intermediate step in the proofs of Propositions~\ref{prop:treeistree},~\ref{prop:distHdistT}
and~\ref{prop:treedisksector}. It will be used only in the proof of Lemma~\ref{lem:sector_contain} that immediately succeeds 
it in the text.

\begin{lemma}\label{lem:euclprep1}
Suppose the (Euclidean) circle $C$ intersects the unit circle $\partial\Dee$ at right angles, and 
intersects the $x$-axis at an angle $\vartheta>0$.
If $r$ denotes the (Euclidean) radius of $C$ and $u$ denotes the intersection point of $C$ and the $x$-axis that falls inside $\Dee$, then

$$ r = \frac{1-\norm{u}^2}{2\norm{u}\sin\vartheta}. $$

\end{lemma}

\begin{proof}
Let $c$ denote the centre of $C$. Since $C$ hits the unit circle at a right angle, we have $\norm{c} = \sqrt{1+r^2}$.
In the Euclidean triangle with corners $o,u,c$, the angle at $u$ equals $\pi/2+\vartheta$. See Figure~\ref{fig:euclprep1}.

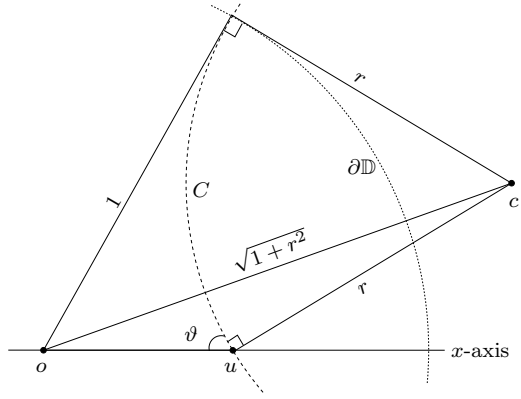
\begin{figure}[h!]
\begin{center}
 \input{euclprep1.pspdftex}
\end{center}
\caption{Illustration of the proof of Lemma~\ref{lem:euclprep1}.\label{fig:euclprep1}}
\end{figure}

Applying the Euclidean cosine rule to the triangle with corners $o,u,c$ we find

$$ 1+r^2 = \norm{u}^2 + r^2 - 2 \norm{u} r \cos( \pi/2 + \vartheta )
= \norm{u^2} + r^2 + 2 \norm{u} r \sin( \vartheta ). $$

The claimed expression follows by reorganizing this last identity.
\end{proof}

The next lemma makes the following perhaps rather counterintuitive observation.
While $\Dee \setminus \sect{u}{v}{\vartheta}$ 
looks ``large'' from the vantage point of $u$ (i.e.~if we isometrically map $u$ to the origin of the
Poincar\'e disk, the sector looks only like a small ``slice''), provided $u,v$ are sufficiently far apart, 
from the vantage point of $v$ the set $\Dee \setminus \sect{u}{v}{\vartheta}\subseteq \sect{v}{u}{\vartheta}$ 
is actually contained in a small ``slice''. 
 
\begin{lemma}\label{lem:sector_contain}
For all $\vartheta>0$ there exists $d_0 = d_0(\vartheta)>0$ such that, for all $u,v\in\Dee$, 
if $\distH(u,v)\geq d_0$, then $\Dee\backslash\sect{u}{v}{\vartheta}\subseteq \sect{v}{u}{\vartheta}$.

(See Figure \ref{fig:hyper_containment}.)
\end{lemma}

\begin{figure}[htb]
\begin{center}
\includegraphics[width=.4\textwidth]{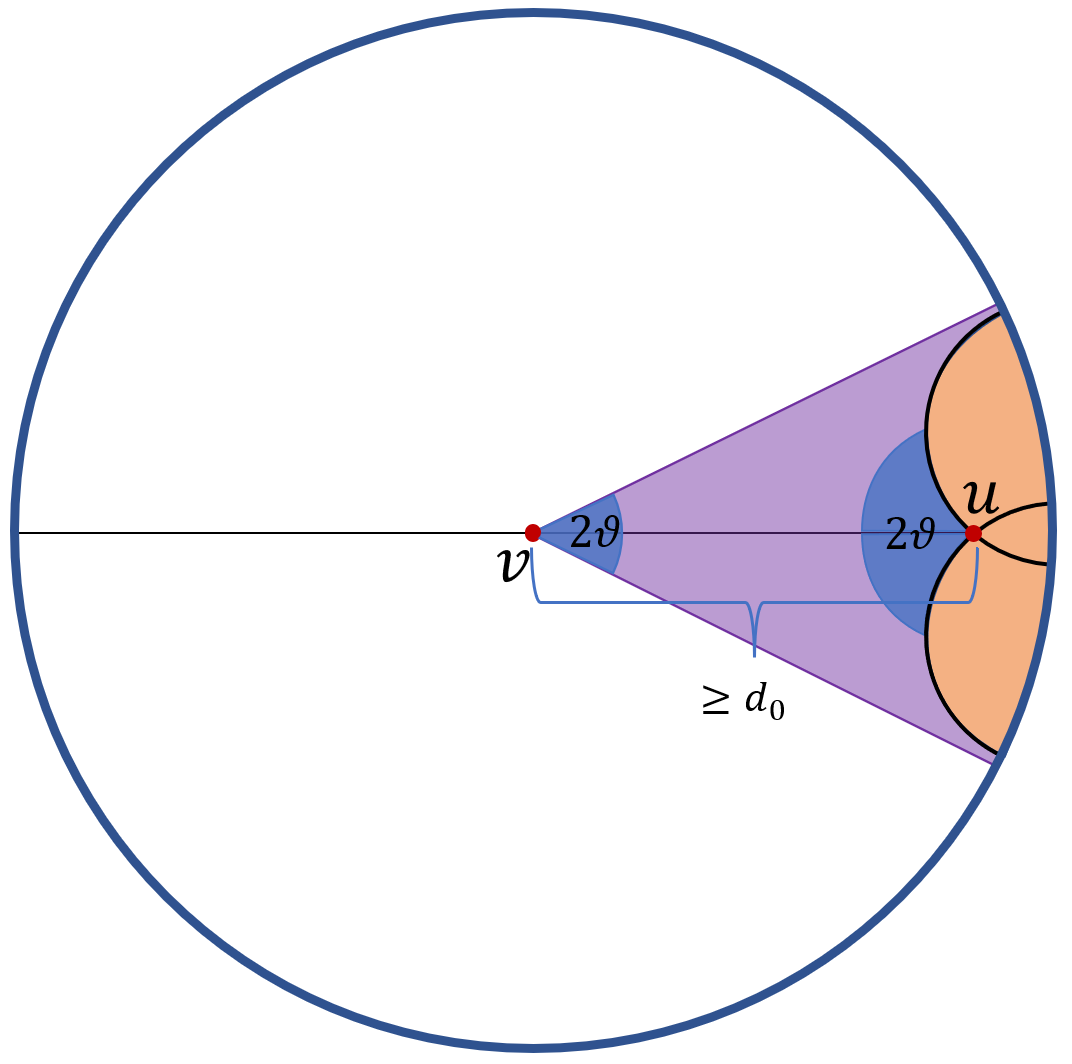}
\caption{The orange region is  $\Dee\backslash\sect{u}{v}{\vartheta}$.  
The union of the purple and the orange region is $\sect{v}{u}{\vartheta}$.\label{fig:hyper_containment}}
\end{center}
\end{figure}

\begin{proof}
Applying a suitable isometry if needed, we can assume without loss of generality that $v=o$ is the origin and $u$ lies on the
positive $x$-axis.

The boundary of $\sect{u}{v}{\vartheta}$ consists of two rays emanating from $u$ that each make an angle $\vartheta$ with the 
hyperbolic line through $u$ and $v$ at $u$ (and a circle segment that is part of the unit circle $\partial\Dee$).
In Euclidean terms, these two rays are circle segments of circles that each meet the $x$-axis in $u$ at 
an angle of $\vartheta$, and intersect the boundary of the unit circle $\partial \Dee$ at right angles.
(See Figure~\ref{fig:hyper_containment}.)
Let us call these circles $C_1, C_2$.
Both circles meet the $x$-axis in $u$ at angle $\vartheta$. 
Hence, by Lemma~\ref{lem:euclprep1}, both have Euclidean radius $(1-\norm{u}^2)/(2\norm{u}\sin\vartheta)$.

We have $\norm{u} \geq \tanh(d_0/2)$ which approaches one as $d_0$ approaches infinity.
In particular, for any constant $\delta>0$, we can choose the constant $d_0$ so that both $C_1,C_2$ have Euclidean radius 
$< \delta$, and both intersect the $x$-axis at a point within distance $\delta$ of $(1,0)$.
In other words, we'll have 

$$ \Dee \setminus \sect{u}{v}{\vartheta} \subseteq \ballR( (1,0), 3\delta ) \cap \Dee. $$ 

Clearly, if we take $\delta=\delta(\vartheta)$ small enough, it now follows  that
$\Dee \setminus \sect{u}{v}{\vartheta} \subseteq \sect{u}{v}{\vartheta}$, as desired.
\end{proof}

As our final preparation for the proofs of Propositions~\ref{prop:treeistree} and~\ref{prop:distHdistT},
we show that if $T$ is a $(\rho,w,\vartheta)$-tree and $uv \in E(T)$ an edge then 
the vertices of $T_{u\setminus v}$ are all contained in a small sector from the vantage point of 
$v$, provided $\rho$ is chosen sufficiently large.

\begin{lemma}\label{lem:VTuvsect}
For every $w,\vartheta_1,\vartheta_2>0$ there exists a $\rho_0 = \rho_0(w,\vartheta_1,\vartheta_2)$ such that,
for every $\rho\geq \rho_0$, every $(\rho,w,\vartheta_1)$-tree $T$ and every edge $uv \in E(T)$:

$$ V(T_{u\setminus v}) \subseteq \sect{v}{u}{\vartheta_2}. $$

\end{lemma}

\begin{proof}
We can and do assume, without loss of generality, that $\vartheta_1 < \frac{1}{1000}$ and $\vartheta_2 < \vartheta_1/2$;
and we set $\rho_0 := w + d_0$ with $d_0 = d_0(\vartheta_2)$ as provided by Lemma~\ref{lem:sector_contain}.

We will use induction on the number of edges $m := |E(T)|$ of $T$.
The statement is clearly true when $m \leq 1$.
Let us thus assume that, for some $m \in \eN$, the statement holds for all $(\rho,w,\vartheta_1)$-trees with 
$< m$ edges, let $T$ be an arbitrary $(\rho,w,\vartheta_1)$-tree with $m$ edges, and pick an arbitrary edge $uv \in E(T)$. 
If $u$ is a leaf then $T_{u\setminus v} = \{u\}$ and the statement is trivial. 
So we can and do assume $u$ has at least two neighbours.
Let us denote the neighbours of $u$ as $N(u) = \{v, x_1,\dots,x_k\}$ and let
$T_{i}$ denote the connected component of $T \setminus ux_i$ containing $x_i$.
By the induction hypothesis $V(T_{i}) \subseteq \sect{u}{x_i}{\vartheta_2}$ for $i=1,\dots,k$.
Since the line segments $uv$ and $ux_i$ make an angle $\geq \vartheta_1$, it follows that 
$\sect{u}{v}{\vartheta_2} \cap \sect{u}{x_i}{\vartheta_2} = \emptyset$.
In other words 

$$ V(T_{u\setminus v}) \subseteq \{u\} \cup \bigcup_{i=1}^k V(T_{i}) \subseteq \Dee \setminus \sect{u}{v}{\vartheta_2}
\subseteq \sect{v}{u}{\vartheta_2}, $$

\noindent
where the last inclusion follows by Lemma~\ref{lem:sector_contain} and our choice of $\rho_0$.
(Using that $\distH(v,u) \geq \rho_0-w \geq d_0$ by choice of $\rho_0$.)
\end{proof}

\begin{proofof}{Proposition~\ref{prop:treeistree}}
We let $\rho_0$ be a sufficiently large constant, to be determined during the course of the proof.
Let $T$ be an arbitrary $(\rho,w,\vartheta)$-tree for some $\rho \geq \rho_0$, and suppose it 
contains a cycle $C$.  Observe that as $C$ is a connected subgraph of $T$, it is also a $(r,w,\vartheta)$-tree.
Let $v_1,v_2,v_3$ be three vertices that are consecutive on $C$. 
Clearly $C_{v_1\setminus v_2} = C \setminus v_2$ is just $C$ with $v_2$ and both incident edges removed.
We let $\vartheta'>0$ be a sufficiently small constant, to be determined more precisely shortly.
By the previous lemma,

\begin{equation}\label{eq:VTv21} 
V( C ) \setminus \{v_2\} \subseteq \sect{v_2}{v_1}{\vartheta'}, 
\end{equation}

\noindent
assuming we chose $\rho_0$ sufficiently large. 
By symmetry (considering $C_{v_3\setminus v_2}$) we also have 

\begin{equation}\label{eq:VTv22}
V(C) \setminus \{v_2\} \subseteq \sect{v_2}{v_3}{\vartheta'}.  
\end{equation}

Since $\angle v_1v_2v_3 > \vartheta$, provided we chose $\vartheta'$ sufficiently small, we have that 

$$ \sect{v_2}{v_1}{\vartheta'} \cap \sect{v_2}{v_3}{\vartheta'} = \emptyset. $$

Combining this with~\eqref{eq:VTv21} and~\eqref{eq:VTv22}, would imply that $V(C) = \{v_2\}$, contradicting our 
assumption that $C$ is a cycle.
\end{proofof}

\begin{proofof}{Proposition~\ref{prop:distHdistT}}
We let $\rho_0, K$ be sufficiently large constants, to be specified more precisely during the course of the proof.
We will use induction on $n := \dist_T(u,v)$.

The base case, when $n  = 1$ is trivial by definition of $(\rho,w,\vartheta)$-tree -- provided 
we chose $K \geq w$.

Let us then assume the statement is true for $n-1$ and let $u=v_0, \dots, v_n=v$ be a $(u,v)$-path in $T$.
By Proposition~\ref{prop:treeistree}, having chosen $\rho_0$ sufficiently large, we know there is 
precisely one path between any pair of vertices.
By the induction hypothesis

$$ \distH(v_1,v_n) \geq (n-1)\cdot(\rho-K), $$

\noindent
and by Lemma~\ref{lem:VTuvsect}, assuming we chose $\rho_0$ sufficiently large, we have that 

$$ v \in V(T_{v_2\setminus v_1}) \subseteq \sect{v_1}{v_2}{\vartheta/2}. $$

By definition of $(\rho,w,\vartheta)$-tree we have $\angle uv_1v_2 > \vartheta$.
It follows that $\angle uv_1v > \vartheta/2$.
Applying Lemma~\ref{lem:cosines}, we find

$$ \distH(u,v) \geq \distH(u,v_1) + \distH(v_1,v) - K', $$

\noindent
for some constant $K' = K'(\vartheta/2)$ provided by Lemma~\ref{lem:cosines}.
Hence,

$$ \distH(u,v) \geq (\rho-w) + (n-1)\cdot(\rho-K) - K' \geq n \cdot (\rho-C), $$

assuming (without loss of generality) we have chosen $K > w+K'$.
\end{proofof}

The final ingredient we need for the proof of Proposition~\ref{prop:treedisksector} is the following observation.
It states that if the distance between $u$ and $v$ is between $r-w$ and $r+w$, where $w$ is constant and $r$ arbitrary, then a ball of 
radius $r+w$ around $u$ is contained in the union of a ball of constant radius $h$ around $v$ and a sector of 
constant opening angle $\vartheta$

\begin{lemma}\label{lem:bounding}
For every $w,\vartheta>0$ there exists $h = h(w,\vartheta)$ such that
for all $r>0$ and all $u,v \in \Dee$ with $r-w < \distH(u,v) < r+w$ we have

$$ \ballH(u,r+w ) \subseteq \ballH(v,h) \cup \sect{v}{u}{\vartheta}. $$

\end{lemma}

\begin{proof}
Applying a suitable isometry if needed, we can assume without loss of generality $v=o$ is the origin and that 
$u$ lies on the positive $x$-axis.

\begin{figure}[htb!]
\begin{center}
\includegraphics[width=1.01\textwidth]{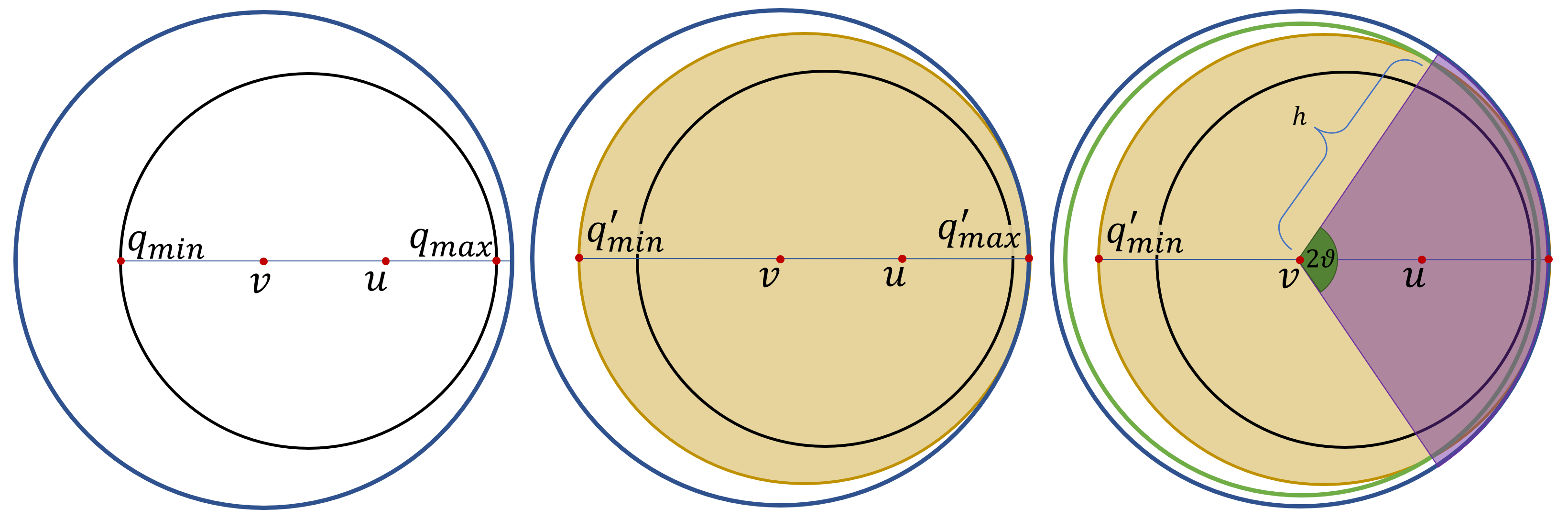}
\caption{On the right, the disk with the black boundary is $\ballH(u,r+w)$. In the centre, the gold disk is $D$.  
On the left, the purple sector is $\sect{v}{u}{\vartheta}$ and the disk with the green boundary is $\ballH(v,h)$.  
\label{fig:disk_and_sector}}
\end{center}
\end{figure}

We recall that the hyperbolic disk $\ballH(u,r+w)$ is also a Euclidean disk $\ballR(z,t)$.
The Euclidean center $z$ must lie on the $x$-axis (an easy way to see this is that reflection in the $x$-axis
is a $\Haa^2$-isometry that leaves $\ballH(u,r+w)$ invariant).
Let $q_{\min} = (x_{\min},0)$ and $q_{\max} = (x_{\max},0)$ be the points where the circle $\partial \ballH(u,r+w)$ 
intersects the $x$-axis. See Figure~\ref{fig:disk_and_sector}, left.
(So the Euclidean center of $\ballH(u,r+w)$ is the midpoint $(q_{\min}+q_{\max})/2$ between these two points.)
Since $v=(0,0) \in \ballH(u,r+w)$ we have $x_{\min} \leq 0 \leq x_{\max}$.
As the points $q_{\min}, v=o, u$ lie on the $x$-axis, which is a hyperbolic line, we have

$$ \distH(q_{\min},u) = \distH(q_{\min},v) + \distH(u,v), $$

\noindent
giving 

$$ \distH(q_{\min}, o ) = \distH(q_{\min},u) - \distH(u,v) \leq r+w - (r-w) = 2w. $$

By~\eqref{eq:distHtanh} it follows that $\norm{q_{\min}} \leq \tanh(w) < 1$. 
In other words, $-\tanh(w) \leq x_{\min} \leq 0$.

We set  

$$ q_{\min}' := (-\tanh(w),0), \quad q_{\max}' := (1,0), $$ 

\noindent
and let 

$$D := \ballR\left( \left(\frac{1-\tanh(w)}{2},0\right), \frac{1+\tanh(w)}{2} \right), $$

\noindent 
be the Euclidean disk with center $(q_{\min}'+q_{\max}')/2 = \left((1-\tanh(w))/2,0\right)$ and radius $(1+\tanh(w))/2$.
Put differently, $D$ is the disk with $q_{\min}', q_{\max}'$ on its boundary, that meets the $x$-axis at a 
right angle at both points. 
This is not a hyperbolic disk (it is what is called a horocycle), but we do have that $\ballH(u,r+w) \subseteq D \subseteq \Dee$.
See Figure \ref{fig:disk_and_sector}, middle. 

Let $R := D \setminus \sect{v}{u}{\vartheta}$ be the part of $D$ that is not contained in the sector $\sect{v}{u}{\vartheta}$.
Then 

$$ h := \sup\{ \distH(z,o) : z \in R \} < \infty, $$

\noindent
since all points of $R$ are at least some positive Euclidean distance away 
from the boundary of the unit disk. See Figure \ref{fig:disk_and_sector}, right.
Evidently we have

$$ \ballR(u,r+w) \subseteq D \subseteq R \cup \sect{v}{u}{\vartheta} \subseteq \ballH(v,h) \cup \sect{v}{u}{\vartheta}, $$

\noindent
as desired. (Note that $h$ depends only on $w$ and $\vartheta$.)
\end{proof}

\begin{proofof}{Proposition~\ref{prop:treedisksector}}
We let $\rho_0$ and $h$ be large constants, to be determined in the course of the proof.
Let $\rho\geq \rho_0$, $T$ be an arbitrary $(\rho,w,\vartheta_1)$-tree, $uv \in E(T)$ an arbitrary edge
and $x \in T_{u\setminus v}$ an arbitrary vertex.

We first assume that $x\neq u$. Then $\dist_T(u,x) \geq 2$ by Proposition~\ref{prop:treeistree}, 
assuming without loss of generality we have chosen $rho_0$ sufficiently large.
And, assuming $\rho_0$ was chosen appropriately, by 
Proposition~\ref{prop:distHdistT} we have 

$$ \distH(v,x) \geq 2(\rho-K), $$ 

\noindent
where $K$ is as provided by the proposition.

By Lemma~\ref{lem:VTuvsect}, provided we chose $\rho_0$ appropriately, we have 

$$ x \in \sect{v}{u}{\vartheta_2/2}. $$

Applying Lemma~\ref{lem:sec_subset}, provided we chose $\rho_0$ sufficiently large, we have

$$ \ballH(x,r+w) \subseteq \sect{v}{x}{\vartheta_2/2} \subseteq \sect{v}{u}{\vartheta_2}. $$

Let us thus consider the situation when $x=u$.
In this case we can apply Lemma~\ref{lem:bounding} to show that, provided we chose $\rho_0$ and $h$ appropriately large,

$$ \ballH(u,r+w) \subseteq \ballH(v,h) \cup \sect{v}{u}{\vartheta_2}. $$

This concludes the proof.
\end{proofof}

\subsection{The upper bound\label{sec:ub}}

Here we will show the following proposition, which constitutes half of our main result. 

\begin{proposition}\label{prop:ub} 
For every $\eps>0$, there exists a $\lambda_0=\lambda_0(\eps)>0$ such that 
for all $0<\lambda<\lambda_0$, we have $p_c(\lambda) \leq (1+\eps)\cdot (\pi/3)\cdot \lambda$.
\end{proposition}

In order to prove this, we add the origin $o$ to the Poisson point process $\Zcal$ and colour it black, and consider
the black component of $o$ in the Delaunay graph for $\Zcal \cup \{o\}$.
It suffices to show that, when $p\geq (1+\eps)\cdot (\pi/3)\cdot \lambda$ and $\lambda$ is sufficiently small, 
with positive probability the origin will be in an infinite black component.
This is because all edges not involving the origin are also in the Delaunay graph for $\Zcal$,
and the origin a.s.~has only finitely many neighbours (by Isokawa's formula), so if the origin is in an infinite component 
then removing the origin may split its component into several components but at least one of these will be infinite. 

For $u,v \in \Dee$ and $r,w,\vartheta > 0$ we define the region

$$ C(u,v,r,w,\vartheta) := 
\{ x \in \Dee : r-w < \distH(u,x) < r+w \text{ and } \angle vux > \vartheta \}. $$

\vspace{15ex} 

\begin{figure}[hbt]
\begin{center}
\begin{picture}(100,100)
\put(-30,0){\includegraphics[width=.4\textwidth]{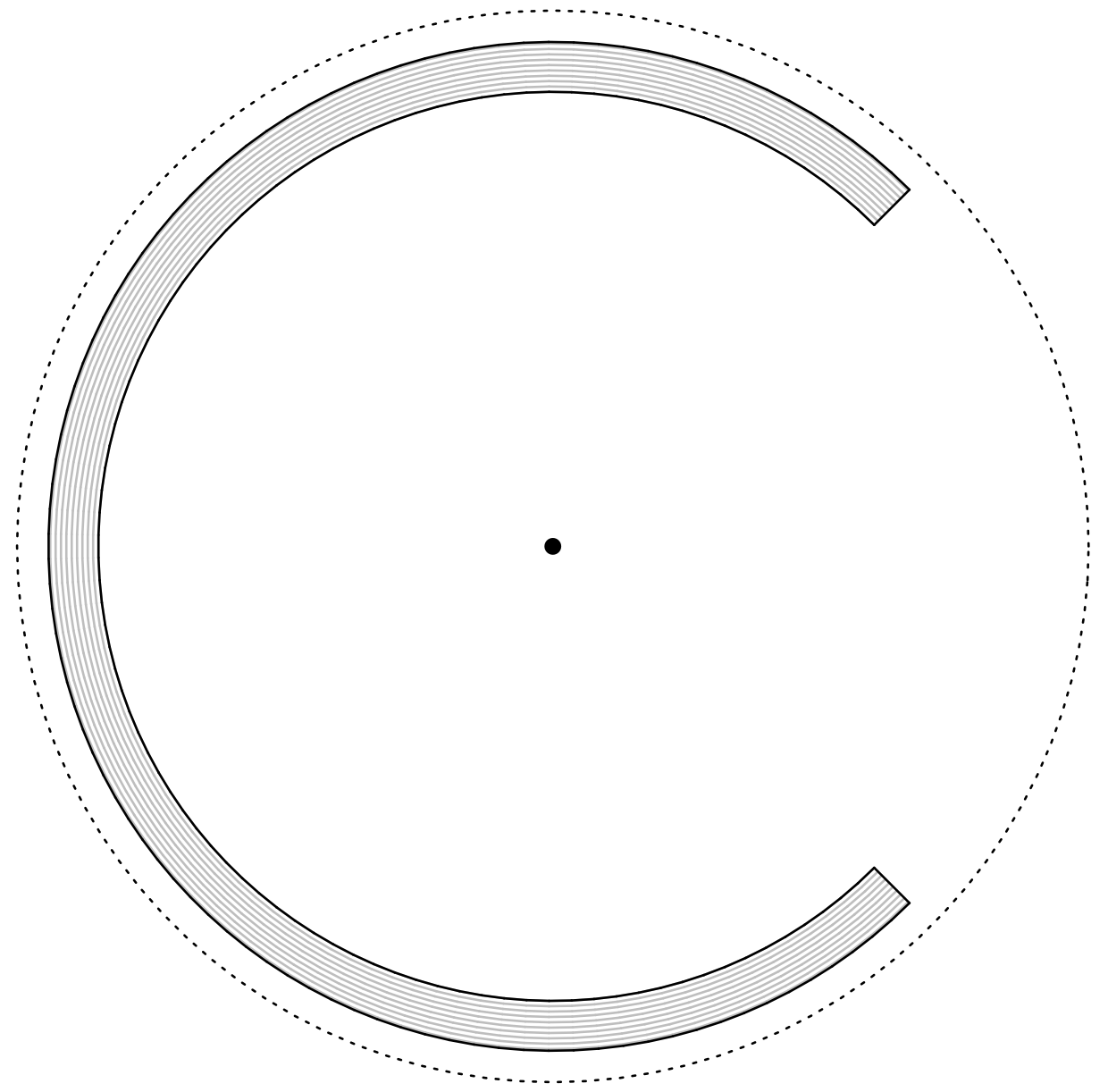}}
\put(50,80){$o$}
\end{picture}
\end{center}
\caption{The C-shaped region $C(o,(1/2,0),3,1/2,\pi/4)$.\label{fig:C}}
\end{figure}

\noindent
We choose to use the notation $C(.,.,.,.)$ because the region is shaped like the letter C, at least for 
some choices of the parameters. See Figure~\ref{fig:C}.

Given $u \neq v \in \Dee$ and $r,w,\vartheta, h>0$ we define $\Xcal = \Xcal(u,v,r,w,\vartheta,h)$
and $X = X(u,v,r,w,\vartheta,h)$ by:

$$ \Xcal := \left\{ z \in \Zcal_b : 
\begin{array}{l} 
z \in C(u,v,r,w,\vartheta) \text{ and;} \\
\text{$\angle z'uz > \vartheta$ for all $z'\neq z \in \Zcal_b \cap C(u,v,r,w,\vartheta)$, and; } \\
\Zcal \cap \ballH(z,h) = \{z\}, \text{ and; } \\
\text{$\exists$ a disk $B$ such that $u,z \in\partial B, B \cap \Zcal = \emptyset, \diam_{\Haa^2}(B) < r+w$.} 
\end{array}\right\}, $$

\begin{equation}\label{eq:Xdef} X := |\Xcal|. \end{equation}

\noindent
We point out that the probability distribution of $X$ 
does not depend on the choice of $u$ and $v$ (as long as they are distinct -- otherwise the definition does not make sense).

\begin{proposition}\label{prop:GWcouple}
For every $w, \vartheta > 0$ there exist $h=h(w,\vartheta)$ and $r_0 = r_0(w,\vartheta)$ such that, 
for all $\lambda>0, 0<p<1$ and $r > r_0$, the size of the black cluster of $o$ stochastically dominates 
the size of a Galton-Watson branching process with offspring distribution $X$.
\end{proposition}

\begin{proof}
We can assume without loss of generality $\vartheta<\pi$ (otherwise $X=0$ by definition, and the theorem holds trivially).
We set $\vartheta' \ll \vartheta$ be a small constant to be determined 
during the course of the proof, let $h=h(w,\vartheta,\vartheta')$ be as provided by Proposition~\ref{prop:treedisksector}, 
and we let $r_0$ be a large constant, to be determined more precisely during the course of the proof.

We consider an ``exploration process'' that iteratively constructs a $(r,w,\vartheta)$-tree $T$ rooted at the origin.
Recall that $r$ will be chosen larger than $r_0$.
At every iteration there will be a (finite) number of nodes, some of which are {\em explored} and some {\em unexplored}. 
The node $v$ having been explored means that all children of $v$ have already been added to the tree. 

For each iteration $i\geq 1$, let $T_i$ denote the tree we have constructed at the end of iteration $i$. 
We let $\Escr_i$ denote the set of nodes of $T_i$ that have been explored at the end of iteration $i$ 
and $\Uscr_i$ be the set of nodes in $T_i$ that 
have not yet been explored; and we set $\Ecal_0 = \emptyset, \Ucal_0 = \{o\}$.

If in any iteration there are no unexplored nodes, i.e.~$\Ucal_i = \emptyset$, then the construction of the tree stops and the 
final result is $T=T_i$. Otherwise, if in each iteration we always have at least one unexplored node, then we continue indefinitely
in which case of course $T := \bigcup_i T_i$.

In the first iteration we add the set $\Xcal_1 := \Xcal(o,(\tanh(r/2),0),r,w,\vartheta,h)$ defined above to the tree, as 
the children of the origin.
To be more precise we define $T_1$ by $V(T_1) = \{o\} \cup \Xcal_1, E(T_1) = \{ oz : z \in \Xcal_1 \}$.
We set $\Ecal_1 = \{o\}, \Uscr_1 = \Xcal_1$. 
In particular, the number of children of the origin will be $X_1 \isd X$.

In any subsequent iteration $i$, assuming $\Uscr_{i-1}\neq\emptyset$ (otherwise the construction process will have finished), 
we pick an arbitrary unexplored node $u_i \in \Uscr_{i-1}$, and we denote by $v_i$ its parent in $T_{i-1}$.
The children of $u_i$ will be $\Xcal_i := \Xcal(u_i,v_i,r,w,\vartheta,h)$, 
and we let $X_i := |\Xcal_i|$ denote the number of children of $u_i$.
We update by defining $T_i$ via $V(T_i) = V(T_{i-1}) \cup \Xcal_i, E(T_i) = E(T_{i-1}) \cup \{ u_iz : z \in \Xcal_i \}$, and
setting $\Ecal_i = \Ecal_{i-1} \cup \{u_i\}, \Uscr_i = (\Uscr_{i-1} \setminus \{u_i\}) \cup \Xcal_i$.

We point out that, since $T_i$ clearly is a $(r,w,\vartheta)$-tree, it follows from 
Proposition~\ref{prop:treeistree} that we do not include any point twice (put differently that 
$\Xcal_i \cap \Xcal_j = \emptyset$ for all $i\neq j$), provided $r_0$ was chosen sufficiently large. 

It remains to see that the size of the tree $T$ constructed via this process dominates 
the size of a Galton-Watson tree with offspring distribution $X$. 
In order to do this, we will first establish that $X_1, X_2, \dots$ are independent, and then that 
$X_2, X_3,\dots$ are i.i.d.~and finally that $X_2$ stochastically dominates $X=X_1$.

We will consider, for each iteration $i$, a (random) region $R_i$ that will be 
{\em revealed} by the exploration process during the $i$-th iteration. 
Here we mean by ``reveal'' that the exploration process will use information about 
$\Zcal \cap R_i$ in order to determine $\Xcal_i$, but -- crucially -- after the $i$-th iteration the exploration process 
will not have uncovered any information on the status of the Poisson point process $\Zcal$ outside of $R_1\cup\dots\cup R_i$.

We define $\Xcal_i^+ \supseteq \Xcal_i$  by 

$$\Xcal_i^+ := \left\{ z \in \Zcal_b : \begin{array}{l} z \in C(u_i,v_i,r,w,\vartheta) \text{ and;} \\
\text{$\angle z'uz > \vartheta$ for all $z'\neq z \in \Zcal_b \cap C(u_i,v_i,r,w,\vartheta)$.} 
\end{array}\right\}, $$

\noindent
setting $u_1 = o, v_1 := (\tanh(r/2),0)$ so that the definition also applies for $i=1$.
We let $T_i^+$ be the tree on vertex set $\{o\} \cup \Xcal_1^+ \cup \dots \cup \Xcal_i^+$, rooted at $u_1=o$ and where the children 
of $u_j$ are $\Xcal_j^+$ for each $j=1,\dots,i$. Clearly $T_i^+$ is also a $(r,w,\vartheta)$-tree for each $i$.

Given $\Xcal_1^+, \dots, \Xcal_{i-1}^+$ and $u_i$, we can determine 
$\Xcal_i^+$ by revealing the status of the Poisson process inside $C(u_i,v_i,r,w,\vartheta)$.
In order to now determine $\Xcal_i \subseteq \Xcal_i^+$ we need to determine, for each $z \in \Xcal_i^+$ 
whether $\ballH(z,h) \cap \Zcal = \{z\}$ and there exists a disk $B$ such that 
$u_i,z\in\partial B, \Zcal\cap B = \emptyset$ and $\diam_{\Haa^2}(B) < r+w$. 
Any such $B$ is contained in $\ballH(u_i,r+w) \cap \ballH(z,r+w)$.
Hence, given $\Xcal_1^+, \dots, \Xcal_{i-1}^+$ and $u_i$, we can determine $\Xcal_i$ by revealing the 
status of the Poisson process inside the region

$$ R_i :=  \bigcup_{z\in C(u_i,v_i,r,w,\vartheta)} \left(\ballH(u_i,r+w)\cap \ballH(z,r+w)\right) 
\cup \bigcup_{z\in \Xcal_i^+} \ballH(z,h). $$

\noindent
By Proposition~\ref{prop:treedisksector} (applied to the tree consisting of a single edge $u_iz$) and 
the choice of $h, r_0$, we have 

$$ \ballH(z,r+w) \subseteq \ballH(u_i,h) \cup \sect{u_i}{z}{\vartheta'}, $$

\noindent
for all $z \in C(u_i,v_i,r,w,\vartheta)$.
It follows that 

\begin{equation}\label{eq:Ri} 
\begin{array}{rcl} R_i 
& \subseteq & \displaystyle 
\ballH(u_i,h) \cup \left( \bigcup_{z \in C(u_i,v_i,r,w,\vartheta)} \sect{u_i}{z}{\vartheta'} \right) \\[4ex]
& \subseteq & \displaystyle 
\ballH(u_i,h) \cup \left( \Dee \setminus \sect{u_i}{v_i}{\vartheta-\vartheta'} \right) \\
& =: & \displaystyle 
S_i,
\end{array} 
\end{equation}

\noindent
(We point out that the arguments giving~\eqref{eq:Ri} also apply to the case when $i=1$, showing that 
$\Xcal_1$ is determined completely by $\Zcal\cap S_1$.)

On the other hand, we also have 

$$ R_j \subseteq \ballH(u_j,r+w) \cup \bigcup_{z\in\Xcal_j^+} \ballH(z,h), $$

\noindent
for every $j$. Hence 

\begin{equation}\label{eq:cupRj} 
\begin{array}{rcl} 
R_1 \cup \dots \cup R_{i-1} 
& \subseteq & \displaystyle 
\left( \bigcup_{j=1}^{i-1} \ballH(u_i,r+w) \right)
\cup \left( \bigcup_{v \in \{o\} \cup \Xcal_1^+ \cup \dots \cup \Xcal_{i-1}^+} 
\ballH(v,h) \right)
\\[6ex]
& \subseteq & \displaystyle 
\ballH(u_i,h) \cup \left( \bigcup_{v \in \{o\} \cup \Xcal_1^+ \cup \dots \cup \Xcal_{i-1}^+, \atop v \neq u_i} 
\ballH(v,r+w) \right) \\[6ex]
& = & \displaystyle 
\ballH(u_i,h) \cup \left( \bigcup_{v \in V( (T_{i-1}^+)_{v_i\setminus u_i} )} \ballH(v,r+w) \right) \\[6ex]
& \subseteq & \displaystyle 
\ballH(u_i,h) \cup \sect{u_i}{v_i}{\vartheta'}, 
\end{array}
\end{equation}

\noindent 
where we apply Proposition~\ref{prop:treedisksector} in the last line, and we assume we chose $r_0$ sufficiently large.

Combining~\eqref{eq:Ri} and~\eqref{eq:cupRj}, having chosen $\vartheta'$ sufficiently small, we see that 

$$ (R_1\cup\dots\cup R_{i-1}) \cap S_i \subseteq \ballH(u_i,h). $$

Moreover, for $i>1$, by construction of the exploration process we have

$$ \Zcal \cap \ballH(u_i,h) = \{u_i\}. $$

Hence, for $i>1$, given $\Xcal_1,\Xcal_1^+, \dots, \Xcal_{i-1}, \Xcal_{i-1}^+, u_1,\dots,u_i$, the random set 
$\Xcal_i$ is completely determined by $\Zcal \cap S_i'$ where

$$ S_i' :=  S_i \setminus \ballH(u_i,h) 
=  \Dee \setminus \left( \ballH(u_i, h) \cup \sect{u_i}{v_i}{\vartheta-\vartheta'} \right). $$ 

\noindent
(For $i=1$ it might be the case that $\ballH(o,h)$ contains point of $\Zcal$. So we cannot say that 
$\Xcal_1$ is completely determined by $\Zcal \cap S_1'$. It is however completely determined by 
$\Zcal \cap S_1$.)

We consider an isometry $\varphi:\Dee\to\Dee$ satisfying $\varphi(u_i)=o$ and that $\varphi(v_i)$ lies on the positive $x$-axis.
Thus

$$ \varphi\left[ S_i' \right] =  \Dee \setminus \left( \ballH(o,h) \cup \sect{o}{v_1}{\vartheta-\vartheta'}\right) = S_1', $$

\noindent
and of course $\varphi\left[ \ballH(u_i,h)\right] = \ballH(o,h)$.

Since $S_i'$ does not not intersect the areas $R_1,\dots,R_{i-1}$ revealed by previous iterations and 
it is isometric to $S_1'$ for each $i>1$, we find that $X_1, X_2, \dots$ are independent, 
and for $i>1$ in fact 

$$ X_i \isd \left( X_1 {\Big|} |\Zcal\cap\ballH(o,h)|=0 \right) \quad \text{ (for $i>1$.) } $$

\noindent
(We write $|\ballH(o,h) \cap \Zcal| = 0$ and not $\ballH(o,h) \cap \Zcal = \{o\}$ since $o \not\in \Zcal$.)

If $X_1, X_2, \dots$ were i.i.d.~then $T$ would be a Galton-Watson tree with offspring distribution $X_1$. 
In our case, we can describe the situation by saying the tree $T$ consists of a root attached to $X_1$ independent copies of a 
Galton-Watson tree with offspring distribution $\tilde{X_1} = \left( X_1 | \Zcal\cap\ballH(o,h)=\emptyset\right)$.
We also point out that the sequence $X_1, X_2, \dots$ completely determines the size of the tree $T$, via

$$ |V(T)|= \inf\{ n : X_1+\dots+X_n\leq n-1\}. $$ 

(See for instance Sections 1.5 and 1.6 of~\cite{mongolia}.
Note that while the discussion there focuses on $X_1,X_2,\dots$ i.i.d., the above equation holds 
much more generally. See the remark following Definition 1.14 in~\cite{mongolia}.)

To conclude the proof, we note that, provided we chose $r_0$ sufficiently large, 
$\ballH(o,h) \cap C(u_1,v_1,r,w,\vartheta) = \emptyset$.
Therefore, any point of $\Zcal$ in $\ballH(o,h)$ can only ``prevent'' the formation of edges between $o$ and some
$z \in \Zcal_b \cap C(u_1,v_1,r,w,\vartheta)$. 
In particular, for any $k, \ell  \in \eN \cup \{0\}$ we have 

$$ \Pee\left( X_1 \geq k {\Big|} |\Zcal \cap \ballH(o,h)| = \ell \right) 
\leq \Pee\left( X_1 \geq k {\Big|} |\Zcal \cap \ballH(o,h)| = 0 \right) 
= \Pee( \tilde{X_1} \geq k ), $$

\noindent 
which gives

$$ \Pee( X_1 \geq k ) 
= 
\sum_{\ell=0}^\infty \Pee\left( X_1 \geq k {\Big|} |\Zcal \cap \ballH(o,h)| = \ell \right) 
\Pee\left( |\Zcal\cap\ballH(o,h)| = \ell \right)
\leq 
\Pee( \tilde{X_1} \geq k ). $$

In other words, $\tilde{X_1}$ stochastically dominates $X_1$. 
By Strassen's theorem (\cite{Strassen65}; an elementary proof of the version we need can for instance 
be found in Section 2.3 of~\cite{VdHofstadvol1}) there is a coupling of the sequence $X_1, X_2, X_3, \dots$ and 
an i.i.d.~sequence $Y_1, Y_2, \dots$ such that $Y_i \isd X_1$ for all $i$ and $X_i \geq Y_i$ almost surely.
The sequence $Y_1,Y_2, \dots$ can be used to generate a Galton-Watson tree $T'$ with offspring distribution $X_1$.
We have $|V(T')| = \inf\{ n\geq 1 : Y_1+\dots+Y_n \leq n-1\}$.
Hence, (under the coupling, almost surely) $|V(T')| \leq |V(T)|$.
In particular the size of the black cluster of the origin (which is at least $|V(T)|$) stochastically dominates $|V(T')|$.
This is what needed to be shown.
\end{proof}

Having established Proposition~\ref{prop:GWcouple}, in order to prove Proposition~\ref{prop:ub} it suffices to 
show that, for $\lambda$ sufficiently small and $p = (1+\eps)\cdot (\pi/3)\cdot \lambda$, 
there is a choice of $w,\vartheta,h,r$ such that 
($h = h(w,\vartheta), r \geq r_0(w,\vartheta)$ with $h(.,.), r_0(.,.)$ as specified in Proposition~\ref{prop:GWcouple}, and)
$\Ee X > 1$.
More specifically, we'll keep $w,\vartheta$ (and $h$) constant, but we'll let $r$ depend on $\lambda$.

We break the argument down in a series of relatively straightforward lemmas.

\begin{lemma}\label{lem:XI}
For every $w, \lambda>0$ and all $p \leq 10\lambda$
the following holds, setting $r := 2\log(1/\lambda)$.
Writing

$$ X_{\bf{I}}  := 
 \left|\Zcal_b \cap \ballH(o,r-w) \right|, $$

\noindent
we have

$$ \Ee X_{\bf{I}} \leq 1000 e^{-w}. $$

\end{lemma}

\begin{proof}
If $r-w < 0$ then clearly $X_{\bf{I}} = 0$ almost and we are done. So we can assume $r-w\geq 0$.
Clearly 

$$ \begin{array}{rcl} 
\Ee X_{\bf{I}} 
& = & p \lambda \cdot \areaH\left( \ballH(o,r-w) \right) \\
& = & p \lambda \cdot 2\pi\left(\cosh(r-w)-1\right) \\
& \leq & 1000 \lambda^2 e^{r-w} \\
& = & 1000 e^{-w}, 
\end{array} $$

\noindent
using that $\cosh(x)-1 \leq e^x$ for $x \geq 0$, and the choice of $r$.
\end{proof}

\begin{lemma}\label{lem:XII}
For every $w,\lambda, \vartheta >0$ and all $p \leq 10\lambda$
the following holds, setting $r := 2\log(1/\lambda)$ and letting
$v \in \Dee\setminus\{o\}$ be an arbitrary (fixed) point.
Writing 

$$ X_{\bf{II}}  := 
 \left|\Zcal_b \cap \ballH(o,r+w) \cap \sect{o}{v}{\vartheta} \right|, $$

we have

$$ \Ee X_{\bf{II}} \leq 1000 \vartheta e^{w}. $$

\end{lemma}

\begin{proof}
Clearly 

$$ \begin{array}{rcl} \Ee X_{\bf{II}} & = & p \lambda \areaH\left( \ballH(o,r+w) \cap \sect{o}{v}{\vartheta} \right) \\
& = & p \lambda \cdot \frac{2\vartheta}{2\pi} \cdot 2\pi \left(\cosh(r+w)-1\right) \\
& \leq & 1000 \vartheta \lambda^2 e^{r+w} \\
& = & 1000 \vartheta e^{w}, 
\end{array} $$

\noindent
as before using that $\cosh(x)-1 \leq e^x$ and the choice of $r$.
\end{proof}

\begin{lemma}\label{lem:XIII}
For every $w,\lambda, \vartheta >0$ and all $p \leq 10\lambda$
the following holds, setting $r := 2\log(1/\lambda)$.
Writing 

$$ X_{\bf{III}}  := 
 \left|\left\{ (z_1,z_2) \in \Zcal_b\times\Zcal_b : \begin{array}{l} z_1,z_2 \in \ballH(o,r+w), \text{ and;} \\
 z_1\neq z_2, \text{ and;}, \\
 \angle z_1oz_2 < \vartheta.
 \end{array} \right\}\right|, $$

we have

$$ \Ee X_{\bf{III}} \leq 1000 \vartheta e^{2w}. $$

\end{lemma}

\begin{proof}
By Corollary~\ref{cor:SlivMeck2col}:

$$ \Ee X_{\bf{III}} = p^2 \lambda^2 \int_\Dee\int_\Dee \Ee\left[g(z_1,z_2,\Zcal\cup\{z_1,z_2\})\right] f(z_1)f(z_2)\dd z_2\dd z_2, $$

\noindent
where $f$ is as given by~\eqref{eq:fdef} and 

$$ g(u_1,u_2,\Ucal) = 1_{\left\{\begin{array}{l} 
u_1,u_2 \in \ballH(o,r+w), \text{ and;} \\
u_1\neq u_2, \text{ and;}, \\
\angle u_1ou_2 < \vartheta.
 \end{array}\right\}}. $$
 
In other words

$$ \begin{array}{rcl} 
\Ee X_{\bf{III}} 
& = & \displaystyle
p^2 \lambda^2 \int_{\ballH(o,r+w)} \int_{\ballH(o,r+w)\cap\sect{o}{z_1}{\vartheta}} f(z_1)f(z_2)\dd z_2\dd z_1 \\[2ex]
& = &  \displaystyle
p^2 \lambda^2 \int_{\ballH(o,r+w)} \areaH\left( \ballH(o,r+w)\cap\sect{o}{z_1}{\vartheta}\right) f(z_1) \dd z_1 \\[2ex]
& = &  \displaystyle
p^2 \lambda^2 \int_{\ballH(o,r+w)} \left(\frac{2\vartheta}{2\pi}\right) \cdot \areaH\left( \ballH(o,r+w) \right) \dd z_1 \\[2ex]
& = &  \displaystyle
p^2 \lambda^2 \left(\frac{2\vartheta}{2\pi}\right) \cdot \areaH\left( \ballH(o,r+w) \right)^2 \\[2ex]
& \leq &  \displaystyle
1000 \lambda^4 \vartheta e^{2r+2w} \\
& = & 
1000 \vartheta e^{2w},
\end{array} $$

\noindent
using rotational symmetry of the hyperbolic area measure in the third line; 
that $\areaH(\ballH(o,x)) = 2\pi(\cosh(x)-1) \leq \pi e^x$ and the bound on $p$ in the fifth line; and 
the choice of $r$ in the last.
\end{proof}

\begin{lemma}\label{lem:XIV}
For every $w,\lambda, h >0$ and all $p \leq 10\lambda$
the following holds, setting $r := 2\log(1/\lambda)$.
Writing 

$$ X_{\bf{IV}}  := 
 \left|\left\{ (z_1,z_2) \in \Zcal_b\times\Zcal_b : \begin{array}{l} z_1 \in \ballH(o,r+w), \text{ and;} \\
 z_1\neq z_2, \text{ and;}, \\
 \distH(z_1,z_2) < h.
 \end{array} \right\}\right|, $$

we have

$$ \Ee X_{\bf{IV}} \leq 1000 \lambda^2 e^{w+h}. $$

\end{lemma}

\begin{proof}
Applying Corollary~\ref{cor:SlivMeck2col} again, we have

$$ \Ee X_{\bf{IV}} = p^2 \lambda^2 \int_\Dee\int_\Dee \Ee\left[g(z_1,z_2,\Zcal\cup\{z_1,z_2\})\right] f(z_1)f(z_2)\dd z_2\dd z_2, $$

\noindent
where $f$ is given by~\eqref{eq:fdef} and this time

$$ g(u_1,u_2,\Ucal) = 1_{\left\{\begin{array}{l} 
u_1 \in \ballH(o,r+w), \text{ and;} \\
u_1\neq u_2, \text{ and;}, \\
\distH(u_1,u_2) < h.
 \end{array}\right\}}. $$
 
In other words 

$$ \begin{array}{rcl}
  \Ee X_{\bf{IV}}  
  & = &  \displaystyle
  p^2 \lambda^2 \int_{\ballH(o,r+w)}\int_{\ballH(z_1,h)} f(z_1)f(z_2)\dd z_2 \dd z_1 \\
  & = & \displaystyle
  p^2 \lambda^2 \cdot \areaH\left(\ballH(o,r+w)\right) \cdot \areaH\left(\ballH(o,h)\right) \\
  & \leq & \displaystyle
  1000 \lambda^2 e^{w+h},
  \end{array} $$
  
\noindent
again using that $\areaH(\ballH(o,x)) = 2\pi\left(\cosh x - 1\right) \leq \pi e^x$ and the choice of $r$. 
\end{proof}

\begin{lemma}\label{lem:XV}
There exists a $\lambda_0>0$ such that, for all $0<\lambda<\lambda_0$, all $w\geq 0$ and all $p \leq 10\lambda$,
the following holds, setting $r := 2\log(1/\lambda)$.
Writing

$$ X_{\bf{V}}  := 
 \left|\left\{ z \in \Zcal_b : \begin{array}{l} \distH( z, o ) \geq  r+w, \text{ and;} \\
 \text{$\exists$ a disk $B$ with $o,z \in \partial B$ and $\Zcal \cap B = \emptyset$.} 
 \end{array} \right\}\right|, $$

we have

$$ \Ee X_{\bf{V}} \leq 1000 e^{w} e^{-e^{w/2}}. $$

\end{lemma}

In the proof we'll make use of the following definition and observations, that we'll reuse later.
For $z_1,z_2 \in \Dee$ the {\em Gabriel disk} is the disk $\BGab(z_1,z_2) := \ballH(c,\distH(z_1,z_2)/2 )$ whose
center $c$ is the midpoint of the hyperbolic line segment between $z_1$ and $z_2$ and whose radius
is half the hyperbolic distance between $z_1$ and $z_2$.
Put differently, among all disks that have both $z_1$ and $z_2$ on their boundary, $\BGab(z_1,z_2)$ is the 
one of smallest hyperbolic radius.

The name ``Gabriel disk'' is in reference to the {\em Gabriel graph}, an object that has received some attention 
in the discrete and computational geometry literature. 
The Gabriel graph associated with a point set $V \subseteq \eR^2$ is the subgraph of the Delaunay graph whose edges
are all pairs $v_1,v_2 \in V$ for which the Euclidean disk whose center is the midpoint between $v_1$ and $v_2$ and 
whose  radius is half their distance contains no other points of $V$.

The hyperbolic line segment between $z_1$ and $z_2$ splits $\BGab(z_1,z_2)$ into two parts of equal hyperbolic area. 
We shall denote by $\BGab^-(z_1,z_2)$ the part that is on the left 
as we travel from $z_1$ to $z_2$ along the hyperbolic line segment, and  by $\BGab^+(z_1,z_2)$ the one that is on the right.
See Figure~\ref{fig:BGab}.

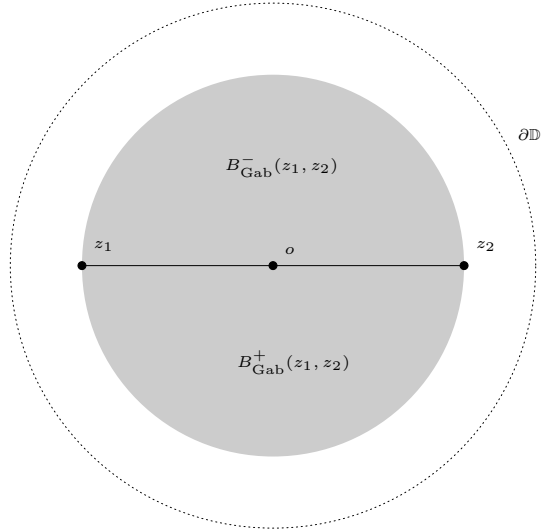
\begin{figure}[hbt]
\begin{center}
\input{BGab.pspdftex}
\end{center}
 \caption{The Gabriel disk $\BGab(z_1,z_2)$, in the special case when 
 $z_1,z_2$ lie on the $x$-axis and their midpoint is the origin.\label{fig:BGab}} 
\end{figure}

We'll repeatedly make use of the following straightforward observation.

\begin{itemize}
\item[{\bf($\spadesuit$)}] For all $z_1,z_2\in\Dee$ and every disk $B$ such that $z_1,z_2 \in \partial B$, we 
have either $\BGab^-(z_1,z_2) \subseteq B$ or 
$\BGab^+(z_1,z_2) \subseteq B$ (or both).
\end{itemize}

\noindent
(This is easily seen by applying a suitable isometry that maps $z_1,z_2$ to the $x$-axis and their midpoint to the origin, so that
$\BGab^-, \BGab^+$ are mapped to ordinary, Euclidean half-disks -- as in Figure~\ref{fig:BGab}.)

In light of observation {\bf($\spadesuit$)}, in order to prove Lemma~\ref{lem:XV} it suffices to prove the following
statement, which we separate out as a lemma for convenient future reference.

\begin{lemma}\label{lem:XVtil}
There exists a $\lambda_0>0$ such that, for all $0<\lambda<\lambda_0$, all $w\geq 0$ and all $p \leq 10\lambda$,
the following holds, setting $r := 2\log(1/\lambda)$.
Writing

$$ \tilde{X}_{\bf{V}}  := 
 \left|\left\{ z \in \Zcal_b : \begin{array}{l} \distH( z, o ) \geq  r+w, \text{ and;} \\
 \text{$\BGab^-(z,o)\cap\Zcal=\emptyset$ or $\BGab^+(z,o)\cap\Zcal=\emptyset$.} 
 \end{array} \right\}\right|, $$

we have

$$ \Ee \tilde{X}_{\bf{V}} \leq 1000 e^{w} e^{-e^{w/2}}. $$

\end{lemma}

\begin{proofof}{Lemma~\ref{lem:XVtil}}
We let $\lambda_0 > 0$ be a small constant, to be determined in the course of the proof.
By Corollary~\ref{cor:SlivMeck2col}

$$ \Ee \tilde{X}_{\bf{V}} = p\lambda \int_\Dee\Ee\left[g(z,\Zcal\cup\{z\})\right] f(z)\dd z, $$

\noindent
where $f$ is again given by~\eqref{eq:fdef} and 

$$ g(u,\Ucal) := 1_{\left\{
\begin{array}{l} \distH( u, o ) \geq  r+w, \text{ and;} \\
 \text{$\BGab^-(u,o)\cap\Ucal=\emptyset$ or $\BGab^+(u,o)\cap\Ucal=\emptyset$.} 
 \end{array}
\right\}}. $$

We have

$$ \begin{array}{rcl} 
\Ee \tilde{X}_{\bf{V}} 
& = & \displaystyle 
p\lambda \int_{\Dee \setminus \ballH(o,r+w)}
\Pee\left( \BGab^-(o,z) \cap \Zcal = \emptyset \text{ or } \BGab^+(o,z) \cap \Zcal = \emptyset \right) f(z) \dd z \\[2ex]
& \leq & \displaystyle 
10 \lambda^2 \int_{\Dee \setminus \ballH(o,r+w)} 2 \exp\left[-\frac12 \cdot \lambda \cdot \areaH\left( \BGab(o,z) \right) \right]
f(z)\dd z \\[2ex]
& \leq & \displaystyle
20 \lambda^2 \int_{\Dee \setminus \ballH(o,r+w)} \exp\left[ -\lambda e^{\distH(o,z)/2}\right] f(z)\dd z,  
\end{array} $$

\noindent
where in the last line we use that $\lambda_0$ was chosen sufficiently small; that 
$2\pi(\cosh x - 1) = (1+o_x(1)) \pi e^x$ as $x\to\infty$
so that $2\pi(\cosh x - 1)\geq 2 e^x$ for $x$ sufficiently large; that $\distH(o,z) \geq r+w$; and that 
$r\geq 2\ln(1/\lambda_0)$.

Switching to hyperbolic polar coordinates (i.e.~$z(\alpha,\rho)=(\cos(\alpha)\cdot\tanh(\rho/2),\sin(\alpha)\cdot\tanh(\rho/2))$)
we find

$$ \begin{array}{rcl} 
\Ee \tilde{X}_{\bf{V}} 
& \leq & \displaystyle 
20 \lambda^2 \int_{r+w}^\infty\int_0^{2\pi} e^{ -\lambda e^{\rho/2}}\sinh(\rho)\dd\alpha\dd\rho \\[2ex]
& = & \displaystyle 
20 \lambda^2 \int_{r+w}^\infty e^{ -\lambda e^{\rho/2}} 2\pi \sinh(\rho)\dd\rho \\[2ex]
& \leq & \displaystyle 
20\pi \lambda^2 \int_{r+w}^\infty e^{ -\lambda e^{\rho/2}} e^{\rho} \dd\rho \\[2ex]
& = & \displaystyle
40\pi \int_{\lambda e^{(r+w)/2}}^\infty e^{-u} u \dd u \\[2ex]
& = & \displaystyle
40\pi \left( \lambda e^{(r+w)/2} + 1 \right) \exp\left[ - \lambda e^{(r+w)/2} \right] \\[2ex]
& = & \displaystyle 
40\pi \left( e^{w/2}+1 \right) e^{-e^{w/2}} \\[2ex]
& \leq & 
1000 e^w e^{-e^{w/2}},
\end{array} $$

\noindent 
using the substitution $u = \lambda e^{\rho/2}$ (so that $\dd\rho = \frac{2\dd u}{u}$) in the third line
\end{proofof}

\begin{lemma}\label{lem:XVI}
There exists a $\lambda_0 >0$ such that for all $0<\lambda<\lambda_0$, all $w\geq 0$ and all $p \leq 10\lambda$
the following holds, setting $r := 2\log(1/\lambda)$.
Writing

$$ X_{\bf{VI}}  := 
 \left|\left\{ z \in \Zcal_b : \begin{array}{l} r-w < \distH( z, o ) <  r+w, \text{ and;} \\
 \text{$\exists$ a disk $B$ with $o,z \in \partial B, \Zcal \cap B = \emptyset$ and $\diam_{\Haa^2}(B) \geq r+w$.} 
 \end{array} \right\}\right|, $$

we have

$$ \Ee X_{\bf{VI}} \leq 1000 e^{w} e^{-e^{w/2}}. $$

\end{lemma}

In the proof we'll make use of the following definition and observations, that we'll reuse later.
For $z_1,z_2 \in \Dee$ and $\rho>0$ with $\distH(z_1,z_2) < \rho$, there
exist precisely two disks $B$ such that $z_1,z_2 \in \partial B$ and $\diam_{\Haa^2}(B) = \rho$.
We let $\DD(z_1,z_2,\rho)$ denote the union of these two disks.
(The notation $\DD$ stands for ``double disk''.)
The hyperbolic line segment between $z_1$ and $z_2$ splits $\DD(z_1,z_2,\rho)$ into two parts 
of equal hyperbolic area.
We denote by $DD^-(z_1,z_2,\rho)$ the part that is on the left as we travel from $z_1$ to $z_2$ along the 
hyperbolic line segment between them, and by $\DD^+(z_1,z_2,\rho)$ the part on the right. 
See Figure~\ref{fig:DD}.

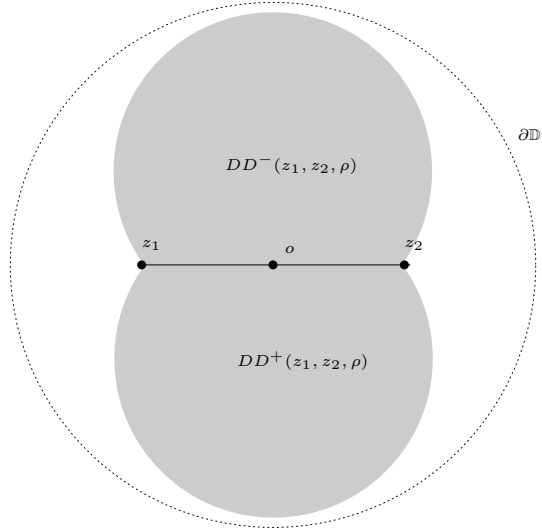
\begin{figure}[hbt]
\begin{center}
\input{DD.pspdftex}
\end{center}
 \caption{The set $\DD(z_1,z_2,\rho)$, in the special case when $z_1,z_2$ lie on the $x$-axis 
 with the origin $o$ as their midpoint.\label{fig:DD}}
\end{figure}

Another  observation that we'll use repeatedly is:

\begin{itemize}
\item[{\bf($\clubsuit$)}] For all $z_1\neq z_2 \in\Dee$ and $\rho>\distH(z_1,z_2)$ and every hyperbolic disk $B$ such that 
$z_1,z_2 \in \partial B$ 
and $\diam_{\Haa^2}(B) \geq \rho$, we have either $\DD^-(z_1,z_2,\rho) \subseteq B$ or $\DD^+(z_1,z_2,\rho) \subseteq B$ (or both).
\end{itemize}

\noindent
(Again this is easily seen by applying a suitable isometry that maps $z_1,z_2$ to the $x$-axis and their midpoint to the origin, 
as in Figure~\ref{fig:DD}. Fact~\ref{lem:conformal1} ensures that the image of the hyperbolic disk $B$ is a 
Euclidean disk with the images of $z_1,z_2$ on its boundary.)

In light of observation {\bf($\clubsuit$)}, in order to prove Lemma~\ref{lem:XVI} it suffices to prove the following
statement, which we separate out as a lemma for convenient future reference.

\begin{lemma}\label{lem:XVItil}
There exists a $\lambda_0 > 0$ such that for all $0<\lambda<\lambda_0$, all $w\geq 0$ and all $p \leq 10\lambda$
the following holds, setting $r := 2\log(1/\lambda)$.
Writing

$$ \tilde{X}_{\bf{VI}}  := 
 \left|\left\{ z \in \Zcal_b : \begin{array}{l} r-w < \distH( z, o ) <  r+w, \text{ and;} \\
 \text{$\DD^-(o,z,r+w)\cap\Zcal=\emptyset$ or $\DD^+(o,z,r+w)\cap\Zcal=\emptyset$.} 
 \end{array} \right\}\right|, $$

we have

$$ \Ee \tilde{X}_{\bf{VI}} \leq 1000 e^{w} e^{-e^{w/2}}. $$

\end{lemma}

\begin{proofof}{Lemma~\ref{lem:XVItil}}
Applying Corollary~\ref{cor:SlivMeck2col}, we have

$$ \Ee \tilde{X}_{\bf{VI}} = p\lambda \int_{\ballH(o,r+w)} \Pee( \DD^-(o,z,r+w) \cap \Zcal = \emptyset \text{ or } 
\DD^-(o,z,r+w) \cap \Zcal = \emptyset ) f(z) \dd z, $$

\noindent
where $f$ is as given by~\eqref{eq:fdef}.
Next we point out that, by symmetry

$$ \areaH( \DD^-(o,z,r+w) ) = \areaH( \DD^+(o,z,r+w) ) \geq \frac12 \areaH( \ballH(o,(r+w)/2 )) 
\geq e^{(r+w)/2}, $$

\noindent
where the last inequality holds provided we chose $\lambda_0$ sufficiently large, 
since $r \geq 2\ln(1/\lambda_0)$ and $\areaH(\ballH(0,x)) = 2\pi(\cosh x - 1) = (1+o_x(1)) \pi e^x$ as $x\to\infty$.

Hence

$$ \begin{array}{rcl} 
\Ee \tilde{X}_{\bf{VI}}
& \leq & p \lambda \areaH( \ballH(0,r+w) ) \cdot 2 e^{-\lambda e^{(r+w)/2}} \\
& \leq & 1000 \lambda^2 e^{r+w} e^{-\lambda e^{(r+w)/2}} \\
& = & 1000 e^w e^{-e^{w/2}}, 
\end{array} $$

\noindent
using the choice of $r$.
\end{proofof}

We are now ready to prove the following statement.

\begin{proposition}\label{prop:EXgr1}
For every $0<\eps<\frac{1}{1000}$ there exist $w,\vartheta>0$ such that the following holds. 
For every $h>0$ there exists a $\lambda_0=\lambda_0(\eps,w,\vartheta,h)$ such that, for all $0<\lambda<\lambda_0$, 
setting $p:=(1+\eps)(\pi/3)\lambda$ and $r := 2\ln(1/\lambda)$, we have that 

$$ \Ee X > 1. $$

\noindent
(With $X$ as defined in~\eqref{eq:Xdef}.)
\end{proposition}

\begin{proof}
As pointed out right after the definition of $X$, its probability distribution is the same for any choice of 
$u\neq v\in\Dee$. For definiteness we take $u=o$ and $v=(1/2,0)$.
We let $h>0$ be an arbitrary constant. We let $w,\lambda_0>0$ be large constants, and $\vartheta>0$ a small constant, 
to be determined in the course of the proof.

Let us denote by $D_b$ the number of black cells adjacent to the typical cell. (Recall that ``typical cell'' refers to the 
cell of the origin in the Voronoi tessellation for $\Zcal\cup\{o\}$.)
By Isokawa's formula

$$ \Ee D_b = p \Ee D = p \cdot \left( 6 + \frac{3}{\pi\lambda} \right) > 1+\eps, $$

\noindent 
using the choice of $p$ in the inequality.
Next, we point out that 

$$ X \geq D_b - \left( X_{\bf{I}} + X_{\bf{II}} + X_{\bf{III}} + X_{\bf{IV}}+X_{\bf{V}}+X_{\bf{VI}}\right). $$

\noindent
with $X_{\bf{I}}$--$X_{\bf{VI}}$ as defined in Lemmas~\ref{lem:XI}--\ref{lem:XVI}.
Applying these lemmas we see that, provided we chose $\lambda_0$ sufficiently small, 

$$ \Ee X > 1+\eps - 1000 \cdot \left( e^{-w} + \vartheta e^w + \vartheta e^{2w} + \lambda_0^2 e^{w+h} + 2 e^w e^{-e^{w/2}} \right)
> 1+\eps/2, $$

\noindent
where the second inequality holds provided we chose the constant $w$ sufficiently large and the constants 
$\vartheta, \lambda_0$ sufficiently small.
(To be more explicit : we can for instance first choose $w$ such that $e^{-w}, e^w e^{-e^{w/2}} < \eps / 10^5$ and then 
we choose $\vartheta$ such that $\vartheta e^{2w} < \eps/10^5$ and finally choose $\lambda_0$ such that 
$\lambda_0^2 e^{w+h} < \eps/10^5$ -- and at the same time $\lambda_0$ is small enough so that Lemma~\ref{lem:XV} and~\ref{lem:XVI} apply.)
\end{proof}

As pointed out earlier, Proposition~\ref{prop:ub}, the upper bound in our main theorem, follows directly from 
Propositions~\ref{prop:GWcouple} and~\ref{prop:EXgr1}.

\subsection{The lower bound\label{sec:lb}}

Here we will show the following proposition, which constitutes the remaining half of our main result. 

\begin{proposition}\label{prop:lb} 
For every $\eps>0$, there exists a $\lambda_0=\lambda_0(\eps)>0$ such that 
for all $0<\lambda<\lambda_0$, we have $p_c(\lambda) \geq (1-\eps)\cdot (\pi/3)\cdot \lambda$.
\end{proposition}

In the proof of Proposition~\ref{prop:ub} we could in a sense afford to ``ignore'' some adjacencies in the Voronoi tessellation. 
The supercritical Galton-Watson tree we constructed in the proof of Proposition~\ref{prop:GWcouple} uses only adjacencies of
convenient lengths and is such that if two edges share an endpoint, the angle they make is not too small. 

Now we need to show that no infinite black component exists almost surely, for $\lambda$ sufficiently small
and $p = (1-\eps)\cdot (\pi/3)\cdot \lambda$. 
We cannot a priori exclude the possibility that an infinite component exists but every infinite connected subgraph 
contains (has to contain)
 ``unusual'' edges. 
 In particular, a hypothetical infinite component
 might exist that does not contain an infinite $(r,w,\vartheta)$-tree. 
 
To make our life easier, we adopt a more generous notion of adjacency, in the form of {\em pseudopaths}.

\begin{definition}
Given parameters $r,w_1,w_2,\vartheta>0$, we say a (finite or infinite) sequence $u_0,u_1, \dots \in \Dee$ of distinct points 
is a {\em pseudopath} (wrt.~$r,w_1,w_2,\vartheta$ and $\Zcal$) if one of the following holds 
for each $i\geq 1$:

\begin{enumerate}
 \item[{\bf I}] $r-w_1 < \distH(u_{i-1},u_i) < r+w_2, \angle u_{i-2}u_{i-1}u_{i} \geq \vartheta$ and one of the following holds.
 There exists a disk $B$ such that $u_{i-1},u_{i}\in\partial B, 
 B \cap \Zcal = \emptyset$ and $\diam_{\Haa^2}(B) < r+w_2$, or; $\DD^-(u_{i-1},u_{i},r+w_2) \cap \Zcal = \emptyset$, or;
 $\DD^+(u_{i-1},u_{i},r+w_2) \cap \Zcal = \emptyset$.
\item[{\bf II}] $\distH(u_{i-1},u_{i}) \leq r-w_1$;
\item[{\bf III}] $r-w_1 < \distH(u_{i-1},u_{i}) < r+w_2$ and $\angle u_{i-2}u_{i-1}u_{i} < \vartheta$;
\item[{\bf IV}] $\distH(u_{i-1},u_{i}) \geq r+w_2$ and either $\BGab^-(u_{i-1},u_{i}) \cap \Zcal = \emptyset$, or
$\BGab^+(u_{i-1},u_{i}) \cap \Zcal = \emptyset$, (or both).
\end{enumerate}
\end{definition}

\noindent
(Of course the demand that $\angle u_{i-2}u_{i-1}u_{i} \geq \vartheta$ in {\bf I} only applies when $i\geq 2$, and 
similarly the case {\bf III} can only occur when $i\geq 2$.)
The length of a finite pseudopath $P=u_0,u_1,\dots,u_n$ is $n$, the number 
of points minus one, and we use the term {\em pseudo-edge} for a consecutive pair of points $u_{i-1}, u_i$ on a pseudopath.
In particular pseudopaths of length one are pseudo-edges, but not all pseudo-edges are pseudopaths of 
length one (because of {\bf III} which can only apply when the pseudopath has length $\geq 2$).
As in the case of ordinary paths, we denote by $V(P)$ the set of vertices of the pseudopath $P$.

With each pair of points $u,v \in \Dee$ of a pseudopath we associate a ``certificate'' 
$\cert(u,v) \subseteq \Dee$. This will be a region such that in order to verify that $uv$ is a pseudo-edge, in addition 
to the location of the previous point 
of the pseudopath if $uv$ is part of a pseudopath of length $\geq 2$, 
only $\Zcal \cap \cert(u,v)$ is relevant. 
Specifically, we set

$$ \cert(u,v) := \begin{cases} 
\DD(u,v,r+w_2) & \text{ if $r-w_1<\distH(u,v)<r+w_2$; } \\
\BGab(u,v) & \text{ if $\distH(u,v)\geq r+w_2$; } \\
\emptyset & \text{ otherwise. }
\end{cases}
$$ 

\noindent
(Note that if $\distH(u,v) < r+w_2$ and $B$ is a disk with $u,v\in\partial B$ and 
$\diam_{\Haa^2}(B) < r+w_2$ then $B \subseteq \DD(u,v,r+w_2)$.
Also note that if $\distH(u,v)\geq r-w_1$ then $\cert(u,v) \supseteq \BGab(u,v)$.)
For notational convenience we also set 

$$ \cert(u_0,\dots,u_k) := \cert(u_0,u_1) \cup \dots \cup \cert(u_{k-1},u_k), $$

\noindent
for every sequence $u_0,\dots,u_k \in \Dee$.

Pseudo-edges of type {\bf I} will be called {\em good}, and all other types of pseudo-edges are {\em bad}.
A pseudopath is good if all its pseudo-edges are good. 

A finite pseudopath $P = u_0,\dots,u_k$ is called a {\em chunk} if
the final pseudo-edge $u_{k-1}u_k$ is bad and all other pseudo-edges are good. 
So a chunk can for instance consist of a single bad pseudo-edge. 

\begin{definition}\label{def:seqchunks}
A {\em linked sequence of chunks} is a (finite or infinite) sequence $P_1, P_2,\dots$ of chunks such that
\begin{enumerate}
 \item $V(P_i)\cap V(P_j) =\cert(P_i) \cap \cert(P_j)=\emptyset$ if $|i-j|>1$, and;
 \item For each $i\geq 2$, writing $P_{i-1} = u_0^{i-1},\dots,u_{k_{i-1}}^{i-1}$ and $P_i = u_0^i,\dots,u_{k_i}^i$, 
 we have 
 
$$ \begin{array}{l} 
\{u_1^i,\dots,u_{k_i}^i\} \cap V(P_{i-1})=\emptyset, \\
\cert(u_1^i,\dots,u_{k_i}^i) \cap \cert(P_{i-1})  = \emptyset, \\
\distH( u_j^i, \cert(P_{i-1}))\geq \frac{r}{1000} \text{ for all $1\leq j \leq k_i$}, 
 \end{array} $$ 

\noindent
and in addition one of the following holds:
 \begin{itemize}
  \item[{\bf a)}] $u_0^i = u_{k_{i-1}}^{i-1}$, or
  \item[{\bf b)}] $u_0^i \not\in V(P_{i-1})$ and $\cert(u_0^i,u_1^i) \cap \cert( P_{i-1} ) \neq \emptyset$, or
  \item[{\bf c)}] $u_0^i \not\in V(P_{i-1}), \cert(u_0^i,u_1^i) \cap\cert( P_{i-1} )=\emptyset$ and 
  $\distH( u_0^i, \cert(P_{i-1})) < r/1000$.
 \end{itemize}
\end{enumerate}
\end{definition}

\noindent
We'll say a linked sequence of chunks is infinite if the sum of the lengths of the chunks is infinite. 
In other words, an infinite linked sequence of chunks either consists of infinitely many finite chunks, or 
consists of finitely many chunks of which the last chunk has infinite length.

\begin{proposition}\label{prop:lbprep}
For every $\lambda>0, 0\leq p \leq 1$ and $r,w_1,w_2,\vartheta>0$, almost surely, 
one of the following holds. 
\begin{enumerate} 
 \item\label{itm:lbprep1} The black component of $o$ in the Voronoi tessellation for $\Zcal \cup \{o\}$ is finite, or;
 \item\label{itm:lbprep2} There is a black, infinite, good pseudopath, or;
 \item\label{itm:lbprep3} There is a black, infinite linked sequence of chunks starting from $o$.
\end{enumerate}
\end{proposition}

\noindent
For clarity, we emphasize that the infinite pseudopath mentioned in item~\ref{itm:lbprep2} does not need to contain 
the origin $o$.

Our strategy for the proof of Proposition~\ref{prop:lb} is of course to show, after having established 
Proposition~\ref{prop:lbprep}, that there is a 
choice of $r,w_1,w_2,\vartheta$ such that, for all small enough $\lambda$ and all $p\leq(1-\eps)\cdot(\pi/3)\cdot\lambda$,
almost surely options~\ref{itm:lbprep2} and~\ref{itm:lbprep3} of Proposition~\ref{prop:lbprep} do not occur. 
This will imply that the black component of the origin in the Voronoi tessellation for $\Zcal\cup\{o\}$ is a.s.~finite. 
A short argument will then show this also implies that, almost surely, all black components in the Voronoi tessellation for 
$\Zcal$ are finite.

For the proof of Proposition~\ref{prop:lbprep} we will use the following observation that is a 
straightforward consequence of the Slivniak-Mecke formula and Isokawa's formula.
We provide a proof for completeness.

\begin{lemma}\label{lem:degfiniteas}
For every $\lambda>0$ and $0\leq p\leq 1$, almost surely, every Voronoi cell is adjacent 
to a finite number of other Voronoi cells.
\end{lemma}

\begin{proof}
Let us write 

$$ I := \left|\left\{ z \in \Zcal : \text{deg}(z) = \infty \right\}\right|. $$

By the Slivniak-Mecke formula

$$ \Ee I = \lambda \int_\Dee \Ee\left[g(z,\Zcal\cup\{z\})\right] f(z)\dd u, $$

\noindent
where $f$ is as given by~\eqref{eq:fdef} and $g(u,\Ucal) := 1_{\left\{\text{deg}(u;\Ucal)=\infty\right\}}$
with $\text{deg}(u;\Ucal)$ denoting the degree of $u$ in the Delaunay graph of $\Ucal$.
By symmetry considerations, for every (fixed) $u \in \Dee$ we have

$$ \Ee\left[\text{deg}(u;\Zcal\cup\{u\})\right] = \Ee\left[\text{deg}(o,\Zcal\cup\{o\})\right]  
= \Ee D = 6 + \frac{3}{\pi\lambda} < \infty, $$

\noindent
where $D$ denotes the typical degree and we apply Isokawa's formula.
It follows that, for every $u \in \Dee$:

$$ \Ee\left[g(u,\Zcal\cup\{u\})\right] = \Pee\left( \text{deg}(u;\Zcal\cup\{u\}) = \infty \right) = 0. $$

\noindent
Hence also $\Ee I = 0$, which implies that $I=0$ almost surely.
\end{proof}

We'll also need the following 
observation in the proof of Proposition~\ref{prop:lbprep}.

\begin{lemma} 
For every $\lambda>0$, almost surely, for every $x>0$ the number of
pseudo-edges $z_1z_2$ with $z_1,z_2\in \Zcal$ for which $\cert(z_1,z_2)\cap \ballH(o,x) \neq \emptyset$ is 
finite.
\end{lemma}

\begin{proof} Fix an arbitrary $x>0$.
Clearly the number of pseudo-edges $z_1z_2$ with $z_1,z_2\in \Zcal$ for which 
$\distH(z_1,z_2) < r+w_2$ and $\cert(z_1,z_2) \cap \ballH(o,x) \neq \emptyset$ 
is bounded by $X^2$ where $X := \Zcal \cap \ballH(o,x+2r+2w_2)$. 
Since $\Ee X = \lambda \cdot \areaH( \ballH(o,x+2r+2w_2) ) <\infty$, the random variable $X$ is finite almost surely.
 
Pseudo-edges $z_1z_2$ not counted by $X^2$ must satisfy $\distH(z_1,z_2) \geq r+w_2$ and hence
$\cert(z_1,z_2) =\BGab(z_1,z_2)$ and either $\BGab^-(z_1,z_1)\cap \Zcal=\emptyset$ or
$\BGab^+(z_1,z_2)\cap\Zcal=\emptyset$ (or both).
By symmetry considerations, it suffices to count

$$ Y := \left|\left\{ (z_1,z_2) \in \Zcal^2 : 
\begin{array}{l} 
\distH(z_1,o)>\distH(z_2,o), \text{ and; }\\
\BGab(z_1,z_2) \cap \ballH(o,x) \neq \emptyset, \text{ and; }\\
\text{$\BGab^-(z_1,z_1)\cap \Zcal=\emptyset$ or $\BGab^+(z_1,z_2)\cap\Zcal=\emptyset$.}
\end{array} \right\}\right|. $$

For each pair $(z_1,z_2)$ counted by $Y$ there is an $n\in\eN$ such that $n-1\leq \distH(o,z_1)\leq n$.
We must also have that $\distH(z_1,z_2) > \distH(o,z_1)-x$,
since $\diam(\BGab(z_1,z_2))=\distH(z_1,z_2)$ and $\BGab(z_1,z_2)\cap\ballH(o,x)\neq \emptyset$.

Writing

$$ Y_n := \left|\left\{ (z_1,z_2) \in \Zcal^2 : 
\begin{array}{l} 
\distH(z_1,o),\distH(z_2,o)\leq n \text{ and; }\\
\distH(z_1,z_2) > n-x-1, \text{ and; } \\
\text{$\BGab^-(z_1,z_1)\cap \Zcal=\emptyset$ or
$\BGab^+(z_1,z_2)\cap\Zcal=\emptyset$.}
\end{array} \right\}\right|, $$

\noindent
we thus have

$$ \Ee Y \leq \sum_n \Ee Y_n. $$

Applying the Slivniak-Mecke formula, we find that 

$$ \Ee Y_n 
= \lambda^2 \int_\Dee\int_\Dee 
\Ee\left[g(z_1,z_2,\Zcal\cup\{z_1,z_2\})\right] f(z_1)f(z_2) \dd z_1\dd z_2, 
$$ 

\noindent
where $f$ is as given by~\eqref{eq:fdef} and 

$$ g(u_1,u_2,\Ucal) 
:= 1_{\left\{\begin{array}{l} \distH(u_1,o),\distH(u_2,o)\leq n, \\ \distH(u_1,u_2) > n-x-1, \\
\text{$\BGab^-(u_1,u_1)\cap \Ucal=\emptyset$ or $\BGab^+(u_1,u_2)\cap\Ucal=\emptyset$.}
\end{array}\right\}}. $$

For every $z_1,z_2\in \Dee$ have 

$$ \begin{array}{rcl} 
\Ee\left[g(z_1,z_2,\Zcal\cup\{z_1,z_2\})\right] 
& \leq & \displaystyle 
1_{\left\{\begin{array}{l} \distH(z_1,o),\distH(z_2,o)\leq n, \\ \distH(z_1,z_2) > n-x-1 \end{array}\right\}}
\cdot 2 e^{-\lambda \cdot \areaH(\BGab(z_1,z_2))/2 } \\
& \leq & \displaystyle 
1_{\left\{\distH(z_1,o),\distH(z_2,o)\leq n\right\}} \cdot 2 e^{-c e^{n/2}},
\end{array} $$

\noindent 
for a suitably chosen small constant $c=c(x,\lambda)$.
(The last step in some more detail : for all $z_1,z_2$ for which $\distH(z_1,z_2)>n-x-1$, we have
$\areaH(\BGab(z_1,z_2)) \geq 2\pi(\cosh((n-x-1)/2)-1)$. Since $\cosh(t)=(1/2+o_t(1))e^t$ as $t\to\infty$, 
there is an $n_0$ such that $2\pi(\cosh((n-x-1)/2)-1)/2 \geq e^{(n-x-1)/2}$ for all $n\geq n_0$.
We can now choose a small $0<c< (\lambda/2)\cdot e^{-(x+1)/2}$ such that $2 e^{-c e^{n/2}} \geq 1$ for all $n\leq n_0$.)
It follows that 

$$ \begin{array}{rcl} \Ee Y_n & \leq & \left( \lambda \cdot \areaH(\ballH(o,n)) \right)^2 \cdot 2 e^{-c e^{n/2}} \\
& = &  \lambda^2 \cdot (2\pi)^2 \cdot \left( \cosh n - 1\right)^2  \cdot 2 e^{-c e^{n/2}} \\
& \leq & 2 \pi^2 \lambda^2 e^{2n} e^{-c e^{n/2}},
   \end{array} $$

\noindent
and hence

$$ \Ee Y \leq 2 \pi^2 \lambda^2 \sum_n e^{2n} e^{-c e^{n/2}} < \infty. $$

\noindent
So $Y$ is finite almost surely.

We have now shown that for every fixed $x>0$ we have $\Pee( N_x < \infty ) = 1$ where

$$ N_x := \left|\left\{\text{pseudo-edges $z_1,z_2\in \Zcal$ for which 
$\cert(z_1,z_2)\cap \ballH(o,x) \neq \emptyset$}\right\}\right|. $$

Since $N_x \leq N_y$ whenever $x<y$ we also have

$$ \Pee\left( \bigcap_{x>0} \{N_x<\infty\} \right)=\lim_{x\to\infty} \Pee(N_x < \infty) = 1, $$

\noindent
which is what needed to be shown.
\end{proof}

\begin{proofof}{Proposition~\ref{prop:lbprep}}
We fix arbitrary $\lambda>0, 0\leq p\leq 1$ and $r,w_1,w_2,\vartheta>0$. 
Since the typical degree has finite expectation by Isokawa's formula, almost surely, the cell of the origin 
in the Voronoi tessellation for $\Zcal\cup\{o\}$
is adjacent to at most finitely many other cells. For every point of $\Zcal$, the number of cells it is adjacent to in the Voronoi
tessellation for $\Zcal\cup\{o\}$ is at most one more than the number of cells it is adjacent to in the tessellation for $\Zcal$.
In particular, by Lemma~\ref{lem:degfiniteas}, also in the Voronoi tessellation for $\Zcal\cup\{o\}$, almost surely, each cell is
adjacent to at most finitely many other cells.

We consider an arbitrary realization of the Poisson process $\Zcal$ (in other words, a locally finite point set $\Zcal \subseteq \Dee$,
partitioned into two parts $\Zcal = \Zcalb\uplus\Zcalw$) 
for which each Voronoi cell of $\Zcal\cup\{o\}$ is adjacent to finitely many others, and such that, for every $x>0$, the ball 
$\ballH(o,x)$ intersects at most finitely many certificates of pseudo-edges.
We will show that in such a realization either the  black cluster of the origin is finite, or 
there exists an infinite, good, black pseudopath, or there exists an infinite 
linked sequence of chunks all of whose points are black (or more than one of the three options occurs).

Suppose thus that the origin $o$ lies in an infinite black component (otherwise we are done).
Since each cell is adjacent to finitely many others, there must exist an infinite
(ordinary) path $P = z_0,z_1, z_2, \dots$ with $z_0=o$ and $z_1,z_2,\dots \in \Zcal_b$ (distinct). 
This is of course also a pseudo-path. If only finitely many pseudo-edges of $P$ are bad then we are done, as $P$ will contain
a black, infinite, good pseudopath. 
Hence from now on we will assume that $P$ has infinitely many bad edges.
In particular, we can find a $t_1\geq 1$ such that 
the pseudo-edges $z_0z_1, z_1z_2,\dots,z_{t_1-2}z_{t_1-1}$ are good and $z_{t_1-1}z_{t_1}$ is bad.
We set $P_1 := z_0,z_1,\dots,z_{t_1}$.

Let us say that an index $i$ {\em interacts} with a subpath $Q = z_s,z_{s+1},\dots, z_t$ of $P$ if 
one of the following holds: $s\leq i\leq t$ or $\distH(z_i,\cert(Q)) < r/1000$ or
$\cert(z_i,z_{i+1}) \cap \cert(Q) \neq \emptyset$.

We now let $s_2\geq t_1$ be the largest index that interacts with $P_1$. 
That $s_2$ exists (is finite) follows from the fact that 

$$  
\bigcup_{u\in\cert(P_1)} \ballH\left(u,\frac{r}{1000}\right) \subseteq \ballH(o,x), $$

\noindent
for some $x>0$.
Since $P$ contains infinitely many bad pseudo-edges, there exists a $t_2 > s_2$ such that the 
pseudo-edges $z_{s_2}z_{s_2+1},\dots,z_{t_2-2}z_{t_2-1}$ are good and 
$z_{t_2-1}z_{t_2}$ is bad. We set $P_2 := z_{s_2}, z_{s_2+1},\dots,z_{t_2}$.
(A minor subtlety needs to be added here. 
When we say $z_{t_2-1}z_{t_2}$ is bad, we mean that it is bad wrt.~the path 
$z_{s_2}, z_{s_2+1},\dots,z_{t_2}$. We do this in order to handle the case where
$z_{s_2}z_{s_2+1}$ is a type {\bf III} pseudo-edge of $P$. If we took 
this single pseudo-edge as the path $P_2$, then $P_2$ would not be a chunk but 
a good pseudopath.)

We now let $s_3 \geq t_2$ be the largest index that interacts with $P_2$.  
Again $s_3$ exists as $\bigcup_{u\in\cert(P_2)} \ballH(u,r/1000)$ is contained in $\ballH(o,x)$ for some $x>0$.
Since $P$ contains infinitely many bad edges, there exists $t_3>s_3$ such that the pseudo-edges 
$z_{s_3}z_{s_3+1},\dots,z_{t_3-2}z_{t_3-1}$ are good and 
$z_{t_3-1}z_{t_3}$ is bad (wrt.~the path $z_{s_3}, z_{s_3+1},\dots,z_{t_3}$). 
We set $P_3 := z_{s_3}, z_{s_3+1},\dots,z_{t_3}$.

We continue defining $s_4,s_5,\dots$ and $t_4,t_5,\dots$ and $P_4, P_5,\dots$ analogously, producing an 
infinite linked sequence of chunks $P_1,P_2,\dots$.

(We point out that the choice of $s_i$ guarantees that no index $j>s_i$ interacts with any of $P_1,\dots,P_{i-1}$ while
$s_i$ interacts with $P_{i-1}$ but not $P_1,\dots,P_{i-2}$.
From this it easily follows that the demands {\bf(i)}, {\bf(ii)} of Definition~\ref{def:seqchunks} are met.)

Since we considered an arbitrary realization of $\Zcal$ that satisfies two conditions that both hold almost surely, it follows
that almost surely either the black cluster of the origin is finite or there exists an infinite linked sequence of 
chunks starting from the origin all of whose points are black.
\end{proofof}

The main thing that now remains, in order to prove Proposition~\ref{prop:lb}, is to show that for a suitable choice 
of $r,w_1,w_2,\vartheta$ and 
for $\lambda$ sufficiently small and $p=(1-\eps)(\pi/3)\lambda$ with $0<\eps<1$ a small constant, 
almost surely there is no black, infinite, good pseudopath and no black, infinite linked sequence of chunks starting from $o$.

As in the proof of the upper bound, from now on we will always take $r := 2\ln(1/\lambda)$.
The parameters $w_1,w_2,\vartheta>0$ will be constants that depend on $\eps$
and that we will choose appropriately in the end.
We will use an inductive approach bounding the expected number of good pseudopaths, respectively
linked sequences of chunks, of length $n$ starting from the origin.
With this in mind, we will prove a number of lemmas designed to deal with the different ways in which an 
additional point can be added to an existing good pseudopath, respectively linked sequence of chunks, of length $n-1$.
But, first we need to derive some additional facts about hyperbolic geometry.
Again, when reading the paper for the 
first time, the reader may wish to only read the definitions and lemma statements and skip over the proofs.

\subsubsection{More geometric ingredients}

Let us say that a sequence of points $u_0,u_1,\dots,u_k \in \Dee$ is a {\em pre-chunk} (wrt.~$r,w_1,w_2,\vartheta$) 
if they are placed in such a way that they will form 
a chunk whenever the random point set $\Zcal$
turns out favourably. That is, $u_0,u_1,\dots,u_k \in \Dee$ is a pre-chunk if $r-w_1<\distH(u_{i-1},u_i)<r+w_2$ for $i=1,\dots,k-1$ and 
$\angle u_{i-2}u_{i-1}u_i > \vartheta$ for all $2\leq i\leq k-1$ and 
in addition $\distH(u_{k-1},u_k)\leq r-w_1$ or $\distH(u_{k-1},u_k)\geq r+w_2$ or $r-w_1<\distH(u_{k-1},u_k)<r+w_2$ and 
$\angle u_{k-2}u_{k-1}u_k \leq \vartheta$. 
In particular, a pre-chunk is such that $u_0,\dots,u_{k-1}$ forms a $(r,w,\vartheta)$-path with $w := \max(w_1,w_2)$.

If  $r-w_1<\distH(u_{i-1},u_i)<r+w_2$ and 
$\angle u_{i-2}u_{i-1}u_i > \vartheta$ (in case $i\geq 2$) for $i=1,\dots, k$ then we speak of a {\em good pre-pseudopath}.

The first geometric observation in this section will be helpful for the situation when we want to add a new pseudo-edge
to an existing pseudopath, and will also be used by later proofs in this section.
It tells us that if we want to add a new pseudo-edge to a good pre-pseudopath, then either the 
intersection of the certificate of the new pseudo-edge and the certificates of the existing path 
is contained in a ball of constant radius (and hence has rather small area), or the new edge 
makes a small angle with the last edge of the existing pseudo-path. 
This will translate into useful bounds once we start
estimating the expected number of linked sequences of chunks later on.

\begin{lemma}\label{lem:certballH}
For every $w_1,w_2,\vartheta_1,\vartheta_2>0$ there are $r_0=r_0(w_1,w_2,\vartheta_1,\vartheta_2)$ and 
$h=h(w_1,w_2,\vartheta_1,\vartheta_2)$ such that 
the following holds for all $r\geq r_0$.
If $u_0,\dots,u_k$ is such that $u_0,\dots,u_{k-1}$ is a good pre-pseudopath wrt.~$r,w_1,w_2,\vartheta_1$ then either 

$$ \cert(u_0,\dots,u_{k-1}) \cap \cert(u_{k-1},u_k) \subseteq \ballH(u_{k-1},h), $$

\noindent
or 

$$ \angle u_{k-2}u_{k-1}u_k < \vartheta_2, $$ 

\noindent
(or both).
\end{lemma}

Before giving the proof, we point out that by setting $\vartheta_1=\vartheta_2$ we obtain:

\begin{corollary}\label{cor:certballH}
For every $w_1,w_2,\vartheta>0$ there are $r_0=r_0(w_1,w_2,\vartheta)$ and 
$h=h(w_1,w_2,\vartheta)$ such that 
the following holds for all $r\geq r_0$.
If $u_0,\dots,u_k$ is a good pre-pseudopath (wrt.~$r,w_1,w_2,\vartheta$) then

$$ \cert(u_0,\dots,u_{k-1}) \cap \cert(u_{k-1},u_k) \subseteq \ballH(u_{k-1},h). $$

\end{corollary}

\begin{proofof}{Lemma~\ref{lem:certballH}}
We fix $0<\vartheta'<\min(\vartheta_2/2,1/1000)$. 
For any $r>0$, if $P=u_0,\dots,u_k$ is a pre-chunk wrt.~$r,w_1,w_2,\vartheta_1$
then $P' = u_0, \dots, u_{k-1}$ is a $(r,w,\vartheta_1)$-path, setting $w:=\max(w_1,w_2)$.
We also point out that, since $P'$ is a good pre-pseudopath, 
$\cert(u_{i-1},u_i)=\DD(u_{i-1},u_i,r+w_2) \subseteq \ballH(u_{i-1},r+w)\cap\ballH(u_i,r+w)$ 
for $i=1,\dots,k-1$.
By Proposition~\ref{prop:treedisksector} there exist constants $r_0=r_0(w,\vartheta_1,\vartheta')$ and $h=h(w,\vartheta_1,\vartheta')$ 
such that, whenever $r\geq r_0$, we have

$$ \begin{array}{rcl} 
\cert(u_0,\dots,u_{k-1})
& \subseteq & \bigcup_{i=0}^{k-2} \ballH(u_i,r+w) = \bigcup_{v \in P_{u_{k-2}\setminus u_{k-1}}} \ballH(v,r+w) \\
& \subseteq & \ballH(u_{k-1},h) \cup \sect{u_{k-1}}{u_{k-2}}{\vartheta'}.  
\end{array} $$

\noindent
Let us write $s := \max\left( \distH(u_{k-1},u_k), r+w \right)$.
By Lemma~\ref{lem:bounding} (applied with $r=s$), we can assume without loss of generality 
that the constant $h$ is such that:

$$ \begin{array}{rcl} 
\cert(u_{k-1},u_k) \subseteq \ballH(u_k,s)  
 \subseteq \ballH(u_{k-1},h) \cup \sect{u_{k-1}}{u_{k}}{\vartheta'}. 
\end{array} $$

\noindent
Since $\vartheta_2>2\vartheta'$, we have that if $\angle u_{k-2}u_{k-1}u_k \geq \vartheta_2$ then
$\sect{u_{k-1}}{u_{k-2}}{\vartheta'} \cap \sect{u_{k-1}}{u_{k}}{\vartheta'} = \emptyset$.
It follows that either $\angle u_{k-2}u_{k-1}u_k < \vartheta_2$ or
$\cert(u_0,\dots,u_{k-1})\cap\cert(u_{k-1},u_k) \subseteq \ballH(u_{k-1},h)$, as claimed.
\end{proofof}

Our next observation provides an upper bound on the number of edges in a pseudopath whose certificates intersect a 
given ball. 
This will be of use to us later on several times, for instance when we want to estimate or bound the 
expected number of pseudo-edges that can act as the start of a new chunk extending a given sequence of chunks.

\begin{lemma}\label{lem:bddintersect}
For every $w_1,w_2,\vartheta > 0$ there exists $r_0=r_0(w_1,w_2,\vartheta)>0$ such that, 
for every $r\geq r_0$, every pre-chunk $u_0,u_1,\dots,u_k$ and every ball $B=\ballH(c,s)$, it holds that 

$$ |\{1\leq i\leq k : \cert(u_{i-1},u_i) \cap B \neq \emptyset\}| \leq \frac{4s}{r} + 10. $$

\end{lemma}

\begin{proof}
By Proposition~\ref{prop:distHdistT} and the remark following the definition of pre-chunk, 
there exists an $r_0, K$ such that if $r \geq r_0$ and $0\leq i,j\leq k-1$ then 

$$ \distH( u_i, u_j ) \geq |i-j| \cdot (r-K) \geq |i-j| \cdot (r/2), $$

\noindent
for all sufficiently large $r$ and all $1\leq i,j \leq k$.

For $1\leq \ell \leq k-1$, let us denote by $c_\ell^+,c_\ell^-$ the centers of the two disks whose union is 
$\cert(u_{\ell-1},u_\ell)=\DD(u_{\ell-1},u_\ell,r+w_2)$.
If $1\leq i < j \leq k-1$ then

$$ \begin{array}{rcl} \distH(c_i^+,c_j^+), \distH(c_i^+,c_j^-), \distH(c_i^-,c_j^+), \distH(c_i^-,c_j^-)
& \geq & \distH( u_i, u_{j+1} ) - 2 \left(\frac{r+w_2}{2}\right) \\
& \geq & (j+1-i) \cdot (r/2) - 2r \\ 
& = & |i-j| \cdot (r/2) - (3/2)\cdot r, 
\end{array} $$

\noindent
(the second inequality holding for $r$ sufficiently large).

Now observe $B\cap \cert(u_{\ell-1},u_\ell) \neq \emptyset$ if and only if
$\min\left(\distH( c, c_\ell^+ ),\distH(c,c_\ell^-)\right) < (r+w_2)/2 + s$.
So if $B \cap \cert(u_{i-1},u_i)\neq\emptyset$ and $B \cap \cert(u_{j-1},u_j) \neq \emptyset$ then

$$ \begin{array}{rcl}
2r+2s & \geq & 2\cdot\left(\frac{r+w_2}{2}+s\right) \\
& \geq & \min\left( \distH(c_i^+,c_j^+), \distH(c_i^+,c_j^-), \distH(c_i^-,c_j^+), \distH(c_i^-,c_j^-)\right) \\
& \geq & |i-j| \cdot (r/2) - (3/2) \cdot r, 
\end{array} $$

\noindent
the first inequality holding provided $r_0$ was chosen sufficiently large. 
Rearranging, we find

$$ |i-j| \leq \frac{4s}{r} + 7. $$

This shows that the number of $1\leq i \leq k-1$ such that $\cert(u_{i-1},u_i) \cap B \neq \emptyset$ is 
at most $\frac{4s}{r} + 8$. Including also $\cert(u_{k-1},u_k)$, we obtain $\frac{4s}{r}+9$
which is smaller than the bound in the statement of the lemma.
\end{proof}

Using the previous lemma, we can now show that, for any ball whose radius is only a constant
larger than $(r+w_2)/2$ (the radius of the two balls whose union is the certificate of a
good pseudoedge), most of the area of the ball lies outside the certificate of any given
pseudopath.

\begin{lemma}\label{lem:bddintersecteps}
For every $\eps,w_1,w_2,\vartheta > 0$ there exists $r_0=r_0(\eps,w_1,w_2,\vartheta)>0$ such that, 
for every $r\geq r_0$, every good pre-pseudopath $u_0,u_1,\dots,u_k$ and every ball $B=\ballH(c,s)$ with radius $s\geq 
(r+w_2)/2+10\ln 10 - \ln\eps$, it holds that 

$$ \areaH( B \setminus \cert(u_0,\dots,u_k) ) \geq (1-\eps) \cdot \areaH(B). $$

\end{lemma}

\begin{proof}
We let $r_0$ be a large constant, to be chosen more precisely during the proof.

By Lemma~\ref{lem:bddintersect}, we have

$$ \begin{array}{rcl} \areaH( B \cap \cert(u_0,\dots,u_k) )
& \leq & \left(\frac{8s}{r}+20\right) \cdot \areaH(\ballH(o,(r+w_2)/2)) \\
& \leq & \left(\frac{8s}{r}+20\right) \cdot \pi e^{(r+w_2)/2}. 
\end{array} $$

As $\areaH(B) = 2\pi(\cosh s - 1 ) = (1+o_s(1)) \cdot \pi e^{s}$, we have, provided we chose $r_0$ sufficiently large,
that $\areaH(B) \geq \frac12 \pi e^s$.
Hence

$$ \frac{\areaH( B \cap \cert(u_0,\dots,u_k) ) }{\areaH(B)} 
\leq \left( (16/r) s e^{-s} + 40e^{-s}\right) \cdot e^{(r+w_2)/2}. $$

\noindent
Since the RHS is decreasing in $s$ for $s\geq 1$, we have

$$ \begin{array}{rcl} 
\frac{\areaH( B \cap \cert(u_0,\dots,u_k) ) }{\areaH(B)}
& \leq & \left( 16 \frac{ (r+w_2)/2+10\ln 10 -\ln\eps}{r} + 40 \right) \cdot e^{-10\ln10+\ln\eps} \\
& \leq & \left( 56 + 16 \frac{w_2/2+10\ln 10 -\ln\eps}{r_0} \right) \cdot 10^{-10} \cdot \eps \\
& < & \eps, 
\end{array} $$

\noindent
provided we chose $r_0$ sufficiently large.
\end{proof}

Next we present another observation that will be helpful for the situation when we want to add a new pseudo-edge
to an existing pseudopath.
It tells us that one of three things must happen, each of which will translate into useful bounds once we start
estimating the number of linked sequences of chunks later on.

\begin{corollary}\label{cor:BminCeitheror}
For every $\eps,w_1,w_2,\vartheta_1,\vartheta_2 > 0$ there exists $r_0=r_0(\eps,w_1,w_2,\vartheta_1,\vartheta_2)>0$ such that, 
for every $r\geq r_0$ and every pre-chunk $u_0,u_1,\dots,u_k$ wrt.~$r,w_1,w_2,\vartheta_1$ 
we have that either

$$ \areaH( \BGab(u_{k-1},u_k) \setminus \cert(u_0,\dots,u_{k-1}) ) \geq (1-\eps) \cdot \areaH(\BGab(u_{k-1},u_k)), $$

\noindent
or 

$$ \distH( u_{k-1},u_k ) \leq r-w_1, $$

\noindent
or

$$ r-w_1 < \distH( u_{k-1},u_k ) < r+w_2 + 20\ln 10 - 2\ln\eps \text{ and } \angle u_{k-2}u_{k-1}u_k < \vartheta_2. $$

\noindent
(or more than one of the above hold).
\end{corollary}

\begin{proof}
We let $r_0$ be a large constant, to be determined in the course of the proof and we let $r\geq r_0$
be arbitrary and we let $u_0,\dots,u_k$ be an arbitrary pre-chunk wrt.~$r,w_1,w_2,\vartheta_1$.

Of course there is nothing to prove if $\distH( u_{k-1},u_k ) \leq r-w_1$.
If $\distH(u_{k-1},u_k)\geq r+w_2 + 20\ln 10 - 2\ln\eps$ then $\BGab(u_{k-1},u_k)$ has radius 
$\geq (r+w_2)/2+10\ln 10 - \ln\eps$ and we are done by Lemma~\ref{lem:bddintersecteps} -- assuming 
without loss of generality we chose $r_0$ larger that the bound provided in that lemma.

Let us thus assume $r-w_1 < \distH(u_{k-1},u_k) < r+w_2+20\ln 10 -\ln\eps =: r+w_2'$.
By Lemma~\ref{lem:certballH}, there are constants $r_0'=r_0'(w_1,w_2',\vartheta_1,\vartheta_2), h=h(w_1,w_2',\vartheta_1,\vartheta_2)$ 
such that if $r \geq r_0'$ then 
either 

$$\BGab(u_{k-1},u_k ) \cap \cert(u_0,\dots,u_{k-1}) \subseteq \cert(u_{k-1},u_k) \cap \cert(u_0,\dots,u_{k-1})
\subseteq \ballH(u_{k-1},h), $$ 

\noindent
or 

$$ \angle u_{k-2}u_{k-1}u_k < \vartheta_2. $$

\noindent
In the latter case we are clearly done. In the former case we note that 

$$ \frac{\areaH( \BGab(u_{k-1},u_k ) \cap \cert(u_0,\dots,u_{k-1}) )}{\areaH(\BGab(u_{k-1},u_k))}
\leq \frac{\cosh h - 1}{\cosh\left(\frac{r_0-w_1}{2}\right)-1} < \eps, $$

\noindent
provided $r_0\geq r_0'$ was chosen sufficiently large.
\end{proof}

For some of the remaining proofs needed for the lower bound, we'll use the following notion
which may be of independent interest. For sets $A, B \subseteq \Dee$ we let 

$$ \ahd(A,B):= \sup_{x\in B} \distH(x,A) = \sup_{x\in B} \inf_{y\in A}\distH(x,y).$$

The abbreviation $\ahd$ stands for ``asymmetric Hausdorff distance''. As the reader may already have recognized, the Hausdorff distance 
between $A, B$ equals $\max( \ahd(A,B), \ahd(B,A) )$.

The next two lemmas are a preparation for the third one, Lemma~\ref{lem:ahdeps}.
That lemma allows us to bound the area of intersection of a given ball with 
the certificates of a given (pre-)pseudopath. 
This will be of use to us when we estimate the expected number of linked sequences of chunks later on.

\begin{lemma}\label{lem:ahdprep}
For every $\eps>0$ there exists an $a=a(\eps)$ such that the following holds.
Let $B = \ballH(c_1,r_1)$ be a disk and $C = \partial\ballH(c_2,r_2)$ a hyperbolic circle 
such that $\ahd(B,C) \geq a$. Then we have $B \cap C \subseteq \sect{c_2}{c_1}{\eps}$. 
\end{lemma}

\begin{proof}
The statement is obviously true if $B\cap C = \emptyset$. Similarly it is also true if $B$ and $C$ 
intersect in a single point (the common point would line on the line joining $c_1$ and $c_2$). 
So from now on we can assume this is not the case.

Let $x_1, x_2$ denote the two intersection points of the circles $\partial B$ and $C$.
It suffices to show that, provided $\ahd(B,C) \geq a$ and $a$ is sufficiently large, 
$\alpha := \angle c_1c_2x_1 (= \angle c_1c_2x_2) \leq \eps$. 
For convenience we write $d := \distH(c_1,c_2)$. We next point out that 

$$ \ahd( B, C ) = r_2-r_1+d. $$

\noindent
(To see this we first note that by the triangle inequality 
$\distH(z,c_1) \leq \distH(z,c_2) + \distH(c_1,c_2) = r_2 + d$ for all $z \in C$.
In other words, $\distH(z,B) = \max(0,\distH(z,c_1)-r_1) \leq \max(0,r_2-r_1+d)$ for all $z \in C$.
Applying a suitable isometry if needed, we can assume without loss of generality that 
$c_2=o$ is the origin and $c_1$ lies on the negative $x$-axis. 
The point $p = (\tanh(r_2/2),0) \in C$ satisfies $\distH(z,B) = \distH(z,c_1)-r_1 = r_2-r_1+d$, as
$c_1,c_2,p$ lies on the $x$-axis which is a hyperbolic line.)

By the hyperbolic law of cosines,

\begin{align*}
e^{r_1}&\geq \cosh(r_1)\\
&=\cosh(r_2)\cosh(d)-\sinh(r_2)\sinh(d)\cos(\alpha)\\
&\geq \frac{e^{r_2+d}}{4}\cdot \left(1-\cos(\alpha)\right).
\end{align*}

Hence 

$$ \cos(\alpha)\geq 1-4e^{r_1-r_2-d} = 1-4e^{-\ahd(B,C)}.$$

The statement follows by taking $a > -\ln\left(\frac{1-\cos\eps}{4}\right)$.
\end{proof}

\begin{lemma}\label{lem:ahd}
We have 

$$ 
\lim_{a\to\infty}\inf\left\{ \frac{\areaH(B_2\backslash B_1)}{\areaH(B_2)}:  \text{$B_1, B_2 \subseteq \Dee$ (hyperbolic) disks, } 
\ahd(B_1,B_2)\geq a \right\}= 1.
$$

\end{lemma}

\begin{proof}
It suffices to show that for every $\eps>0$ there exists $a_0$ such that $\ahd(B_1,B_2) \geq a_0$
implies that $\areaH(B_2\setminus B_1) \geq (1-\eps)\cdot \areaH(B_2)$.
We thus let $a_0$ be a large constant, to be determined in the course of the proof, and we let 
$B_1 = \ballH(c_1,r_1), B_2 = \ballH(c_2,r_2)$ be arbitrary disks with $a := \ahd(B_1,B_2) \geq a_0$.

If $r_2 \leq a/2$ then $B_1$ and $B_2$ are disjoint and we are done.
For the remainder of the proof, we therefore assume $r_2 > a/2$.

Each of the circles $C_s := \partial \ballH(c_2,s)$ with $r_2-\frac{a}{1000} \leq s \leq r_2$ satisfies 
$\ahd( B_1, C_s ) \geq \left(\frac{999}{1000}\right) a$.
Applying Lemma~\ref{lem:ahdprep}, having chosen $a_0$ sufficiently large, we 
can assume that $C_s \cap B_1 \subseteq \sect{c_2}{c_1}{\frac{\eps}{1000}}$ for all $r_2-\frac{a}{1000} \leq s \leq r_2$.
It follows that 

$$ B_1 \cap B_2 \subseteq \ballH\left(c_2,r_2-\frac{a}{1000}\right) \cup \sect{c_2}{c_1}{\frac{\eps}{1000}}. $$

Next we note that, by rotational symmetry of the hyperbolic area measure

\begin{equation}\label{eq:ahdeq1} 
\frac{\areaH\left( \sect{c_2}{c_1}{\frac{\eps}{1000}} \cap B_2 \right)}{\areaH(B_2)} = \frac{\eps}{1000\cdot \pi}. 
\end{equation}

We also have

\begin{equation}\label{eq:ahdeq2} 
\begin{array}{rcl}
\displaystyle \frac{\areaH(\ballH(c_2,r_2-\frac{a}{1000})}{\areaH(c_2,r_2)} 
& = & \displaystyle \frac{\cosh(r_2-\frac{a}{1000})-1}{\cosh(r_2)-1} \\[2ex]
& \leq & \displaystyle \frac{\cosh(r_2-\frac{a}{1000})}{\cosh(r_2)} \\[2ex]
& \leq & \displaystyle \frac{\cosh(\frac{499}{1000} a)}{\cosh(a/2)} \\[2ex]
& \leq & \displaystyle \frac{\eps}{1000},
\end{array}
\end{equation}

\noindent
where the second inequality follows from the fact that $x \mapsto \cosh(x-y)/\cosh(x)$ is decreasing 
in $x$ for $x \geq y$ and the last inequality holds provided we chose $a_0$ sufficiently large, and uses
that $\cosh(\frac{499}{1000} a)/\cosh(a/2) \to 0$ 
as $a\to\infty$.

Combining~\eqref{eq:ahdeq1} and~\eqref{eq:ahdeq2}, we see that 

$$ \areaH( B_1\cap B_2 ) < \eps \cdot \areaH(B_2), $$

\noindent
provided we chose $a_0$ sufficiently large.
This clearly implies the statement of the lemma.
\end{proof}

\begin{lemma}\label{lem:ahdeps}
For every $\eps, w_1,w_2,\vartheta > 0$ there exist $r_0=r_0(\eps, w_1,w_2,\vartheta), a = a(\eps,w,\vartheta) >0$ such that, 
for every pre-chunk $u_0,\dots,u_k$ and every
disk $B$ with $\ahd( \cert(u_0,\dots,u_k), B ) \geq a$  we have 

$$ \areaH\left( B \setminus \cert(u_0,\dots,u_k) \right) \geq 
(1-\eps) \areaH(B). $$

\end{lemma}

\begin{proof}
We fix a large constant $K>0$, to be determined more precisely during the course of the proof.
For convenience we'll write $C:=\cert(u_0,\dots,u_k)$ and $B = \ballH(c,s)$.
We let $B_1=\ballH(c_1,r_1),\dots,B_N=\ballH(c_N,r_N)$ be the balls that feature in the definition of $\cert(u_0,u_1),\dots,
\cert(u_{k-1},u_k)$. That is each $B_i$ is either $\BGab(u_{k-1},u_k)$ or a disk of diameter $r+w_2$ with 
$u_{j-1}, u_j \in \partial B_i$ for some $1\leq j\leq k$.
In particular $C=B_1\cup\dots\cup B_N$ and $N\leq 2k$.

By Lemma~\ref{lem:ahd} we can take the constant $a$ such that
the assumption that $\ahd( C, B ) \geq a$ implies that $\areaH( B\setminus B_i ) \geq (1-\eps/K)\areaH( B )$ for each $i=1,\dots,N$.

We set $I := \{ i : B_i \cap B \neq \emptyset \}$. 
If $|I| \leq K$ then the statement clearly holds. For the remainder of the proof we thus assume $|I| > K$.

By Lemma~\ref{lem:bddintersect}, assuming $r_0$ was chosen sufficiently large, we have
$|I| \leq 8(s/r) + 20$ (each certificate is either a single ball or the union of two balls). 
Hence, provided we chose $r_0$ and $K$ sufficiently large, $|I| > K$ implies $s \geq (K/10) \cdot r$.
We have 

$$ \begin{array}{rcl} \areaH\left( B \setminus C \right)
& \geq & \areaH(B) - \sum_{i\in I} \areaH( B \cap B_i ) \\
& \geq & \areaH(B) - \left(\frac{\eps}{K}\right)\areaH(B) - \sum_{i\in I : B_i \neq \BGab(u_{k-1},u_k)} \areaH(B_i) \\
& \geq & (1-\eps/K) \cdot \areaH(B) - \left(8(s/r)+20\right) \cdot \areaH\left( \ballH(o,(r+w_2)/2) \right) \\
& \geq & (1-\eps/K) \cdot \areaH(B) - s \cdot \areaH\left( \ballH(o,(r+w_2)/2) \right), 
\end{array} $$

\noindent
the last inequality holding assuming we have chosen $r_0$ and $K$ sufficiently large.
Next we point out that 

$$ \begin{array}{rcl} \frac{s \cdot \area\left( \ballH(o,(r+w_2)/2) \right)}{\areaH( B )} 
& = & \frac{s \cdot (\cosh((r+w_2)/2)-1)}{\cosh s - 1} \\
& \leq & 1000 \cdot e^{(r+w_2)/2} \cdot s \cdot e^{-s} \\
& \leq & 1000 \cdot e^{(r+w_2)/2} \cdot (K/10) \cdot r \cdot e^{-(K/10) r} \\
& = & 1000 \cdot e^{w_2/2} \cdot (K/10) \cdot r \cdot e^{-\left(\frac{K-5}{10}\right) r} \\
& \leq & 1000 \cdot e^{w_2/2} \cdot (K/10) \cdot r_0 \cdot e^{-\left(\frac{K-5}{10}\right) r_0} \\
& \leq & \eps/2, 
\end{array} $$

\noindent
provided $r_0$ and $K$ were chosen sufficiently large, using in the second line that 
$\cosh x - 1 = (1+o_x(1)) \cdot \frac12 \cdot e^x$ as $x \to \infty$; in the third line that $s\mapsto se^{-s}$ 
is decreasing for $s>1$;
and in the last line that $r \cdot e^{-cr} \to 0$ as $r\to\infty$ (for every $c>0$).
It follows that if $|I|>K$ then also 

$$ \areaH\left( B \setminus C \right) > (1-\eps) \cdot \areaH( B ), $$

\noindent
as claimed.
\end{proof}

We conclude this subsection, with two lemmas giving conditions under which hyperbolic angles are small.
Again, we will use these when estimating the expected number of linked sequences of chunks later on.

\begin{lemma}\label{lem:angle}
There is a universal constant $K>0$ such that the following holds.
If $B_1 = \ballH(c_1,r_1), B_2 = \ballH(c_2,r_2)$ and $p \in \partial B_2$ are such that
$B_1\cap B_2 \neq \emptyset$ then 
either $\distH(p,c_1) < r_1+K$ or $\angle c_1pc_2 <  10 \exp\left[ (r_1 - \distH(p,c_1))/2 \right]$.
\end{lemma}

\begin{proof}
Let $K = K(\pi/2)$ be as provided by Lemma~\ref{lem:cosines}. 
For notational convenience we write $\alpha := \angle c_1pc_2, \rho := \distH(c_1,p)$ and $d := \distH(c_1,c_2)$.

The disks $B_1$ and $B_2$ intersect if and only if $d < r_1 + r_2$.
If $\alpha \geq \pi/2$ then Lemma~\ref{lem:cosines} tells us that
$d \geq \rho + r_2 - K$.
It follows that if $\alpha \geq \pi/2$ and $B_1\cap B_2 \neq \emptyset$ then 
$\rho < r_1+K$. 

Let us then assume $\alpha < \pi/2$.
If $B_1\cap B_2 \neq \emptyset$ then the hyperbolic cosine rule gives

$$ \begin{array}{rcl} e^{r_1+r_2} & > & \cosh d \\ 
& = & \cosh(\rho) \cosh(r_2) - \cos(\alpha)\sinh(\rho)\sinh(r_2) \\
& \geq & \frac{1}{4} e^{\rho+r_2} (1-\cos\alpha) \\
& \geq & \frac{1}{4\pi} e^{\rho+r_2} \alpha^2, \end{array} $$

\noindent
where we have used that $1-\cos(\alpha)\geq \frac{\alpha^2}{\pi}$ for all $0<\alpha<\pi/2$.
Rearranging and taking square roots gives $\alpha < \sqrt{4\pi} \cdot e^{(r_1-\rho)/2} < 10 e^{(r_1-\rho)/2}$.
\end{proof}

We'll use another rather straightforward consequence of the hyperbolic cosine rule:

\begin{lemma}\label{lem:distHvsangle}
For all $x_1,x_2 \in \Dee$ we have 

$$ 2\pi \cdot \exp\left[ \frac12\cdot \left( \distH(x_1,x_2)
- \distH(o,x_1)-\distH(o,x_2) \right)\right] > \angle x_1ox_2. $$

\end{lemma}

\begin{proof}
For notational convenience, write $\rho_1 = \distH(o,x_1), \rho_2 = \distH(o,x_2)$ and $\gamma = \angle x_1ox_2$.
It suffices to show that $e^{\distH(x_1,x_2)} > \frac{1}{4\pi^2} \cdot e^{\rho_1+\rho_2} \cdot \gamma^2$.

If $\gamma < \pi/2$ then the hyperbolic cosine rule 
gives that 

$$ \begin{array}{rcl} 
e^{\distH(x_1,x_2)} & > & \cosh(\distH(x_1,x_2))  \\
& = & \cosh(\rho_1)\cosh(\rho_2)-\cos(\gamma)\sinh(\rho_1)\sinh(\rho_2) \\
& \geq & \frac14 \cdot e^{\rho_1+\rho_2} \cdot (1-\cos\gamma) \\
& \geq &  \frac{1}{4\pi} \cdot e^{\rho_1+\rho_2} \cdot \gamma^2.
\end{array} $$

\noindent
(Using that $1-\cos(\gamma)\geq \frac{\alpha^2}{\pi}$ for $0<\gamma<\pi/2$ in the third line.)

If on the other hand $\gamma \in [\pi/2,\pi]$ then similarly

$$ \begin{array}{rcl} 
e^{\distH(x_1,x_2)} 
& > & \cosh(\rho_1)\cosh(\rho_2)-\cos(\gamma)\sinh(\rho_1)\sinh(\rho_2) \\
& \geq & \cosh(\rho_1)\cosh(\rho_2) \\
& \geq & \frac14 \cdot e^{\rho_1+\rho_2} \\
& \geq &  \frac{1}{4\pi^2} \cdot e^{\rho_1+\rho_2} \cdot \gamma^2.
\end{array} $$

\end{proof}

\subsubsection{Counting good pseudopaths}

The next point of order is to bound the expected number of black, good pseudopaths starting from the origin.
We plan to use an inductive approach and begin with a couple of lemmas designed for the situation 
where we add a good pseudo-edge to an existing, good pseudo-path.

\begin{lemma}\label{lem:YIa}
For every $w_1,w_2,\vartheta>0$ and $0<\eps<1$ there exist $\lambda_0=\lambda_0(\eps,w_1,w_2,\vartheta)$ such that
the following holds for all $0<\lambda<\lambda_0$ and all $p \leq (1-\eps)\cdot(\pi/3)\cdot \lambda$,  
setting $r := 2\log(1/\lambda)$.

Let $u_0,\dots,u_k$ be an arbitrary good pre-pseudopath and, writing
$C := \cert(u_0,\dots,u_k)$, define

$$ Y_{\text{\bf{I-i}}} :=
 \left|\left\{ z \in \Zcal_b : \begin{array}{l} 
 r-w_1 \leq \distH(u_k,z) \leq r+w_2, \text{ and;} \\
 \angle u_{k-1}u_kz > \vartheta, \text{ and;} \\
 \text{$\exists$ a disk $B$ with } \Zcal \cap B \setminus C = \emptyset, \\
 u_k,z\in\partial B \text{ and } \diam_{\Haa^2}(B) < r+w_2.
 \end{array} \right\}\right|. $$

We have

$$ \Ee Y_{\text{\bf{I-i}}} \leq (1-\eps/2). $$

\end{lemma}

Before proceeding with the proof, let us clarify that in the above lemma we also allow $k=0$. In that case the demand
$\angle u_{k-1}u_kz > \vartheta$ of course does not apply.

\begin{proof}
We let $\lambda_0$ be a small constant
to be determined in the course of the proof, and take an arbitrary $\lambda < \lambda_0$ and $p \leq (1-\eps)\cdot(\pi/3)\cdot\lambda$. 
Once again we apply Corollary~\ref{cor:SlivMeck2col} to find that

$$ \Ee Y_{\text{\bf{I-i}}} = p\lambda \int_{\Dee} \Ee\left[g(z,\Zcal\cup\{z\})\right] f(z) \dd z, $$

\noindent
where $f$ is again as given by~\eqref{eq:fdef} and this time

$$ g(u,\Ucal) := 
1_{\left\{\begin{array}{l}  
 r-w_1\leq \distH(u_k,u) \leq r+w_2, \text{ and;} \\
 \angle u_{k-1}u_ku > \vartheta, \text{ and;} \\
 \text{$\exists$ a disk $B$ with } \Ucal \cap B \setminus C = \emptyset, \\
 u_k,u\in\partial B \text{ and } \diamH(B) < r+w_2.
 \end{array}
\right\}}. $$

By Corollary~\ref{cor:certballH}, assuming $\lambda_0$ was chosen sufficiently small, 
if $z \in\Dee$ is such that $\Ee\left[g(z,\Zcal\cup\{z\})\right] \neq 0$ then
$C \cap \cert(u_k,z) \subseteq \ballH(u_k,h)$.

Applying a suitable isometry if needed (which leaves the distribution of $\Zcal$ invariant), we can 
assume without loss of generality that $u_k=o$ is the origin.
Recall that any disk $B$ satisfying $o,z\in\partial B$ and 
$\diamH(B) < r+w$ satisfies $B \subseteq \DD(o,z,r+w)$.
So if $z \in \Dee$ is such that $r-w_1\leq \distH(0,z) \leq r+w_2$ and $\angle u_{k-1}oz > \vartheta$ then 

$$ \begin{array}{rcl} 
\Ee\left[g(z,\Zcal\cup\{z\})\right]
& = & 
\Pee\left( \begin{array}{l} \text{$\exists$ a disk $B$ with } \Zcal \cap B  = \emptyset, \\
 o,z\in\partial B \text{ and } \diam_{\Haa^2}(B) \leq r+w_2 \end{array} {\Big|} \Zcal \cap C = \emptyset \right) \\[2ex]
& = & 
\Pee\left( \begin{array}{l} \text{$\exists$ a disk $B$ with } \Zcal \cap B  = \emptyset, \\
 o,z\in\partial B \text{ and } \diam_{\Haa^2}(B) \leq r+w_2 \end{array} {\Big|} 
 \Zcal \cap C \cap \cert(o,z) = \emptyset \right) \\[2ex]
& \leq & \displaystyle 
\frac{
\Pee\left( \text{$\exists$ a disk $B$ with } \Zcal \cap B  = \emptyset \text{ and } o,z\in\partial B  \right)  
}{ 
\Pee\left( \Zcal \cap C \cap \cert(o,z) = \emptyset \right)
} \\[2ex]
& \leq & \displaystyle 
\frac{ 
\Pee\left( \text{$\exists$ a disk $B$ with } \Zcal \cap B  = \emptyset \text{ and } o,z\in\partial B \right) 
}{
\Pee\left( \Zcal \cap \ballH(o,h) = \emptyset \right)
} \\[2ex]
& = &  \displaystyle
\frac{ 
\Ee\left[\hat{g}(z,\Zcal\cup\{z\})\right]
}{
\Pee( \Zcal \cap\ballH(o,h) = \emptyset )
}, 
\end{array} $$

\noindent
where $\hat{g}(u,\Ucal) = 1$ if $ou$ is an edge of the Delaunay graph for $\Ucal \cup \{o\}$ and $\hat{g}(u,\Ucal)=0$ otherwise.
It follows that

$$ \begin{array}{rcl} 
\Ee Y_{\text{\bf{I-i}}} 
& \leq & \displaystyle 
\frac{p\lambda}{\Pee( \Zcal \cap\ballH(o,h) = \emptyset )} \cdot
\int_\Dee \Ee\left[\hat{g}(z,\Zcal\cup\{z\})\right] f(z) \dd z \\[2ex]
& = & \displaystyle 
\frac{p}{\Pee( \Zcal \cap\ballH(o,h) = \emptyset )} \cdot \left( 6 + \frac{3}{\pi\lambda} \right) \\[2ex]
& \leq & \displaystyle 
\frac{(1-\eps)\cdot (1+2\pi\lambda)}{1-\lambda\cdot\areaH(\ballH(o,h))},
\end{array} $$

\noindent
where the second line follows by Isokawa's formula (and the Slivniak-Mecke formula to phrase the
typical degree as an integral). 

Since $h$ is a constant we have $\lambda\cdot\areaH(\ballH(o,h)) \to 0$ as $\lambda\searrow 0$.
So, provided we chose $\lambda_0$ sufficiently small, we have 

$$\Ee Y_{\text{\bf{I-i}}} 
< 1 - \eps/2, $$

\noindent 
as claimed.
\end{proof}

\begin{lemma}\label{lem:YIb}
For every $w_1,w_2,\vartheta>0$ there exists a $\lambda_0=\lambda_0(w_1,w_2,\vartheta)$ such that
the following holds for all $\lambda<\lambda_0$ and all $p \leq 10\lambda$,  
setting $r := 2\log(1/\lambda)$.

Let $u_0,\dots,u_k$ be an arbitrary good pre-pseudopath and, writing
$C := \cert(u_0,\dots,u_k)$, define

$$ Y_{\text{\bf{I-ii}}} :=
 \left|\left\{ z \in \Zcal_b : \begin{array}{l} 
 r-w_1 < \distH(u_k,z) < r+w_2, \text{ and;} \\
 \angle u_{k-1}u_kz >\vartheta , \text{ and;} \\
 \text{either } \Zcal \cap \DD^-(u_k,z,r+w) \setminus C = \emptyset, \\
 \text{or } \Zcal \cap \DD^+(u_k,z,r+w) \setminus C = \emptyset, \\
 \text{(or both).}
 \end{array} \right\}\right|. $$

We have

$$ \Ee Y_{\text{\bf{I-ii}}} \leq 10^4 \cdot e^{w_2} \cdot e^{-e^{w_2/2}}. $$

\end{lemma}

\begin{proof}
The proof is nearly identical to the proof of the previous lemma.
Reasoning as the proof of Lemma~\ref{lem:YIa}, there is a constant $h=h(w_1,w_2,\vartheta)$ such that 

$$ \Ee Y_{\text{\bf{I-ii}}} \leq \frac{\Ee\tilde{X}_{\bf{VI}} }{ 1 - \lambda \cdot \areaH( \ballH(o,h) ) }, $$

\noindent 
for all sufficiently small $\lambda$, where 
$\tilde{X}_{\bf{VI}}$ is as in Lemma~\ref{lem:XVItil}.

Since $1-\lambda \cdot \areaH( \ballH(o,h) ) > 1/10$ provided we chose $\lambda_0$ sufficiently small,
the result follows from Lemma~\ref{lem:XVItil}
\end{proof}

These last two lemmas allow us to prove the following statement using an inductive approach.
Here and in the rest of the paper we write

$$ G_k := \left|\left\{\text{black, good pseudopaths of length $k$ starting from $o$}\right\}\right|. $$

\begin{lemma}\label{lem:numbergoodpspths}
For every $0<\eps<1$ and $w_1,w_2,\vartheta>0$ there exists a 
$\lambda_0 = \lambda_0(\eps,w_1,w_2,\vartheta)$ such that the following holds 
for all $0<\lambda<\lambda_0$ and all $p\leq (1-\eps)\cdot(\pi/3)\cdot\lambda$, setting 
$r := 2\ln(1/\lambda)$.

$$ \Ee G_k \leq \left( 1 - \eps/2 + 10^4 e^{w_2} e^{-e^{w_2/2}} \right)^k. $$

\end{lemma}

\begin{proof}
By Corollary~\ref{cor:SlivMeck2col} we have 

\begin{equation}\label{eq:EGnCampbell} 
\Ee G_k = (p\lambda)^n\int_\Dee\dots\int_\Dee \Ee\left[g_k(z_1,\dots,z_k;\Zcal\cup\{z_1,\dots,z_k\})\right]
f(z_1)\dots f(z_k)\dd z_k\dots\dd z_1, 
\end{equation}

\noindent
where $f$ is as given by~\eqref{eq:fdef} and $g_k$ is the indicator function of the event that $o,z_1,\dots,z_k$
form a good pseudopath (wrt.~the parameters $r,w_1,w_2,\vartheta$ and the point set 
$\Zcal \cup \{o,z_1,\dots,z_k\}$).
We note that 

\begin{equation}\label{eq:gnineq} 
\begin{array}{rcl}
g_k(z_1,\dots,z_k;\Zcal\cup\{z_1,\dots,z_k\}) 
& \leq & 
g_{k-1}(z_1,\dots,z_{k-1};\Zcal\cup\{z_1,\dots,z_{k-1}\}) \cdot \\
& & g_{Y_{\text{\bf{I}}}}(z_1,\dots,z_k,\Zcal), 
\end{array}
\end{equation}

\noindent
where 

$$
g_{Y_{\text{\bf{I}}}}(z_1,\dots,z_k,\Zcal) =  1_{\left\{\begin{array}{l} 
r-w_1<\distH(z_{k-1},z_k)<r+w_2, \text{ and;}\\
\angle z_{n-2}z_{k-1}z_k > \vartheta, \text{ and; }\\
\text{either $\exists$ a disk $B$ with $z_{k-1},z_k\in\partial B$,$\diamH(B) < r+w_2,$} \\
\text{ \hspace{5ex} $\Zcal \cap B \setminus \cert(o,z_1,\dots,z_{k-1}) =\emptyset$, } \\
\text{ \hspace{2ex} or $\Zcal \cap \DD^-(z_{k-1},z_k,r+w_2) \setminus \cert(o,z_1,\dots,z_{k-1}) =\emptyset$, } \\
\text{ \hspace{2ex} or $\Zcal \cap \DD^+(z_{k-1},z_k,r+w_2) \setminus \cert(o,z_1,\dots,z_{k-1}) =\emptyset$, } \\
\text{(or several of the above three hold.)}
\end{array} \right\}}. $$

\noindent
Next, we note that if $o,z_1,\dots,z_k$ is not a good pre-pseudopath
then $g_k(z_1,\dots,z_k,\Zcal\cup\{z_1,\dots,z_k\}) = 0$ and whenever $o,z_1,\dots,z_k$ is a good pre-pseudopath then 

\begin{equation}\label{eq:Egn} 
\begin{array}{c} 
\Ee\left[ g_k(z_1,\dots,z_{k-1};\Zcal\cup\{z_1,\dots,z_{k-1}\}) \cdot g_{Y_{\text{\bf{I}}}}(z_1,\dots,z_k,\Zcal) \right] \\
= \\
\Ee\left[ g_{k-1}(z_1,\dots,z_{k-1};\Zcal\cup\{z_1,\dots,z_{k-1}\}) \right] \cdot 
\Ee\left[ g_{Y_{\text{\bf{I}}}}(z_1,\dots,z_k,\Zcal)\right], 
\end{array} 
\end{equation}

\noindent
as $\Zcal\cap\cert(o,z_1,\dots,z_{k-1})$ and $\Zcal \setminus \cert(o,z_1,\dots,z_{k-1})$ are independent;
and the former determines the value of $g_k(z_1,\dots,z_{k-1};\Zcal\cup\{z_1,\dots,z_{k-1}\})$, while the 
latter determines $g_{Y_{\text{\bf{I}}}}(z_1,\dots,z_k,\Zcal)$.

Next, notice that by Corollary~\ref{cor:SlivMeck2col} and Lemmas~\ref{lem:YIa} and~\ref{lem:YIb}, we have 

\begin{equation}\label{eq:innerint} 
p\lambda \int_\Dee \Ee\left[ g_{Y_{\text{\bf{I}}}}(z_1,\dots,z_k,\Zcal)\right] f(z_k) \dd z_k
\leq 1-\eps/2 +10^4 e^{w_2} e^{-e^{w_2/2}}, 
\end{equation}

\noindent
for every $z_1,\dots,z_{k-1}$ such that $o,z_1,\dots,z_{k-1}$ is a good pre-pseudopath.
If $o,z_1,\dots,z_{k-1}$ is not a good pre-pseudopath then
$g_{k-1}(z_1,\dots,z_{k-1};\Zcal\cup\{z_1,\dots,z_{k-1}\})$ $=$ \\
$g_{k}(z_1,\dots,z_{n};\Zcal\cup\{z_1,\dots,z_{n}\})$ $= 0$.
Combining~\eqref{eq:EGnCampbell}--\eqref{eq:innerint}, we see that 

$$ \begin{array}{rcl} \Ee G_k
& \leq & \displaystyle 
(p\lambda)^n\int_\Dee\dots\int_\Dee \Ee\left[g_{k-1}(z_1,\dots,z_{k-1};\Zcal\cup\{z_1,\dots,z_{k-1}\})\right] \cdot \\
& & \hspace{15ex} \Ee\left[ g_{Y_{\text{\bf{I}}}}(z_1,\dots,z_k,\Zcal)\right]
f(z_1)\dots f(z_k)\dd z_k\dots\dd z_1 \\
& = & \displaystyle 
(p\lambda)^{n-1}\int_\Dee\dots\int_\Dee \Ee\left[g_{k-1}(z_1,\dots,z_{k-1};\Zcal\cup\{z_1,\dots,z_{k-1}\})\right] \cdot \\
& & \displaystyle \hspace{15ex} 
\left( p\lambda \int_\Dee \Ee\left[ g_{Y_{\text{\bf{I}}}}(z_1,\dots,z_k,\Zcal)\right] f(z_k)\dd z_k \right)
f(z_1)\dots f(z_{k-1})\dd z_{k-1}\dots\dd z_1 \\
& \leq & \displaystyle 
\left( 1-\eps/2 +10^4 e^{w_2} e^{-e^{w_2/2}} \right) \cdot \\
& & \displaystyle 
(p\lambda)^{n-1} \cdot 
\int_\Dee\dots\int_\Dee \Ee\left[g_{k-1}(z_1,\dots,z_{k-1};\Zcal\cup\{z_1,\dots,z_{k-1}\})\right] \cdot \\
& & \displaystyle \hspace{17ex} 
f(z_1)\dots f(z_{k-1})\dd z_{k-1}\dots\dd z_1 \\[2ex]
& = & \displaystyle
\left( 1-\eps/2 +10^4 e^{w_2} e^{-e^{w_2/2}} \right) \cdot \Ee G_{k-1}. 
\end{array} $$

%
%

\noindent
The lemma follows by iterating the recursive inequality.
\end{proof}

The previous lemma immediately gives that, provided we chose $w_2$ a sufficiently large constant (how large 
we need to take it depends on $\eps$), 
for all sufficiently small intensities $\lambda$,
almost surely, there are no infinite, black, good pseudopaths starting from the origin.
In fact something slightly stronger holds:

\begin{corollary}\label{cor:noinfgoodpaths}
For every $0<\eps<1$ there exists a $c=c(\eps)$ such that
for all $w_1,\vartheta>0$ and $w_2>c$ the following holds.

There exists a $\lambda_0=\lambda_0(\eps,w_1,w_2,\vartheta)$
such that, when $0<\lambda<\lambda_0$ and $p\leq (1-\eps)\cdot(\pi/3)\cdot\lambda$, 
setting $r:=2\ln(1/\lambda)$, almost surely there are no infinite, black, good pseudopaths.
\end{corollary}

\begin{proof}
As pointed out earlier, we can choose $c=c(\eps)$ such that for all $w_2>c$ and $w_1,\vartheta>0$ there
is a $\lambda_0$ such that, whenever when $0<\lambda<\lambda_0$ and $p\leq (1-\eps)\cdot(\pi/3)\cdot\lambda$, 
almost surely there is no infinite, black, good pseudopath starting from the origin.

Let $x>0$ be arbitrary and write 

$$ N_x:= \left|\left\{ z \in \Zcal_b\cap\ballH(o,x) : 
\begin{array}{l} \text{there is an infinite, black, good} \\ 
\text{pseudopath starting from $z$ }\end{array} \right\}\right|. $$

\noindent
By Corollary~\ref{cor:SlivMeck2col}

$$ \Ee N_x =p\lambda \int_\Dee \Ee\left[g(z,\Zcal\cup\{z\})\right] f(z)\dd z, $$

\noindent
where $f$ is given by~\eqref{eq:fdef} as usual and 
$g(u,\Ucal)=1$ if $u \in \ballH(o,x)$ and $u \in \Ucal$ and there is an infinite, black, good, pseudopath 
(wrt.~$r,w_1,w_2,\vartheta$ and the point set $\Ucal$) starting from $u$;
and $g(u,\Ucal) = 0$ otherwise.
If $z \not\in \ballH(o,x)$ then $g(z,\Zcal\cup\{z\})=0$ by definition of $g$, and if $z \in \ballH(o,x)$ then 

$$ \Ee\left[ g(z,\Zcal\cup\{z\})\right] = \Ee\left[ g(o,\Zcal\cup\{o\}) \right] = 0, $$

\noindent
since the distribution of $\Zcal$ is invariant under the action of isometries of the Poincar\'e disk
($\Zcal \isd \varphi[\Zcal]$ if $\varphi$ is an isometry), and in particular under 
an isometry that maps $z$ to $o$; and we know that the origin a.s.~is not on any infinite, back, good, pseudopath as 
remarked just before the statement of the lemma.
It follows that $\Ee N_x = 0$ and hence also $N_x=0$ a.s., for each $x>0$ separately.
Thus

$$ \Pee\left(\text{$\exists$ infinite, black, good pseudopath}\right)
= \Pee\left( \bigcup_{x>0} \{N_x \neq 0\} \right) 
= \lim_{x\to\infty} \Pee( N_x \neq 0 ) = 0, $$

\noindent
as $\{N_x \neq 0 \} \subseteq \{N_y \neq 0\}$ whenever $x<y$.
\end{proof}

Recall that our strategy for the proof that $p_c \geq (1-\eps)\cdot(\pi/3)\cdot \lambda$ for small $\lambda$ is to prove 
that almost surely no black, infinite, good pseudopath and no black, infinite, linked sequence of chunks starting from the origin will 
exist (for an appropriate choice
of $w_1,w_2,\vartheta>0$ and $r=2\ln(1/\lambda)$).
In the light of the last corollary, we now turn our attention to linked sequences of chunks.

\subsubsection{Counting chunks starting from the origin}

Next, we bound the number of black chunks starting from the origin. 
We will exploit our results on good pseudopaths, using that a chunk is just a 
good pseudopath with a single bad pseudo-edge stuck to the end.
We begin with a few lemmas designed for adding a last, bad edge to an existing good pseudopath.
The names of the random variables described by the lemmas are $L$ with some subscript, where
$L$ stands for ``last'' and the subscript corresponds to the type of edge under consideration.

\begin{lemma}\label{lem:LIVa}
For all $w_1,w_2,\vartheta>0$ there exist $\lambda_0=\lambda_0(w_1,w_2,\vartheta)>0$ such that
the following holds for all $0<\lambda<\lambda_0$ and all $p \leq 10 \lambda$,  
setting $r := 2\log(1/\lambda)$.

Let $u_0,\dots,u_k \in \Dee$ be arbitrary, and set $C := \cert(u_0,\dots, u_k)$ and

$$ L_{\text{\bf{IV-i}},\geq d}  := 
 \left|\left\{ z \in \Zcal_b : \begin{array}{l} 
 \distH(u_k,z) \geq d, \text{ and;} \\
 \areaH( \BGab(u_k,z) \cap C ) \leq \frac{1}{10}\cdot \areaH( \BGab(u_k,z) ), \text{ and;} \\
 \Zcal \cap \BGab^-(u_k,z) \setminus C = \emptyset
 \text{ or } \Zcal \cap \BGab^+(u_k,z) \setminus C = \emptyset.
 \end{array} \right\}\right|. $$

We have

$$ \Ee L_{\text{\bf{IV-i}},\geq r+v} \leq 1000 \cdot e^{v/2} \cdot e^{-e^{v/2}}, $$

\noindent
for all $v \geq 0$.
\end{lemma}

\begin{proof}
We let $\lambda_0>0$ be a small constant to be determined in the course of the proof, and take an 
arbitrary $\lambda < \lambda_0$ and $p \leq 10\lambda$. 
Without loss of generality (applying a suitable isometry if needed) we can assume $u_k=o$ is the origin.

If $z$ satisfies $\distH(o,z) \geq r$ and $\areaH( \BGab(o,z) \cap C ) \leq \frac{1}{10}\cdot\areaH( \BGab(o,z) )$ 
then of course also

$$
\areaH( \BGab^-(o,z) \setminus C) \geq \left(\frac12-\frac{1}{10}\right)\cdot\areaH( \BGab(o,z) )
\geq e^{\distH(o,z)/2},
$$

\noindent
using that 

$$ \areaH( \BGab(o,z) ) = 2\pi\left( \cosh\left( \distH(o,z)/2 \right)-1\right) = 
(1+o_{\lambda_0}(1)) \pi e^{\distH(o,z)/2}, $$

\noindent
(the $o_{\lambda_0}(1)$ term referring to the situation where $\lambda_0$ tends to zero)  and
assuming that we have chosen $\lambda_0$ sufficiently small (so that $\distH(o,z)\geq r \geq 2\ln(1/\lambda_0)$ is large). 
Analogously we also have 

$$\areaH( \BGab^+(o,z) \setminus C) \geq e^{\distH(o,z)/2}, $$

\noindent
for all such $z$.
By Corollary~\ref{cor:SlivMeck2col} and this last observation:

$$ \begin{array}{rcl} 
\Ee L_{\text{\bf{IV-i}},\geq r+v} 
& = & \displaystyle 
p\lambda \int_{\Dee\setminus\ballH(o,r+v)} 
1_{\{\areaH( \BGab(o,z) \cap C ) \leq \frac{1}{10}\cdot\areaH( \BGab(o,z) )\}} \cdot \\[2ex]
& & \displaystyle \hspace{10ex} 
\Pee\left( \Zcal \cap \BGab^-(o,z) \setminus C = \emptyset
 \text{ or } \Zcal \cap \BGab^+(o,z) \setminus C = \emptyset\right) f(z) \dd z \\[2ex]
 & \leq & \displaystyle 
 p\lambda \int_{\Dee\setminus\ballH(o,r+v)} 2 e^{-\lambda e^{\distH(o,z)/2}} f(z) \dd z,
 \end{array} $$

\noindent
with $f$ given by~\eqref{eq:fdef}.

Applying a change to hyperbolic polar coordinates to $z$ 
we find

$$ \begin{array}{rcl} 
\Ee L_{\text{\bf{IV-i}},\geq r+v} 
& \leq & \displaystyle
2p\lambda \int_{r+v}^{\infty} \int_0^{2\pi} e^{-\lambda e^{\rho/2}} 
\sinh(\rho) \dd\alpha\dd\rho \\
& \leq & \displaystyle 
200 \lambda^2 \int_{r+v}^\infty e^{-\lambda e^{\rho/2}} e^\rho \dd\rho. 
\end{array} $$

Applying the substitution $t = \lambda e^{\rho/2}$ (so that $\dd\rho = \frac{2\dd t}{t}$), we find 

$$ \begin{array}{rcl} 
\Ee L_{\text{\bf{IV-i}},\geq r+v} 
& \leq & \displaystyle 
200 \lambda^2 \int_{\lambda e^{(r+v)/2}}^\infty e^{-t}\cdot\left(\frac{t}{\lambda}\right)^2 \frac{2\dd t}{t} \\[2ex]
& = & \displaystyle 
400 \int_{e^{v/2}}^\infty te^{-t} \dd t \\[2ex]
& = & \displaystyle 
400 (e^{v/2}+1) e^{-e^{v/2}} \\[2ex]
& \leq & \displaystyle 
1000 \cdot e^{v/2}\cdot e^{-e^{v/2}}. 
\end{array} $$

\end{proof}

\begin{corollary}\label{cor:YIV}
For all $w_1,w_2,\vartheta>0$ there exist $\lambda_0=\lambda_0(w_1,w_2,\vartheta)>0$ such that
the following holds for all $0<\lambda<\lambda_0$ and all $p \leq 10 \lambda$,  
setting $r := 2\log(1/\lambda)$.

Let $u_0,\dots,u_k \in \Dee$ be an arbitrary good pre-pseudopath, and set 
$C := \cert(u_0,\dots,u_k)$ and

$$ L_{\text{\bf{IV}},\geq d}  := 
 \left|\left\{ z \in \Zcal_b : \begin{array}{l} 
 \distH(u_k,z) \geq d, \text{ and; } \\
 \Zcal \cap \BGab^-(u_k,z) \setminus C = \emptyset \text{ or } \Zcal \cap \BGab^+(u_k,z) \setminus C = \emptyset.
 \end{array} \right\}\right|. $$

We have

$$ \Ee L_{\text{\bf{IV}},\geq r+v} \leq 10^4 \cdot e^{v/2} \cdot e^{-e^{v/2}}, $$

\noindent
for all $v \geq 0$.
\end{corollary}

\begin{proof}
We let $\lambda_0, \vartheta'>0$ be small constants to be determined in the course of the proof, and take an 
arbitrary $0<\lambda<\lambda_0$ and $p \leq 10\lambda$. 
Without loss of generality (applying a suitable isometry if needed) we can assume $u_k=o$ is the origin.

Since $v\geq 0$, by Corollary~\ref{cor:BminCeitheror} with $\eps := \frac{1}{10}$ and $\vartheta_1=\vartheta, 
\vartheta_2=\vartheta'$, assuming we have chosen $\lambda_0$ sufficiently small, 
we have 

$$ L_{\text{\bf{IV}},\geq r+v} \leq L_{\text{\bf{IV-i}},\geq r+v} + L_{\text{\bf{IV-ii}}}, $$

\noindent
where $L_{\text{\bf{IV-i}},\geq r+v}$ is as defined in Lemma~\ref{lem:LIVa} above and

$$ L_{\text{\bf{IV-ii}}} := \left|\Zcal_b \cap \ballH(o,r+w) \cap \sect{o}{u_{k-1}}{\vartheta'}\right|. $$

\noindent
setting $w:=w_2+22\ln 10$.

In particular $L_{\text{\bf{IV}},\geq r+v} = L_{\text{\bf{IV-i}},\geq r+v}$ if $v\geq w$. In this case
we are clearly done by Lemma~\ref{lem:LIVa}.

To prove it also for $0\leq v < w$, we also need to bound $\Ee L_{\text{\bf{IV-ii}}}$.
Lemma~\ref{lem:XII} shows that 

$$ \Ee L_{\text{\bf{IV-ii}}} \leq 1000 \vartheta' e^w \leq 1000 \cdot e^{v/2}\cdot e^{-e^{v/2}},$$

\noindent
having chosen $\vartheta'$ sufficiently small for the second inequality.
(To be precise, having chosen $\displaystyle \vartheta_2 < e^{-w} \cdot \min_{0\leq x \leq w} e^{x/2}\cdot e^{-e^{x/2}}$.)

Adding the bounds on $\Ee L_{\text{\bf{IV-i}},\geq r+v}$ and $\Ee L_{\text{\bf{IV-ii}}}$ proves the result for $0\leq v \leq w$.
\end{proof}

We will denote

$$ C_k^{\text{\bf{II}}} := \left|\left\{\text{black chunks of length $k$, starting from $o$, with final pseudo-edge of 
type {\bf{II}}}\right\}\right|, $$%
$$ C_k^{\text{\bf{III}}} := \left|\left\{\text{black chunks of length $k$, starting from $o$, with final pseudo-edge of 
type {\bf{III}}}\right\}\right|, $$

\noindent
and, for all $v \geq w_2$:%
$$ C_k^{\text{\bf{IV}},v} := \left|\left\{\begin{array}{l}
\text{black chunks of length $k$, starting from $o$, with final pseudo-edge of type {\bf{IV}}}, \\ 
\text{ and the final pseudo-edge having length $\in[r+v,r+v+1)$}
\end{array}\right\}\right|. $$
%

Recall that $G_k$ denotes the number of black, good, pseudopaths starting from the origin of length $k$.

\begin{corollary}\label{cor:Ck}
 For every $w_1,w_2,\vartheta$ there exists a $\lambda_0=\lambda_0(w_1,w_2,\vartheta)>0$ such that the 
 following holds for all $0<\lambda<\lambda_0$ and $p\leq 10\lambda$, setting $r:=2\ln(1/\lambda)$.
%
 
 We have
 
 \begin{equation}\label{eq:CnII} 
 \Ee C_{k}^{\text{\bf{II}}} \leq   10^3 \cdot e^{-w_1} \cdot \Ee G_{k-1},
 \end{equation}
 
 \begin{equation}\label{eq:CnIII} 
 \Ee C_{k}^{\text{\bf{III}}} \leq 10^3 \cdot \vartheta \cdot e^{w_2} \cdot \Ee G_{k-1}, 
 \end{equation}
 
 \noindent
 and 
 
 \begin{equation}\label{eq:CnIV}  
 \Ee C_{k}^{\text{\bf{IV}},v} \leq 10^4 \cdot e^{v/2} \cdot e^{-e^{v/2}} \cdot \Ee G_{k-1},
 \end{equation} 
 
 \noindent
 for all $v\geq w_2$.
\end{corollary}

\begin{proof} We will argue analogously to Lemma~\ref{lem:numbergoodpspths}.
We start by considering $C_k^{\text{\bf{II}}}$.
By Corollary~\ref{cor:SlivMeck2col}

$$ 
\Ee C_k^{\text{\bf{II}}} 
= \left(p\lambda\right)^k \int_\Dee \dots\int_\Dee \Ee\left[g_k^{\text{\bf{II}}}(z_1,\dots,z_k,\Zcal\cup\{z_1,\dots,z_k\})\right] 
\cdot f(z_1)\cdot\dots\cdot f(z_k) \dd z_k\dots\dd z_1,$$

\noindent
where $f$ is given by~\eqref{eq:fdef} and $g_k^{\text{\bf{II}}}$ is the indicator function 
that $o,z_1,\dots,z_k$ is a chunk (wrt.~the point set $\Zcal\cup\{o,z_1,\dots,z_k\}$), 
whose last pseudo-edge $z_{k-1}z_k$ is of type {\bf II}.

Let $h_{k-1}$ denote the indicator function that $o,z_1,\dots,z_{k-1}$ is a good pseudo-path wrt.~the point set 
$\Zcal\cup\{o,z_1,\dots,z_{k-1}\}$. 
We have

$$ \begin{array}{rcl} g_k^{\text{\bf{II}}}(z_1,\dots,z_k,\Zcal\cup\{z_1,\dots,z_k\})
& \leq & \displaystyle 
h_{k-1}(z_1,\dots,z_{k-1},\Zcal\cup\{z_1,\dots,z_{k-1}\}) \cdot \\ 
& & \displaystyle 
1_{\{\distH(z_{k-1},z_k) \leq r-w_1\}}. \end{array} $$

It follows that 

$$ \begin{array}{rcl} 
    
\Ee C_k^{\text{\bf{II}}} 
& \leq & \displaystyle 
\left(p\lambda\right)^{k-1} \int_\Dee \dots\int_\Dee \Ee\left[h_{k-1}(z_1,\dots,z_{k-1},\Zcal\cup\{z_1,\dots,z_{k-1}\})\right] 
\cdot f(z_1)\cdot\dots\cdot f(z_{k-1}) \cdot \\
& & \displaystyle \hspace{15ex}
\left( p\lambda \int_\Dee 1_{\{\distH(z_{k-1},z_k) \leq r-w_1\}} f(z_k)\dd z_k\right)
\dd z_{k-1}\dots\dd z_1 \\
& = & \displaystyle 
p\lambda \cdot \areaH(\ballH(o,r-w_1)) \cdot \\
& & \displaystyle 
(p\lambda)^{k-1} \cdot \int_\Dee \dots\int_\Dee \Ee\left[h_{k-1}(z_1,\dots,z_{k-1},\Zcal\cup\{z_1,\dots,z_{k-1}\})\right] 
\cdot \\
& & \displaystyle \hspace{17ex} 
f(z_1)\cdot\dots\cdot f(z_{k-1}) \dd z_{k-1}\dots\dd z_1 \\
& = & \displaystyle 
p\lambda \areaH(\ballH(o,r-w_1)) \cdot \Ee G_{k-1},
   \end{array} $$

\noindent
Applying Lemma~\ref{lem:XI} gives~\eqref{eq:CnII}.

The proof of~\eqref{eq:CnIII} is analogous, using the indicator function 
$1_{\{\distH(z_{k-1},z_k)<r+w_2, \angle z_{k-2}z_{k-1}z_k \leq \vartheta\}}$ in place of $1_{\{\distH(z_{k-1},z_k) \leq r-w_1\}}$, 
replacing $\areaH(\ballH(o,r-w_1))$ with $\areaH(\ballH(o,r+w_2)\cap\sect{o}{v}{\vartheta})$ (for $v\neq o$ arbitrary) and 
using Lemma~\ref{lem:XII} in place of Lemma~\ref{lem:XI}.

For the proof of~\eqref{eq:CnIV}, we use that if $g_k^{\text{\bf{IV}},v}$ denotes the indicator
function that $o,z_1,\dots,z_k$ form a chunk (wrt.~$\Zcal\cup\{o,z_1,\dots,z_k\}$) 
whose last pseudo-edge has length $\in [r+v,r+v+1)$ then

$$ \begin{array}{rcl} 
g_k^{\text{\bf{IV}},v}(z_1,\dots,z_k, \Zcal\cup\{z_1,\dots,z_k\}) 
& \leq & 
h_{k-1}(z_1,\dots,z_{k-1},\Zcal\cup\{z_1,\dots,z_{k-1}\}) \cdot  \\
& & 
g_{L_{\text{\bf{IV}},\geq r+v}}(z_1,\dots,z_k,\Zcal), 
\end{array} $$

\noindent 
where 

$$ g_{L_{\text{\bf{IV}},\geq r+v}}(z_1,\dots,z_k,\Zcal)
:= 1_{
\left\{\begin{array}{l}
       \distH(z_{k-1},z_k) \geq r+v, \text{ and;} \\
       \text{ either }\Zcal \cap \BGab^-(z_{k-1},z_k) \setminus \cert(o,z_1,\dots,z_{k-1}) = \emptyset, \\
       \text{ or } 
       \Zcal \cap \BGab^+(z_{k-1},z_k) \setminus \cert(o,z_1,\dots,z_{k-1}) = \emptyset, \\
       \text{(or both).}
      \end{array}
\right\}
}. $$

\noindent
Since $\Zcal \cap \cert(o,z_1,\dots,z_{k-1})$ and $\Zcal \setminus \cert(o,z_1,\dots,z_{k-1})$ are independent for 
any choice of $z_1,\dots,z_{k-1}$, we have

$$\begin{array}{rcl} 
\Ee\left[g_k^{\text{\bf{IV}},v}(z_1,\dots,z_k, \Zcal\cup\{z_1,\dots,z_k\})\right]
& \leq & 
\Ee\left[h_{k-1}(z_1,\dots,z_{k-1},\Zcal\cup\{z_1,\dots,z_{k-1}\})\right]  \cdot \\
& & \Ee\left[g_{L_{\text{\bf{IV}},\geq r+v}}(z_1,\dots,z_k,\Zcal)\right].
  \end{array} $$

\noindent 
If $z_1,\dots,z_{k-1}$ do not form a good pre-pseudopath then $h_{k-1}(z_1,\dots,z_{k-1},\Zcal\cup\{z_1,\dots,z_{k-1}\})=0$.
Otherwise we can apply Corollary~\ref{cor:YIV} (and Corollary~\ref{cor:SlivMeck2col})
to show that 

$$ p\lambda \int_\Dee \Ee\left[g_{L_{\text{\bf{IV}},\geq r+v}}(z_1,\dots,z_k,\Zcal)\right] f(z_k) \dd z_k
\leq 10^4 e^{v/2} e^{-e^{v/2}}. $$

\noindent 
In other words, 

$$ \begin{array}{rcl} 
\Ee C_k^{\text{\bf{IV}},v} 
& = & \displaystyle 
\left(p\lambda\right)^k \int_\Dee \dots\int_\Dee \Ee\left[g_k^{\text{\bf{IV}},\geq r+v}(z_1,\dots,z_k,\Zcal\cup\{z_1,\dots,z_k\})\right] 
\cdot f(z_1)\cdot\dots\cdot f(z_k) \dd z_k\dots\dd z_1 \\
& \leq & \displaystyle 
\left(p\lambda\right)^{k-1} \int_\Dee \dots\int_\Dee \Ee\left[h_{k-1}(z_1,\dots,z_{k-1},\Zcal\cup\{z_1,\dots,z_{k-1}\})\right] 
\cdot f(z_1)\cdot\dots\cdot f(z_{k-1}) \cdot \\
& & \displaystyle \hspace{15ex}
\left( p\lambda \int_\Dee \Ee\left[g_{L_{\text{\bf{IV}},\geq r+v}}(z_1,\dots,z_k,\Zcal)\right]  f(z_k)\dd z_k\right)
\dd z_{k-1}\dots\dd z_1 \\
& \leq & \displaystyle 
\Ee G_{k-1} \cdot 10^4 e^{v/2} e^{-e^{v/2}}, 
\end{array} $$

\noindent
as desired.
\end{proof}

\subsubsection{Counting linked sequences of chunks}

Next, we turn attention to counting the number of linked sequences of chunks.
The first few lemmas are designed for dealing with the first pseudo-edge 
of a new chunk that will be linked to an existing chunk.
The names of the random variables described in these lemmas are $F$ with some subscript, where
$F$ stands for ``first'' and the subscript corresponds to the type of the new pseudo-edge to be added and the type of link
under consideration. 
(The link types being {\bf a}, {\bf b}, {\bf c} corresponding to the cases listed in part {\bf(ii)} of 
Definition~\ref{def:seqchunks}.)

\begin{lemma}\label{lem:FIb}
For every $w_1,w_2,\vartheta>0$ there exist $\lambda_0=\lambda_0(w_1,w_2,\vartheta)$ such that
the following holds for all $0<\lambda<\lambda_0$ and $p \leq 10\lambda$, 
setting $r := 2\log(1/\lambda)$.

Let $u_0,\dots,u_k\in\Dee$ be an arbitrary pre-chunk and define

$$ F_{\text{\bf{I-b}}}  := 
 \left|\left\{ (z_1,z_2) \in \Zcal_b\times\Zcal_b :  
\begin{array}{l}
 0 < \distH(z_1,z_2) < r+w_2, \text{ and;} \\
 \cert(z_1,z_2) \cap \cert(u_0,\dots,u_k) \neq \emptyset.
 \end{array} 
 \right\}\right|. $$  

Then we have

$$ \Ee F_{\text{\bf{I-b}}} \leq 10^4 \cdot k \cdot \exp\left[ \frac32 w_2+\max\left(w_2,\distH(u_{k-1},u_k)-r\right)/2\right].  $$

\end{lemma}

\begin{proof}
We can write $\cert(u_0,\dots,u_k) = B_1 \cup \dots \cup B_m$ with
$m \leq 2k$ and 
$B_1 = \ballH(c_1,s_1), \dots, B_m = \ballH(c_m,s_m)$ balls whose radii 
are either $(r+w_2)/2$ or $\distH(u_{k-1},u_k)/2$.

For $i=1,\dots,m$ we define

$$ F_i := \left|\left\{ (z_1,z_2) \in \Zcal_b\times\Zcal_b :  
\begin{array}{l}
 0 < \distH(z_1,z_2) < r+w_2, \text{ and;} \\
 \ballH(z_1,r+w) \cap B_i \neq \emptyset, \text{ and; }\\
 \ballH(z_2,r+w) \cap B_i \neq \emptyset. 
 \end{array} 
 \right\}\right|. $$  
 
\noindent
Observe that if $\distH(z_1,z_2) < r+w_2$ then $\cert(z_1,z_2)$ either equals $\DD(z_1,z_2,r+w_2)$ or $\emptyset$.
In particular we have $\cert(z_1,z_2) \subseteq \ballH(z_1,r+w) \cap \ballH(z_2,r+w)$.

We see that 

$$ F_{\text{\bf{I-b}}} \leq F_1 + \dots + F_m. $$ 

\noindent
We can thus focus on bounding the individual expectations $\Ee F_i$.
For the moment, we pick some $1\leq i \leq m$.   
Without loss of generality (applying a suitable isometry if needed) the center $c_i$ of $B_i$ is the origin $o$.
Applying Corollary~\ref{cor:SlivMeck2col} and a (double) switch to hyperbolic polar coordinates:

$$ \begin{array}{rcl}
\Ee F_i
& = & \displaystyle 
p^2 \lambda^2 \int_0^\infty \int_0^\infty \int_0^{2\pi} \int_0^{2\pi}
g(z_1(\alpha_1,\rho_1), z_2(\alpha_2,\rho_2))
\sinh(\rho_1)\sinh(\rho_2)\dd\alpha_2\dd\alpha_1\dd\rho_2\dd\rho_1 \\
& \leq & 
\displaystyle 
100 \lambda^4 \int_0^\infty \int_0^\infty \int_0^{2\pi} \int_0^{2\pi}
g(z_1(\alpha_1,\rho_1), z_2(\alpha_2,\rho_2))
e^{\rho_1+\rho_2}\dd\alpha_2\dd\alpha_1\dd\rho_2\dd\rho_1, 
\end{array} $$

\noindent
where

$$ g(u_1,u_2) 
:= 1_{\left\{
\begin{array}{l}
 0 < \distH(u_1,u_2) < r+w_2, \text{ and;} \\
 \ballH(u_1,r+w_2) \cap B_i \neq \emptyset, \text{ and; }\\
 \ballH(u_2,r+w_2) \cap B_i \neq \emptyset. 
 \end{array} 
\right\}}. $$

\noindent
Note that $\ballH(z_i,r+w_2) \cap B_i \neq \emptyset$ implies that $\rho_i = \distH(o,z_i) < s_i+r+w_2$.
Applying Lemma~\ref{lem:distHvsangle}, we see that $\distH(z_1,z_1) < r+w_2$ 
implies that $\angle z_1oz_2 = |\alpha_1-\alpha_2|_{2\pi} < 2\pi e^{(r+w_2-(\rho_1+\rho_2))/2}$.
(Here and elsewhere $|x|_{2\pi} := \min(|x|, 2\pi-|x|)$ for $-2\pi\leq x \leq 2\pi$.)
Therefore

$$ g(z_1,z_2)
\leq 1_{\left\{\begin{array}{l} 
\rho_1,\rho_2 < s_i+r+w_2, \\  
|\alpha_1-\alpha_2|_{2\pi} < 2\pi e^{(r+w_2-(\rho_1+\rho_2))/2}
\end{array}\right\}}. $$

It follows that 

\begin{equation}\label{eq:VIIeq1} 
\begin{array}{rcl}
\Ee F_i 
& \leq & \displaystyle 
100 \lambda^4 \int_0^{s_i+r+w_2} \int_0^{s_i+r+w_2} \int_0^{2\pi} \int_0^{2\pi}
1_{\left\{|\alpha_1-\alpha_2|_{2\pi}< 2\pi e^{(r+w_2-\rho_1-\rho_2)/2}\right\}} \cdot \\
& & \displaystyle 
\hspace{35ex} e^{\rho_1+\rho_2} \dd\alpha_2\dd\alpha_1\dd\rho_2\dd\rho_1.
\end{array} 
\end{equation}

By symmetry considerations

$$ 
\int_0^{2\pi} \int_0^{2\pi}
1_{\left\{|\alpha_1-\alpha_2|_{2\pi}< 2\pi e^{(r+w_2-(\rho_1+\rho_2))/2}\right\}}
\dd\alpha_2\dd\alpha_1
\leq 8 \pi^2 \cdot e^{(r+w_2-(\rho_1+\rho_2))/2}.  $$

\noindent
(We have an inequality and not an equality to account for the possibility that $2\pi e^{r+w_2-(\rho_1+\rho_2)} \geq \pi$.)
Filling this back into~\eqref{eq:VIIeq1}, we obtain

$$ \begin{array}{rcl} 
\Ee F_i 
& \leq &  \displaystyle 
10^4 \lambda^4 e^{(r+w_2)/2}\int_0^{s_i+r+w_2} \int_0^{s_i+r+w_2} e^{(\rho_1+\rho_2)/2} \dd\rho_2\rho_1 \\
& = & \frac{1}{4}\cdot 10^4 \lambda^4 e^{(r+w_2)/2} \cdot e^{s_i+r+w_2} \\
& = & \frac14 \cdot 10^4 \cdot \exp\left[\frac32 w_2+s_i-r/2\right] \\
& \leq & \frac14 \cdot 10^4 \cdot \exp\left[ \frac32 w_2+\max\left(w_2, \distH(u_{k-1},u_k)-r\right)/2\right].
\end{array} $$

\noindent
using in the last line that $2s_i$, the diameter of $B_i$, is either $r+w_2$ or $\distH(u_{k-1},u_k)$.
Adding the bounds on $\Ee Y_i'$ and using the inequality $m \leq 2k$ gives

$$ \begin{array}{rcl} 
\Ee F_{\text{\bf{I-b}}} 
& \leq & \Ee F_1+\dots+\Ee F_m \\
& \leq & 2k \cdot \frac14 \cdot 10^4 \cdot \exp\left[ \frac32 w_2+\max\left(w_2, \distH(u_{k-1},u_k)-r\right)/2\right] \\
& \leq & 10^4 \cdot k \cdot \exp\left[\frac32 w_2+\max\left(w_2, \distH(u_{k-1},u_k)-r\right)/2\right], 
\end{array} $$

\noindent
as claimed in the lemma statement.
\end{proof}

\begin{lemma}\label{lem:FIc}
For every $w_1,w_2,\vartheta>0$ there exist $\lambda_0=\lambda_0(w_1,w_2,\vartheta)$ such that
the following holds for all $0<\lambda<\lambda_0$ and $p \leq 10\lambda$, 
setting $r := 2\log(1/\lambda)$.

Let $u_0,\dots,u_k\in\Dee$ be an arbitrary pre-chunk and define

$$ F_{\text{\bf{I-c}}}  := 
 \left|\left\{ (z_1,z_2) \in \Zcal_b\times\Zcal_b :  
\begin{array}{l}
 0 < \distH(z_1,z_2) < r+w_2, \text{ and;} \\
 z_1z_2 \text{ is a pseudoedge, and; } \\
 \cert(z_1,z_2)\cap \cert(u_0,\dots,u_k) = \emptyset, \text{ and;} \\
 \distH(z_1,\cert(u_0,\dots,u_k))<r/1000.
 \end{array} 
 \right\}\right|. $$  

Then we have

$$ \Ee F_{\text{\bf{I-c}}} \leq k \cdot \exp\left[ \max\left(w_2,\distH(u_{k-1},u_k)-r\right)/2\right].  $$

\end{lemma}

\begin{proof}
As usual, we let $\lambda_0>0$ be a small constant, to be determined in the course of the proof.
As in the previous proof, we write $C := \cert(u_0,\dots,u_k) = B_1 \cup \dots \cup B_m$ with
$m \leq 2k$ and 
$B_1 = \ballH(c_1,s_1), \dots, B_m = \ballH(c_m,s_m)$ balls whose radii 
are either $(r+w_2)/2$ or $\distH(u_{k-1},u_k)/2$.
Let us write $C' := \{ u \in \Dee : \distH(u,C) < r/1000\}$.
By Corollary~\ref{cor:SlivMeck2col}

$$ \begin{array}{rcl} 
\Ee F_{\text{\bf{I-c}}} 
& \leq & \displaystyle 
(p\lambda)^2 \int_\Dee\int_\Dee
1_{\{z_1\in C'} \dot 1_{\left\{\begin{array}{l}z_1z_2 \text{ pseudoedge,} \\ \distH(z_1,z_2)<r+w_2\end{array}\right\}} 
f(z_1)f(z_2)\dd z_1\dd z_2 \\
& = & \displaystyle
(p\lambda)^2 \areaH(C') \cdot \int_{\ballH(o,r+w_2)} 1_{\{0z \text{ pseudoedge}\}} f(z)\dd z \\
& \leq & p\lambda \areaH(C') \cdot \left( p \Ee D + \Ee X_{\text{\bf{I}}} + \Ee\Xtil_{\text{\bf{VI}}} \right) \\
& \leq & 10 \cdot \lambda^2 \cdot 2k \cdot \pi e^{ \max(r+w_2,\distH(u_{k-1},u_k))/2+r/1000} \cdot \\
& & \displaystyle \quad \quad 
\left( 10 \lambda \cdot \left( 6 + \frac{3}{\pi\lambda}\right) + 1000 e^{-w_2} + 1000 e^{w_2}e^{-e^{w_2/2}} \right) \\
& = & \displaystyle 
k \cdot e^{\max\left(w_2,\distH(u_{k-1},u_k)-r\right)/2} \cdot e^{ - \left(\frac{499}{1000}\right)r } \cdot \\
& & \displaystyle \quad \quad 
\cdot \left( 120 \pi \lambda + 60 + 2000 \pi e^{-w_2} + 2000 \pi e^{w_2}e^{-e^{w_2/2}} \right) \\
& \leq & k \cdot e^{\max\left(w_2,\distH(u_{k-1},u_k)-r\right)/2},
\end{array} $$

\noindent 
where $D$ is the ``typical degree'' as given by~\eqref{eq:typdegdef}, $X_{\text{\bf{I}}}$ is as defined 
in Lemma~\ref{lem:XI} and $\Xtil_{\text{\bf{VI}}}$ as defined in Lemma~\ref{lem:XVItil}, and;
we apply Isokawa's formula (Theorem~\ref{thm:Isokawa}) and Lemmas~\ref{lem:XI} and~\ref{lem:XVItil}
and we use that $p \leq 10\lambda$ and that $\areaH( \ballH(u,s) ) \leq \pi e^s$ for all $u\in\Dee, s>0$ 
to obtain the fourth line. In the fifth line we use that $r = 2\ln(1/\lambda)$ and 
in the last line that we chose $\lambda_0=\lambda_0(w_1,w_2,\vartheta)$ sufficiently small (so that $r$ is large).
\end{proof}

%

%
%
%
%
%
%
%
%
%

\begin{lemma}\label{lem:FIVa}
For every $w_1,w_2,\vartheta>0$ there exist $\lambda_0=\lambda_0(w_1,w_2,\vartheta)$ such that
the following holds for all $0<\lambda<\lambda_0$ and $p \leq 10\lambda$, 
setting $r := 2\log(1/\lambda)$.

Let $u_0,\dots,u_k \in \Dee$ be an arbitrary pre-chunk and 
write $C := \cert(u_0,\dots,u_k)$ and

$$ F_{\text{\bf{IV-a}},\geq d}  := 
 \left|\left\{ z \in \Zcal_b  :  
\begin{array}{l}
 \distH(z,u_k) \geq d, \text{ and; } \\ 
 \distH(z,C) \geq r/1000, \text{ and; } \\
 \Zcal \cap \BGab^-(z,u_k) \setminus C = \emptyset 
 \text{ or } \Zcal \cap \BGab^+(z,u_k) \setminus C = \emptyset.
 \end{array} 
 \right\}\right|. $$  

Then we have

$$ \Ee F_{\text{\bf{IV-a}},\geq r+v} \leq 
10^3 \cdot e^{v/2} \cdot e^{-e^{v/2}},  $$

\noindent
for all $v \geq w_2$. 
\end{lemma}

\begin{proof}
The demand that $\distH(z,C) > r/1000$ implies that 

$$ \ahd(C,\BGab(u_k,z)) > r/1000, $$ 

\noindent
as well. Provided $\lambda_0$ is sufficiently small (so that $r$ is sufficiently large), 
Lemma~\ref{lem:ahdeps} (with $\eps=1/10$) implies that 

$$ \areaH(\BGab(u_k,z) \cap C ) \leq \frac{1}{10} \cdot \areaH(\BGab(u_k,z)). $$

\noindent
The result now follows from immediately Lemma~\ref{lem:LIVa}.
\end{proof}

\begin{lemma}\label{lem:FIVb}
For every $w_1,w_2,\vartheta>0$ there exist $\lambda_0=\lambda_0(w_1,w_2,\vartheta)$ such that
the following holds for all $0<\lambda<\lambda_0$ and $p \leq 10\lambda$, 
setting $r := 2\log(1/\lambda)$.

Let $u_0,\dots,u_k \in \Dee$ be an arbitrary pre-chunk and 
write $C := \cert(u_0,\dots,u_k)$ and

$$ F_{\text{\bf{IV-b}},\geq d}  := 
 \left|\left\{ (z_1,z_2) \in \Zcal_b\times\Zcal_b :  
\begin{array}{l}
 \distH(z_1,z_2) \geq d, \text{ and; } \\ 
 \distH( z_1, C ) \geq r/1000, \text{ and;} \\
 \BGab(z_1,z_2) \cap C \neq \emptyset, \text{ and; }\\
 \Zcal \cap \BGab^-(z_1,z_2) \setminus C = \emptyset 
 \text{ or } \Zcal \cap \BGab^+(z_1,z_2) \setminus C = \emptyset.
 \end{array} 
 \right\}\right|. $$  

Then we have

$$ \Ee F_{\text{\bf{IV-b}},\geq r+v} \leq 
10^6 \cdot k \cdot e^{v} \cdot e^{-e^{v/2}} \cdot e^{\max(w_2,\distH(u_{k-1},u_k)-r)/2},  $$

\noindent
for all $v \geq w_2$. 
\end{lemma}

\begin{proof}
As usual, we let $\lambda_0>0$ be a small positive constant, to be determined during the course of the proof, and 
we fix an arbitrary $v \geq w_2$, $0< \lambda < \lambda_0$ and $p \leq 10 \lambda$.

We can write $\cert(u_0,\dots,u_k) = B_1 \cup \dots \cup B_m$ with
$m \leq 2k$ and 
$B_1 = \ballH(c_1,s_1), \dots, B_m = \ballH(c_m,s_m)$ balls whose radii 
are either $(r+w_2)/2$ or $\distH(u_{k-1},u_k)/2$.

For $i=1,\dots,m$ we define

$$ F_i := \left|\left\{ (z_1,z_2) \in \Zcal_b\times\Zcal_b :  
\begin{array}{l}
 \distH(z_1,z_2) \geq r+v, \text{ and; } \\ 
 \distH( z_1, C ) > r/1000, \text{ and;} \\
 \BGab(z_1,z_2) \cap B_i \neq \emptyset, \text{ and; }\\
 \Zcal \cap \BGab^-(z_1,z_2) \setminus C = \emptyset 
 \text{ or } \Zcal \cap \BGab^+(z_1,z_2) \setminus C \neq \emptyset.
 \end{array} 
 \right\}\right|, $$  

\noindent
and we point out that 

$$ F_{\text{\bf{IV-b}},\geq r+v} \leq F_1 + \dots + F_m. $$

\noindent
In particular it suffices to bound each expectation $\Ee F_i$ separately.
For the moment we fix $1\leq i \leq m$. Applying a suitable isometry if needed, we
assume without loss of generality that the center of $B_i$ is $c_i=o$ the origin.

By Corollary~\ref{cor:SlivMeck2col}:

$$ \Ee F_i = 
p^2\lambda^2 \int_{\Dee}\int_{\Dee} \Ee\left[g(z_1,z_2,\Zcal\cup\{z_1,z_2\})\right]  
f(z_1) f(z_2)\dd z_2\dd z_1, $$

\noindent
with $f$ given by~\eqref{eq:fdef} and

$$ g(u_1,u_1,\Ucal) := 
1_{\left\{\begin{array}{l} 
\distH( u_1, C ) > r/1000, \text{ and;} \\
\distH( u_1,u_2) \geq r+v, \text{ and;} \\
 \BGab(u_1,u_2) \cap B_i \neq \emptyset, \text{ and; }\\
\Ucal \cap \BGab^-(u_1,u_2) \setminus C = \emptyset
\text{ or } \Ucal \cap \BGab^+(z_1,z_2) \setminus C \neq \emptyset.
\end{array}
\right\}}. $$

Applying a change to hyperbolic polar coordinates to $z_1$ 
we find

\begin{equation}\label{eq:VIeq1} 
\begin{array}{rcl} 
\Ee F_i 
& = & \displaystyle 
p^2\lambda^2 \int_{s+r/1000}^\infty \int_0^{2\pi} \int_\Dee 
\Ee\left[g( z_1(\alpha_1,\rho_1), z_2,\Zcal\cup\{z_1(\alpha_1,\rho_1), z_2\} )\right]  \cdot \\
& & \displaystyle \hspace{25ex}
f(z_2)\dd z_2 \dd \alpha_1 \sinh(\rho_1) \dd \rho_1. 
\end{array}
\end{equation}

Next, we consider the inner integral for some (fixed for the moment) 
$z_1 = (\cos(\alpha)\cdot\tanh(\rho_1/2),\sin(\alpha)\cdot\tanh(\rho_1/2))$ satisfying $\distH(z_1,C) > r/1000$.
Let $\varphi : \Dee \to \Dee$ be a hyperbolic isometry that maps $z_1$ to the origin and $c_i=o$ to the negative $x$-axis.
(So that $\varphi(c_i) = (-\tanh(r_1/2),0)$.)
Note that

$$ \Ee\left[g(z_1,z_2,\Zcal\cup\{z_1,z_2\})\right] = h(\varphi(z_2)), $$

\noindent
where 

$$ h(z) := 1_{\left\{\begin{array}{l}
\ballH( \varphi(c_i),s_i ) \cap \BGab( o, z ) \neq \emptyset, \\ 
\distH(o,z) > r+v, \\
\distH(o,C) \geq r/1000
\end{array}\right\}} \cdot 
\Pee\left(
\begin{array}{l}
\text{ either } \Zcal \cap \BGab^-(o,z) \setminus \varphi[C] = \emptyset, \\
\text{ or } \Zcal \cap \BGab^+(o,z) \setminus  \varphi[C] = \emptyset, \\
\text{ (or both). }
\end{array} \right). $$

We point out that if $\ballH( \varphi(c),s_i ) \cap \BGab( o, z ) \neq \emptyset$ 
then it must certainly hold that $\distH(o,z) > \distH(o,\varphi(c))-s_i = \rho_1 - s$.
By Lemma~\ref{lem:angle}, using that $\distH(o,\varphi(c)) = \distH(z_1,c)=\rho_1 > s_i+r/1000$ by 
assumption, we must also have that 

$$ \angle \varphi(c)oz < 10e^{(s_i-\rho_1)/2}. $$

We see that 

$$ 1_{\left\{\begin{array}{l}
\ballH( \varphi(c),s_i ) \cap \BGab( o, z ) \neq \emptyset, \\ 
\distH(o,z) > r+v.
\end{array}\right\}}
\leq 
1_{\left\{\begin{array}{l}
\distH(o,z) > \max(\rho_1-s_i,r+v), \\
\angle \varphi(c)oz < 10e^{(s_i-\rho_1)/2}
\end{array}
\right\}}. $$

Since $\distH(o,\varphi[C] ) = \distH(z_1,C) >r/1000$, we certainly have 

$$\ahd(  \varphi[C], \BGab(0,z) ) > r/1000. $$

\noindent
Assuming $\lambda_0$ was chosen sufficiently small, we can apply Lemma~\ref{lem:ahdeps} to see that

$$ \areaH( \BGab(o,z) \cap \varphi[C] ) \leq \frac{1}{1000} \cdot \areaH( \BGab(o,z) ), $$

\noindent
and hence 

$$ \areaH( \BGab^-(o,z) \setminus \varphi[C] )
 \geq  \frac{499}{1000} \cdot \areaH( \BGab(0,z) ) 
\geq e^{\distH(o,z)/2}, $$

\noindent
where the last inequality holds assuming we chose $\lambda_0$ sufficiently small, and assuming that $\distH(o,z) > r+v$ (and using
$r = 2\ln(1/\lambda)$).
Completely analogously, $\areaH( \BGab^+(o,z) \setminus \varphi[C] )  \geq e^{\distH(o,z)/2}$ as well.

We see that 

$$ h(z)  \leq 1_{\left\{\begin{array}{l}
\distH(o,z) > \max(\rho_1-s_i,r+v), \\
\angle \varphi(c)oz < 10e^{(s_i-\rho_1)/2}
\end{array}
\right\}} 
\cdot 
2 e^{-\lambda e^{\distH(o,z)/2}}
=: \psi(z). $$

Applying~\eqref{eq:isosubst}, we find

$$ \int_\Dee \Ee\left[g( z_1(\alpha,r_1), z_2,\Zcal )\right] f(z_2)\dd z_2
\leq \int_\Dee \psi(\varphi(z_2)) f(z_2) \dd z_2 
= \int_\Dee \psi(u) f(u) \dd u. $$

Changing to hyperbolic coordinates, i.e.~$u(\alpha_2,\rho_2) = 
(\cos(\alpha_2)\cdot\tanh(\rho_2/2), \sin(\alpha_2)\cdot \tanh(\rho_2/2) )$
gives

$$ \begin{array}{rcl}
\displaystyle 
\int_\Dee \psi(u) f(u) \dd u 
& = & \displaystyle
\int_0^\infty \int_0^{2\pi} \psi( u(\alpha_2,\rho_2) ) \dd \alpha_2 \sinh(\rho_2) \dd \rho_2 \\[2ex]
& \leq &  \displaystyle
 \int_{\max(\rho_1-s_i,r+v)}^\infty \int_0^{2\pi} 1_{\{ |\alpha_2-\pi| <10 e^{(s_i-\rho_1)/2}\}} 
2 e^{-\lambda e^{\rho_2/2}}
 \dd \alpha_2 \sinh(\rho_2) \dd\rho_2 \\[2ex]
& = & \displaystyle
\int_{\max(\rho_1-s_i,r+v)}^\infty 20 e^{(s_i-\rho_1)/2} 2 e^{-\lambda e^{\rho_2/2}} \sinh(\rho_2) \dd\rho_2 \\[2ex]
& = & \displaystyle
40 e^{(s-\rho_1)/2} \int_{\max(\rho_1-s_i,r+v)}^\infty e^{-\lambda e^{\rho_2/2}} \sinh(\rho_2) \dd\rho_2 \\[2ex]
& \leq & \displaystyle
20 e^{(s_i-\rho_1)/2} \int_{\max(\rho_1-s_i,r+v)}^\infty e^{-\lambda e^{\rho_2/2}} e^{\rho_2} \dd\rho_2 
\end{array} $$

\noindent
We next apply the substitution $t := \lambda e^{\rho_2/2}$ (so that $\dd\rho_2 = \frac{2\dd t}{t}$) to obtain

$$ \begin{array}{rcl} \displaystyle 
\int_\Dee \psi(u) f(u) \dd u 
& \leq & \displaystyle 
\frac{20 e^{(s_i-\rho_1)/2}}{\lambda^2} \cdot \int_{\lambda e^{\max(\rho_1-s_i,r+v)/2}}^\infty te^{-t} \dd t \\[2ex]
& = & \displaystyle 
\frac{20 e^{(s_i-\rho_1)/2}}{\lambda^2} \cdot 
\left( \lambda e^{\max(\rho_1-s_i,r+v)/2} + 1 \right) \cdot e^{-\lambda e^{\max(\rho_1-s_i,r+v)/2}} \\[2ex]
& \leq & \displaystyle 
\frac{40 e^{(s_i-\rho_1)/2}}{\lambda^2} \lambda e^{\max(\rho_1-s_i,r+v)/2} e^{-\lambda e^{\max(\rho_1-s_i,r+v)/2}},     
   \end{array} $$
   
\noindent
using in the last line that $\lambda e^{r/2} = 1$ and $v\geq w > 0$.

Filling this back into~\eqref{eq:VIeq1} we find

$$ \begin{array}{rcl} \Ee F_i 
& \leq & \displaystyle 
p^2\lambda^2 \int_{s_i+r/1000}^\infty \int_0^{2\pi}
\frac{40\pi e^{(s_i-\rho_1)/2}}{\lambda^2} \lambda e^{\max(\rho_1-s_i,r+v)/2} e^{-\lambda e^{\max(\rho_1-s_i,r+v)/2}}
\dd\alpha_1 \sinh(\rho_1) \dd\rho_1 \\[2ex] 
& = & \displaystyle 
80\pi p^2 \lambda e^{s_i/2} \int_{s+r/1000}^\infty e^{-\rho_1/2} \cdot e^{\max(\rho_1-s_i,r+v)/2} e^{-\lambda e^{\max(\rho_1-s_i,r+v)/2}}
\sinh(\rho_1) \dd\rho_1 \\[2ex]
& \leq & \displaystyle 
10^4 \lambda^3 e^{s_i/2} \int_{s_i+r/1000}^\infty e^{\rho_1/2} \cdot e^{\max(\rho_1-s_i,r+v)/2} e^{-\lambda e^{\max(\rho_1-s_i,r+v)/2}}
\dd\rho_1 \\
& = & \displaystyle 
10^4 \lambda^3 e^{s_i/2} \cdot \left( 
\int_{s+r/1000}^{s_i+r+v} e^{\rho_1/2} \cdot e^{(r+v)/2} e^{-\lambda e^{(r+v)/2}} \dd\rho_1 \right. \\[2ex]
& & \displaystyle \left. \hspace{10ex} + 
\int_{s_i+r+v}^\infty e^{\rho_1/2} \cdot e^{(\rho_1-s_i)/2} e^{-\lambda e^{(\rho_1-s_i)/2}}\dd\rho_1 \right) \\[2ex]
& =: & \displaystyle
10^4 \lambda^3 e^{s_i/2} \cdot (I_1+I_2),
\end{array} $$

\noindent
using that $\sinh(\rho_1) \leq \frac12 e^{\rho_1}$ and $p \leq 10\lambda$ in the third line.

Now

\begin{equation}\label{eq:VIeq2} 
\begin{array}{rcl} 
10^4 \lambda^3 e^{s_i/2} I_1 
& \leq & 2\cdot 10^4 \lambda^3 e^{s_i/2} \cdot e^{(s_i+r+v)/2} \cdot e^{(r+v)/2} e^{-\lambda e^{(r+v)/2}} \\
& = & 
2\cdot 10^4 \cdot e^{s_i-r/2} \cdot e^{v} \cdot e^{-e^{v/2}}, 
\end{array}
\end{equation}

\noindent
using that $r=2\ln(1/\lambda)$.
   
We also have 

\begin{equation}\label{eq:VIeq3} 
\begin{array}{rcl} 
10^4 \lambda^3 e^{s_i/2} I_2 
& = & \displaystyle 
10^4 \lambda^3 e^{s_i} \int_{s_i+r+v}^\infty e^{\rho_1-s_i} e^{-\lambda e^{(\rho_1-s_i)/2}}\dd\rho_1 \\[2ex]
& = & \displaystyle 
10^4 \lambda^3 e^{s_i}  \int_{r+v}^\infty e^u e^{-\lambda e^{u/2}} \dd u \\[2ex]
& = & \displaystyle 
2 \cdot 10^4 \lambda e^{s_i} \int_{e^{v/2}}^\infty t e^{-t} \dd t \\[2ex]
& \leq & \displaystyle 
4 \cdot 10^4 \lambda e^{s_i} e^{v/2} e^{-e^{v/2}} \\[2ex]
& \leq & \displaystyle 
4 \cdot 10^4 \cdot e^{s_i-r/2} \cdot e^v \cdot e^{-e^{v/2}},
\end{array}
\end{equation}

\noindent
using the substitution $u=\rho_1-s_i$ in the second line, the substitution $t := \lambda e^{u/2}$ in the third line
and that $r = 2\ln(1/\lambda)$.
Adding~\eqref{eq:VIeq2} and~\eqref{eq:VIeq3} and multiplying by $2k$ gives the result.
\end{proof}

\begin{lemma}\label{lem:FIVc}
For every $w_1,w_2,\vartheta>0$ there exist $\lambda_0=\lambda_0(w_1,w_2,\vartheta)$ such that
the following holds for all $0<\lambda<\lambda_0$ and $p \leq 10\lambda$, 
setting $r := 2\log(1/\lambda)$.

Let $u_0,\dots,u_k \in \Dee$ be an arbitrary pre-chunk and 
write $C := \cert(u_0,\dots,u_k)$ and

$$ F_{\text{\bf{IV-c}},\geq d}  := 
 \left|\left\{ (z_1,z_2) \in \Zcal_b\times\Zcal_b :  
\begin{array}{l}
 \distH(z_1,z_2) \geq d, \text{ and; } \\ 
 \distH( z_1, C ) \leq r/1000, \text{ and;} \\
 \BGab(z_1,z_2) \cap C = \emptyset, \text{ and; }\\
 \Zcal \cap \BGab^-(z_1,z_2)  = \emptyset 
 \text{ or } \Zcal \cap \BGab^+(z_1,z_2) = \emptyset.
 \end{array} 
 \right\}\right|. $$  

Then we have

$$ \Ee F_{\text{\bf{IV-c}},\geq r+v} \leq 
10^3 \cdot k \cdot e^{v} \cdot e^{-e^{v/2}} \cdot e^{\max(w_2,\distH(u_{k-1},u_k)-r)/2},  $$

\noindent
for all $v \geq w_2$. 
\end{lemma}

\begin{proof}
As usual, we let $\lambda_0>0$ be a small constant, to be chosen appropriately during the course of the proof.
Writing $C' := \{ u \in\Dee : \distH(u,C) < r/1000\}$ and applying Corollary~\ref{cor:SlivMeck2col}, 
we have

$$ \begin{array}{c} 
\Ee F_{\text{\bf{IV-b}},\geq r+v} \\
\leq \\
\displaystyle 
(p \lambda)^2 \int_\Dee \int_\Dee 
1_{\left\{z_1\in C', \distH(z_1,z_2) \geq r+v \right\}} 
\cdot  \Pee( \Zcal \cap \BGab^-(z_1,z_2)  = \emptyset 
 \text{ or } \Zcal \cap \BGab^+(z_1,z_2) = \emptyset ) \\
\displaystyle \quad \quad
 f(z_1)f(z_2) \dd z_2\dd z_1 \\[2ex]
 \leq \\
 p\lambda \cdot \areaH(C') \cdot 10^3 \cdot e^{v} \cdot e^{-e^{v/2}},  
   \end{array} $$
 
 \noindent
 applying Lemma~\ref{lem:XVtil}.
 Next we remark
 
 $$ \begin{array}{rcl} 
 p\lambda \cdot \areaH(C') 
 & \leq & 
 10 \lambda^2 \cdot 2k \cdot \pi e^{\max(r+w_2,\distH(u_{k-1},u_k))/2 + r/1000} \\
 & \leq & 
 20 \pi \cdot k \cdot e^{\max(w_2,\distH(u_{k-1},u_k)-r)/2 - \left(\frac{499}{1000}\right)r } \\
 & \leq & 
 k \cdot e^{\max(w_2,\distH(u_{k-1},u_k)-r)/2 },
 \end{array} $$
 
 \noindent
 the last inequality holding provided we chose $\lambda_0$ sufficiently small.
\end{proof}

Recall that $C_k^{\text{\bf{II}}}, C_k^{\text{\bf{III}}}, C_k^{\text{\bf{IV}},v}$ denote the number
of black chunks starting from the origin, with length $k$ and final pseudo-edge of type {\bf{II}}, respectively
type {\bf{III}}, respectively type {\bf{IV}} and final pseudo-edge of length $\in [r+v,r+v+1)$.

We now need to introduce analogous notations to deal with linked sequences of chunks.
We set 

$$
S_{n,k}^{\text{\bf{II}}} := \left|\left\{\begin{array}{l}
\text{black linked sequences of chunks starting from $o$, consisting of precisely $n$ chunks,} \\ 
\text{having a last chunk of length $k$ and final pseudo-edge of type $\text{\bf{II}}$}
\end{array} \right\}\right|.
$$
We let $S_{n,k}^{\text{\bf{III}}}$ and $S_{n,k}^{\text{\bf{IV}},v}$ be defined analogously.
Of course 

\begin{equation}\label{eq:S1isC} 
S_{1,k}^{\text{\bf{II}}} = C_k^{\text{\bf{II}}}, \quad S_{1,k}^{\text{\bf{III}}} = C_k^{\text{\bf{III}}},
\quad S_{1,k}^{\text{\bf{IV}},v} = C_k^{\text{\bf{IV}},v}. 
\end{equation}

In particular, Corollary~\ref{cor:Ck} provides bounds on the expectations 
$\Ee S_{n,k}^{\text{\bf{II}}}, \Ee S_{n,k}^{\text{\bf{III}}}$ and $S_{n,k}^{\text{\bf{IV}},v}$ when $n=1$.
The next lemma provides a system of recursive inequalities that will allow us to also
bound these expectations for $n\geq 2$. 
From now on, it will be convenient to assume the parameter $w_2$ is an integer. Note that so far the only result that 
puts any restrictions on the value of $w_2$ is Corollary~\ref{cor:noinfgoodpaths} which just states it has to be 
taken sufficiently large.

\begin{lemma}\label{lem:recurs}
For every $w_1,\vartheta>0$ and $w_2 \in\eN$ there exists a $\lambda_0=\lambda_0(w_1,w_2,\vartheta)$ such that for all
$0<\lambda<\lambda_0$ and $p\leq 10\lambda$, setting $r:=2\ln(1/\lambda)$, the following holds.

For all $n\geq 1$, writing

$$ \Sigma_n := 10^7 \cdot \sum_{\ell=1}^\infty \ell \cdot \left( e^{2w_2} \cdot \Ee S_{n,\ell}^{\text{\bf{II}}} 
 + e^{2w_2} \cdot \Ee S_{n,\ell}^{\text{\bf{III}}} 
 + \sum_{v=w_2}^\infty e^{2v+2} \cdot \Ee S_{n,\ell}^{\text{\bf{IV}},v} \right), $$
 
\noindent
we have

$$ \begin{array}{rcl}
\Ee S_{n+1,1}^{\text{\bf{II}}} & \leq &  
 \Sigma_n \cdot 10^3 \cdot e^{-w_1} \\[2ex]
 \Ee S_{n+1,2}^{\text{\bf{III}}} & \leq &  
 \Sigma_n \cdot 10^3 \cdot \vartheta \cdot e^{2w_2} \\[2ex]
 \Ee S_{n+1,1}^{\text{\bf{IV}},v} & \leq & 
 \Sigma_n \cdot 10^6 \cdot e^{v} \cdot e^{-e^{v/2}}
 \end{array} $$

\noindent
and 

$$
\Ee S_{n+1,k}^{\tau} \leq 
\Sigma_n \cdot \Ee C_{k-1}^{\tau}, 
$$

\noindent
for $\tau$ one of $\text{\bf{II}},\text{\bf{III}}$
or $(\text{\bf{IV}},v)$ with $v\geq w_2$, and for all $k\geq 3$ when $\tau=\text{\bf{III}}$ and all
$k\geq 2$ otherwise.
\end{lemma}

\begin{proof}
We'll need to introduce even more notation. We let

$$ \Tscr := \{\text{\bf{II}}, \text{\bf{III}}, (\text{\bf{IV}},v) : v = w_2, w_2+1,\dots \}, 
\quad \Lscr := \{ a, b, c\},$$

\noindent 
denote the possible {\em types} of the final pseudo-edges of chunks, respectively 
types of links between consecutive chunks -- corresponding to the cases described in part~{\bf(ii)} of Definition~\ref{def:seqchunks}. 

For $n\in\eN$ and $\underline{k} \in \eN^n, \underline{\tau} \in \Tscr^n, \underline{t} \in \Lscr^{n-1}$ we write
$S_{\underline{k},\underline{\tau},\underline{t}}$ for the number of black, linked sequences of chunks 
$P_1,\dots,P_n$ starting from $o$, such that $P_i$ has length $k_i$ and final pseudo-edge of type $\tau_i$ (for all $i=1,\dots,n$), and 
such that the link between $P_{i}$ and $P_{i+1}$ 
corresponds to case $t_i$ of part~{\bf(ii)} of Definition~\ref{def:seqchunks} (for $i=1,\dots,n=1$).
Clearly we can write for each $n,k\geq 1$ and $\tau \in \Tscr$:

\begin{equation}\label{eq:Ssum} 
S_{n,k}^\tau = 
\sum_{{\underline{k} \in \eN^n, \underline{\tau} \in \Tscr^n, \underline{t} \in \Lscr^{n-1},}\atop \tau_n=\tau,k_n=k} 
S_{\underline{k},\underline{\tau},\underline{t}}. 
\end{equation}

We will first derive the last inequality claimed in the lemma statement.
Let us fix $n, k,\ell \in \eN$ and $\sigma, \tau \in \Tscr$ such that either $k\geq 2$ in case $\tau\neq\text{\bf{III}}$ or $k\geq 3$ in case
$\tau =\text{\bf{III}}$. We pick vectors $\underline{k}\in\eN^n, \underline{\tau}\in\Tscr^n,\underline{t}\in\Lscr^{n-1}$ 
satisfying $k_{n-1}=\ell,k_n=k$ and $\tau_{n-1}=\sigma,\tau_n=\tau$.

The parameters $\underline{k},\underline{\tau},\underline{t}$ determine the number of vertices in any of the linked sequences of 
pseudopaths counted by $S_{\underline{k},\underline{\tau},\underline{t}}$. For any such sequence of pseudopaths we have 
$V(P_1)\cup\dots\cup V(P_n) 
= \{o,z_1,\dots,z_m\}$ with $z_1,\dots,z_m\in\Zcalb$ and $k_1+\dots+k_n \leq m \leq k_1+\dots+k_n+n-1$.
(To be precise $m = k_1+\dots+k_n + |\{1\leq i \leq n-1 : t_i \neq a \}|$ as $V(P_i) = k_i+1$ for all $i$ but every time $t_i=a$ the last 
vertex of $P_i$ coincides with the first vertex of $P_{i+1}$.)
By Corollary~\ref{cor:SlivMeck2col} we can write, for $m=m(\underline{k},\underline{\tau},\underline{t})$ as determined by the parameters

$$
\Ee S_{\underline{k},\underline{\tau},\underline{t}}
= (p\lambda)^m \int_\Dee\dots\int_\Dee
\Ee\left[g_{\underline{k},\underline{\tau},\underline{t}}(z_1,\dots,z_m;\Zcal\cup\{z_1,\dots,z_m\})\right] 
f(z_1)\dots f(z_m)\dd z_1\dots\dd z_m, 
$$

\noindent
where $f$ is given by~\eqref{eq:fdef} and $g_{\underline{k},\underline{\tau},\underline{t}}$ is the indicator function 
that $o,z_1,\dots,z_m$ form a
linked sequence of pseudopaths of the required kind wrt.~the point set $\Zcal\cup\{o,z_1,\dots,z_m\}$.

The vertices of $P_n$ are $z_{m-k}, \dots, z_m$.
If $t_{n-1}=a$ then $z_m$ is both the first vertex of $P_n$ and also the last vertex of $P_{n-1}$.
Recall that $\cert(z_{m-k+1},\dots,z_m)$ needs to be disjoint from $\cert(P_1)\cup\dots\cup \cert(P_{n-1})$
(and this set is completely determined by $z_1,\dots,z_{m-k}$), otherwise $g_{\underline{k},\underline{\tau},\underline{t}}$ 
will equal zero.
So, provided $t_{n-1}=a$ we can write, for every $z_1,\dots,z_m \in \Dee$:

\begin{equation}\label{eq:Ega} 
\begin{array}{c} 
\Ee\left[g_{\underline{k},\underline{\tau},\underline{t}}(z_1,\dots,z_m;\Zcal\cup\{z_1,\dots,z_m\})\right] \\
 \leq \\
\Ee\left[ g_{(k_1,\dots,k_{n-1});(\tau_1,\dots,\tau_{n-1});(t_1,\dots,t_{n-2})}
(z_1,\dots,z_{m-k};\Zcal\cup\{z_1,\dots,z_{m-k}\})\right] \cdot \\
\Ee\left[ g_{C^\tau_{k-1}}(z_{m-k+1},\dots,z_m; \Zcal \cup \{z_{m-k+1},\dots,z_m\})\right] \cdot \\
1_{\{\distH(z_{m-k},z_{m-k+1})<r+w_2\}},
\end{array} 
\end{equation}

\noindent
where $g_{C^\tau_{k-1}}$ is the indicator function that $z_{m-k+1},\dots,z_m$ 
forms a chunk of length $k-1$ and with final edge of type $\tau$ (with respect to the 
point set $\Zcal \cup\{z_{m-k+1},\dots,z_m\}$). 
Here we use that $g_{\underline{k},\underline{\tau},\underline{t}}=0$
unless the position of $z_1,\dots,z_m$ is such that 
$\cert(z_{m-k+1},\dots,z_m) \cap \cert(P_i) = \emptyset$ for $i=1,\dots,n-1$;
and if $z_1,\dots,z_m$ are such that all these intersections are empty then 
the event that $g_{(k_1,\dots,k_{n-1});(\tau_1,\dots,\tau_{n-1});(t_1,\dots,t_{n-2})}
(z_1,\dots,z_{m-k};\Zcal\cup\{z_1,\dots,z_{m-k}\})=1$ and the event that 
$g_{C^\tau_{k-1}}(z_{m-k+1},\dots,z_m; \Zcal \cup \{z_{m-k+1},\dots,z_m\})=1$ are independent.

We can write

\begin{equation}\label{eq:ESa} 
\begin{array}{rcl} 
\Ee S_{\underline{k},\underline{\tau},\underline{t}} 
& \leq & \displaystyle 
(p\lambda)^{m-k} \int_\Dee\dots\int_\Dee \\[2ex]
& & \displaystyle 
\Ee\left[ g_{(k_1,\dots,k_{n-1});(\tau_1,\dots,\tau_{n-1});(t_1,\dots,t_{n-2})}
(z_1,\dots,z_{m-k};\Zcal\cup\{z_1,\dots,z_{m-k}\})\right] \cdot \\[2ex]
& & \displaystyle
{\bigg (} p\lambda \int_\Dee 1_{\{\distH(z_{m-k},z_{m-k+1})<r+w_2\}} \cdot 
{\bigg (} (p \lambda)^{k-1} \int_\Dee\dots\int_\Dee  \\[2ex]
& & \displaystyle  
\hspace{5ex} \Ee\left[ g_{C^\tau_{k-1}}(z_{m-k+1},\dots,z_m; \Zcal \cup \{z_{m-k+1},\dots,z_m\})\right] \cdot \\
& & \displaystyle 
 \hspace{5ex}   f(z_{m-k+1}) \dots f(z_m) \dd z_m \dots \dd z_{m-k+1} {\bigg )}  
f(z_{m-k+1}) \dd z_{m-k+1} {\bigg )} \cdot \\[2ex]    
& & \displaystyle 
\hspace{5ex} f(z_1)\dots f(z_{m-k}) \dd z_{m-k}\dots\dd z_1
\end{array} \end{equation}

\noindent
For every fixed $z_1,\dots,z_{m-k+1}$, applying an isometry that maps $z_{m-k+1}$ to $o$
and Corollary~\ref{cor:SlivMeck2col}, we have 

\begin{equation}\label{eq:ECkmin1} 
\begin{array}{c} 
\displaystyle 
(p\lambda)^{k-1} \int_\Dee\dots\int_\Dee
\Ee\left[ g_{C^\tau_{k-1}}(z_{m-k+1},\dots,z_m; \Zcal \cup \{z_{m-k+1},\dots,z_m\})\right] \cdot \\
\displaystyle \hspace{5ex}
f(z_{m-k+2}) \dots f(z_m) \dd z_m \dots \dd z_{m-k+2} \\
= \\
\displaystyle 
(p\lambda)^{k-1} \int_\Dee\dots\int_\Dee
\Ee\left[ g_{C^\tau_{k-1}}(o, z_{1},\dots,z_{k-1}; \Zcal \cup \{o,z_{1},\dots,z_{k-1}\})\right] \cdot \\
\displaystyle \hspace{5ex}
f(z_{1}) \dots f(z_{k-1}) \dd z_1 \dots \dd z_{k-1} \\
= \\
\Ee C_{k-1}^\tau.  
\end{array} 
\end{equation}

\noindent
Of course, for every fixed $z_{m-k} \in \Dee$, we have

$$ \begin{array}{rcl} 
\displaystyle 
p\lambda \int_\Dee 1_{\{\distH(z_{m-k},z_{m-k+1})<r+w_2\}}  f(z_{m-k+1}) \dd z_{m-k+1} 
& = &  
\displaystyle 
p\lambda \cdot \areaH( \ballH(o,r+w_2) ) \\
& \leq & 
10^3 \cdot e^{w_2}, 
\end{array} $$

\noindent
by Lemma~\ref{lem:XII}.
It follows that:

\begin{equation}\label{eq:Sa} \begin{array}{rcl} 
\Ee S_{\underline{k},\underline{\tau},\underline{t}} 
& \leq  & 
\displaystyle 
10^3 \cdot e^{w_2} \cdot \Ee C_{k-1}^\tau \cdot
(p\lambda)^{m-k} \int_\Dee\dots\int_\Dee \\[2ex]
& & \displaystyle 
\Ee\left[ g_{(k_1,\dots,k_{n-1});(\tau_1,\dots,\tau_{n-1});(t_1,\dots,t_{n-2})}
(z_1,\dots,z_{m-k};\Zcal\cup\{z_1,\dots,z_{m-k}\})\right] \cdot \\[2ex]
& & \displaystyle 
f(z_1)\dots f(z_{m-k}) \dd z_{m-k}\dots\dd z_1 \\[2ex]
& = & \displaystyle 
10^3 \cdot e^{w_2} \cdot \Ee C_{k-1}^\tau \cdot \Ee S_{(k_1,\dots,k_n),(\tau_1,\dots,\tau_{n-1}),(t_1,\dots,t_{n-2})}.
\end{array} \end{equation}

\noindent
(Provided $t_{n-1}=a$.)

If $t_{n-1}=b$ then $z_{m-k}$, the first vertex of $P_n$, does not lie on $P_{n-1}$ and 
we need that $\distH(z_{m-k},z_{m-k+1}) < r+w_2$ and $\cert(z_{m-k},z_{m-k+1})$ intersects $\cert(P_{n-1}) 
= \cert(z_{m-k_{n-1}-k-2}, \dots, z_{m-k-1})$. 
So instead of~\eqref{eq:Ega} we can now write

\begin{equation}\label{eq:Egb} 
\begin{array}{c} 
\Ee\left[g_{\underline{k},\underline{\tau},\underline{t}}(z_1,\dots,z_m;\Zcal\cup\{z_1,\dots,z_m\})\right] \\
 \leq \\
\Ee\left[ g_{(k_1,\dots,k_{n-1});(\tau_1,\dots,\tau_{n-1});(t_1,\dots,t_{n-1})}
(z_1,\dots,z_{m-k};\Zcal\cup\{z_1,\dots,z_{m-k}\})\right] \cdot \\
\Ee\left[ g_{C^\tau_{k-1}}(z_{m-k+1},\dots,z_m; \Zcal \cup \{z_{m-k+1},\dots,z_m\})\right] \cdot \\
1_{\left\{\begin{array}{l}\cert(z_{m-k},z_{m-k+1})\cap \cert(z_{m-k_{n-1}-k-2}, \dots, z_{m-k-1})\neq \emptyset, \atop 
\distH(z_{m-k},z_{m-k+1})<r+w_2 \end{array}\right\}} 
\end{array} 
\end{equation}

\noindent
Arguing as in~\eqref{eq:ESa} and reusing~\eqref{eq:ECkmin1} we find that

$$ \begin{array}{rcl} 
\Ee S_{\underline{k},\underline{\tau},\underline{t}} 
& \leq  & 
\displaystyle 
\Ee C_{k-1}^\tau \cdot
(p\lambda)^{m-(k+1)} \int_\Dee\dots\int_\Dee \\[2ex]
& & \displaystyle 
\Ee\left[ g_{(k_1,\dots,k_{n-1});(\tau_1,\dots,\tau_{n-1});(t_1,\dots,t_{n-2})}
(z_1,\dots,z_{m-k};\Zcal\cup\{z_1,\dots,z_{m-k-1}\})\right] \cdot \\[2ex]
& & \displaystyle 
{\bigg (} (p\lambda)^2 \int_\Dee\int_\Dee
1_{\left\{\begin{array}{l}\cert(z_{m_k},z_{m-k+1})\cap \cert(z_{m-k_{n-1}-k-2}, \dots, z_{m-k-1})\neq \emptyset, \atop 
\distH(z_{m_k},z_{m-k+1})<r+w_2 \end{array}\right\}} \cdot \\[2ex]
& & \displaystyle f(z_{m-k})f(z_{m-k+1})
\dd z_{m-k} \dd z_{m-k+1} {\bigg )} \cdot \\
& & \displaystyle 
f(z_1)\dots f(z_{m-k}) \dd z_{m-k}\dots\dd z_1 
\end{array} $$ 

For any fixed $z_1,\dots,z_{m-k-1} \in \Dee$ we have that if
$z_{m-k_{n-1}-k-2}, \dots, z_{m-k-1}$ is not a pre-chunk then
$g_{(k_1,\dots,k_{n-1});(\tau_1,\dots,\tau_{n-1});(t_1,\dots,t_{n-2})}
(z_1,\dots,z_{m-k};\Zcal\cup\{z_1,\dots,z_{m-k-1}\}) = 0$ and otherwise 
we can apply Lemma~\ref{lem:FIb} to get

$$ \begin{array}{c} \displaystyle 
(p\lambda)^2 \int_\Dee\int_\Dee
1_{\left\{\begin{array}{c} \cert(z_{m_k},z_{m-k+1})\cap \cert(z_{m-k_{n-1}-k-2}, \dots, z_{m-k-1})\neq \emptyset, \\ 
\distH(z_{m_k},z_{m-k+1})<r+w_2\end{array}\right\}} \cdot \\[3ex]
 \displaystyle f(z_{m-k})f(z_{m-k+1}) \dd z_{m-k} \dd z_{m-k+1} \\
\leq \\
10^5 \cdot k_{n-1} \cdot e^{2v(\tau_{n-1})+2} \\
\leq \\
10^6 \cdot k_{n-1} \cdot e^{2v(\tau_{n-1})}, 
\end{array} $$

\noindent
where $v(\sigma) = w_2$ if $\sigma \in \{\text{\bf{II}}, \text{\bf{III}}\}$ and 
otherwise $v(\sigma)$ is determined via $\sigma =: (\text{\bf{IV}},v(\sigma)-1)$.
(That is, if $\sigma =  (\text{\bf{IV}},x)$ for some $x\geq w_2$ then $v(\sigma) = x+1$.)
We conclude that if $t_{n-1}=b$ then 

\begin{equation}\label{eq:Sb} 
\Ee S_{\underline{k},\underline{\tau},\underline{t}} 
\leq 10^6 \cdot k_{n-1} \cdot e^{2v(\tau_{n-1})} \cdot \Ee C_{k-1}^\tau \cdot
\Ee S_{(k_1,\dots,k_{n-1});(\tau_1,\dots,\tau_{n-1});(t_1,\dots,t_{n-2})}. 
\end{equation}

If $t_{n-1}=c$ then we replace~\eqref{eq:Egb} with

\begin{equation}\label{eq:Egc} 
\begin{array}{c} 
\Ee\left[g_{\underline{k},\underline{\tau},\underline{t}}(z_1,\dots,z_m;\Zcal\cup\{z_1,\dots,z_m\})\right] \\
 \leq \\
\Ee\left[ g_{(k_1,\dots,k_{n-1});(\tau_1,\dots,\tau_{n-1});(t_1,\dots,t_{n-1})}
(z_1,\dots,z_{m-k};\Zcal\cup\{z_1,\dots,z_{m-k}\})\right] \cdot \\
\Ee\left[ g_{C^\tau_{k-1}}(z_{m-k+1},\dots,z_m; \Zcal \cup \{z_{m-k+1},\dots,z_m\})\right] \cdot \\
1_{\left\{\begin{array}{l}\distH(z_{m-k},\cert(z_{m-k_{n-1}-k-2}, \dots, z_{m-k-1}))<r/1000, \\ 
\distH(z_{m-k},z_{m-k+1})<r+w_2, \\
\cert(z_{m-k},z_{m-k+1}) \cap \cert(z_{m-k_{n-1}-k-2}, \dots, z_{m-k-1}) = \emptyset, \\
z_{m-k}z_{m-k+1} \text{ is a pseudoedge. }
\end{array}\right\}}.
\end{array} 
\end{equation}

\noindent
Arguing as before, but using Lemma~\ref{lem:FIc} in place of Lemma~\ref{lem:FIb} gives:

\begin{equation}\label{eq:Sc}
\Ee S_{\underline{k},\underline{\tau},\underline{t}} 
\leq k_{n-1} \cdot e^{v(\tau_{n-1})/2} \cdot \Ee C_{k-1}^\tau \cdot
\Ee S_{(k_1,\dots,k_{n-1});(\tau_1,\dots,\tau_{n-1});(t_1,\dots,t_{n-2})}.
\end{equation}

\noindent
(Provided $t_{n-1}=c$.)

Combining~\eqref{eq:Ssum} and~\eqref{eq:Sa},~\eqref{eq:Sb},~\eqref{eq:Sc} gives

$$ \begin{array}{rcl} 
\Ee S^\tau_{n,k}
& = & \displaystyle 
\sum_{{\underline{k} \in \eN^n, \underline{\tau} \in \Tscr^n, \underline{t} \in \Lscr^{n-1},}\atop \tau_n=\tau,k_n=k} 
\Ee S_{\underline{k},\underline{\tau},\underline{t}} \\[5ex]
& \leq & \displaystyle 
\sum_{{k_1,\dots,k_{n-1}\in\eN}\atop{{\tau_1,\dots,\tau_{n-1}\in\Tscr}\atop t_1,\dots,t_{n-2}\in\Lscr}} 
\Ee C_{k-1}^\tau \cdot 
\left(10^3 e^{w_2} + 10^6 \cdot k_{n-1} \cdot e^{2v(\tau_{n-1})} + k_{n-1} \cdot e^{v(\tau_{n-1})/2} \right) \cdot \\
& & \displaystyle \hspace{20ex} 
\Ee S_{(k_1,\dots,k_{n-1});(\tau_1,\dots,\tau_{n-1});(t_1,\dots,t_{n-2})} \\[5ex]
& \leq & \displaystyle 
\sum_{{k_1,\dots,k_{n-1}\in\eN}\atop{{\tau_1,\dots,\tau_{n-1}\in\Tscr}\atop t_1,\dots,t_{n-2}\in\Lscr}} 
\Ee C_{k-1}^\tau \cdot 10^7 \cdot k_{n-1} \cdot e^{2v(\tau_{n-1})} \cdot 
\Ee S_{(k_1,\dots,k_{n-1});(\tau_1,\dots,\tau_{n-1});(t_1,\dots,t_{n-2})} \\[5ex]
& = & \displaystyle 
10^7 \cdot \Ee C_{k-1}^\tau \cdot \sum_{\ell=1}^\infty  \ell \cdot 
\sum_{\sigma\in\Tscr} e^{2v(\sigma)} \cdot
\sum_{{k_1,\dots,k_{n-1}\in\eN}\atop{{{\tau_1,\dots,\tau_{n-1}\in\Tscr}\atop t_1,\dots,t_{n-2}\in\Lscr}\atop k_{n-1}=\ell,\tau_{n-1}=\sigma}} 
 \Ee S_{(k_1,\dots,k_{n-1});(\tau_1,\dots,\tau_{n-1});(t_1,\dots,t_{n-2})} \\[5ex]
& = & \displaystyle 
\Ee C_{k-1}^\tau \cdot 10^7 \cdot \sum_{\ell=1}^\infty  \ell \cdot 
\sum_{\sigma\in\Tscr} e^{2v(\sigma)} \cdot \Ee S_{n-1,\ell}^\sigma \\[5ex]
& = & \displaystyle 
\Ee C_{k-1}^\tau \cdot 10^7 \cdot \sum_{\ell=1}^\infty  \ell \cdot 
\left( e^{2w_2} \cdot \Ee S_{n-1,\ell}^{\text{\bf{II}}}
+ e^{2w_2} \cdot \Ee S_{n-1,\ell}^{\text{\bf{III}}} 
+ \sum_{v\geq w_2} e^{2v+2} \cdot \Ee S_{n-1,\ell}^{\text{\bf{IV}},v} \right) \\[5ex]
& = & \displaystyle
\Ee C_{k-1}^\tau \cdot \Sigma_n,
\end{array} $$ 

\noindent
establishing the last inequality in the lemma statement.

The first inequality (the case when $k=1$ and $\tau=\text{\bf{II}}$) follows analogously, replacing
$g_{C_{k-1}^\tau}$ in~\eqref{eq:Ega},~\eqref{eq:Egb},~\eqref{eq:Egc} by 
$1_{\{\distH( z_{m-1},z_m ) < r-w_1\}}$ and applying Lemma~\ref{lem:XII} to the 
innermost integral in the analogues of~\eqref{eq:ESa}.

The second inequality (the case when $k=2$ and $\tau=\text{\bf{III}}$) follows in the same way, 
now using the indicator function that $\distH(z_{m-2},z_{m-1}),\distH(z_{m-1},z_m) < r+w_2$ and 
$\angle z_{m-2}z_{m-1}z_m < \vartheta$ and applying Lemma~\ref{lem:XIII}.

For the case when $k=1$ and $\tau=(\text{\bf{IV}},v)$ we replace 
$g_{C_{k-1}^\tau}$ by the indicator function corresponding to $F_{\text{IV-a},\geq v}$ from Lemma~\ref{lem:FIVa} in case
$t_{n-1}=a$; by the indicator function corresponding to $F_{\text{IV-b},\geq v}$ from Lemma~\ref{lem:FIVb} in case
$t_{n-1}=b$; and the indicator function corresponding to $F_{\text{IV-c},\geq v}$ form Lemma~\ref{lem:FIVc} in case
$t_{n-1}=c$; and then apply Lemma~\ref{lem:FIVa}, respectively Lemma~\ref{lem:FIVb}, respectively Lemma~\ref{lem:FIVc}
to the innermost integral in the analogues of~\eqref{eq:ESa}.
\end{proof}

\begin{lemma}\label{lem:noinflinkedseq}
For every $0<\eps<1$ and $c>0$ there exist $\lambda_0,w_1,\vartheta>0$ and $w_2>c$ such that 
$w_2 \in\eN$ and for all $0<\lambda<\lambda_0$ and $p\leq(1-\eps)\cdot(\pi/3)\cdot\lambda$, 
setting $r := 2\ln(1/\lambda)$, 
almost surely there are no infinite, black, linked sequences of chunks starting from the origin $o$.
\end{lemma}

\begin{proof}
We will make use of the bounds provided by Lemma~\ref{lem:numbergoodpspths}, Corollary~\ref{cor:Ck} and Lemma~\ref{lem:recurs}.
We will choose the integer $w_2>c$ so large that 
$10^4 e^{w_2} e^{-w_2/2} < \eps/6$. So in particular, for $\lambda$ sufficiently small, the expected number of good, black 
pseudopaths starting from the origin satisfies

$$ \Ee G_k \leq (1-\eps/3)^k. $$

Let $\Sigma_n$ be as in the statement of Lemma~\ref{lem:recurs}.
By the observation~\eqref{eq:S1isC} and Corollary~\ref{cor:Ck}, we have

$$ \begin{array}{rcl} 
\Sigma_1 
& = & \displaystyle
10^7 \cdot \sum_{\ell=1}^\infty \ell \cdot \left( e^{2w_2} \cdot \Ee C_{\ell}^{\text{\bf{II}}} 
 + e^{2w_2} \cdot \Ee C_{\ell}^{\text{\bf{III}}} 
 + \sum_{v\geq w_2} e^{2v} \cdot \Ee C_{\ell}^{\text{\bf{IV}},v} \right) \\
& \leq & \displaystyle
10^7 \cdot \left( 10^3 \cdot e^{2w_2} \cdot \left(e^{-w_1}+\vartheta e^{w_2}\right) + 10^4 \cdot 
\sum_{v=w_2}^\infty e^{(5/2)v} e^{-e^{v/2}} \right)
\cdot \left( \sum_{\ell=1}^\infty \ell\cdot(1-\eps/3)^{\ell-1}\right)  \\
& < & 
\infty. 
\end{array} $$

\noindent
The recursive relations given by Lemma~\ref{lem:recurs} show that for all $n\geq 1$:

$$ \begin{array}{rcl} 
\Sigma_{n+1} 
& \leq & \displaystyle 
\Sigma_n \cdot 10^7 \cdot \left( e^{2w_2} \cdot 10^3 \cdot e^{-w_1} + 2 e^{2w_2} \cdot 10^3 \cdot \vartheta \cdot e^{w_2} 
+ 10^6 \cdot \sum_{v=w_2} e^{3v} e^{-e^{v/2}} + \right. \\
& & \displaystyle \hspace{10ex} \left.
\sum_{\ell=2}^\infty  e^{2w_2} \cdot \ell \cdot \Ee C_{\ell-1}^{\text{\bf{II}}}
+ \sum_{\ell=3}^\infty e^{2w_2} \cdot \ell \cdot \Ee C_{\ell-1}^{\text{\bf{III}}}
+ \sum_{\ell=2}^\infty e^{2v} \cdot \ell \cdot \Ee C_{\ell-1}^{\text{\bf{IV}},v}
\right) \\[2ex]
& \leq & \displaystyle
\Sigma_n \cdot 10^7 \cdot \left( 10^3 \cdot e^{2w_2-w_1} + 2 \cdot 10^3 \cdot \vartheta \cdot e^{3w_2} 
+ 10^6 \cdot \sum_{v=w_2}^\infty e^{3v} e^{-e^{v/2}} + \right. \\
& & \displaystyle 
\left.
  \left(10^3 \cdot e^{2w_2-w_1} + 10^3 \cdot \vartheta \cdot e^{3w_2}
  + 10^4 \cdot \sum_{v=w_2}^\infty e^{(5/2)v} \cdot e^{-e^{v/2}} \right)
  \cdot \left( \sum_{\ell=2}^\infty \ell \cdot (1-\eps/3)^{\ell-2} \right) \right) \\[2ex]
& \leq & \displaystyle 
\Sigma_n \cdot 10^{13} \cdot \left( e^{2w_2-w_1} + \vartheta \cdot e^{3w_2} 
+ \sum_{v=w_2}^\infty e^{3v} e^{-e^{v/2}} \right) \cdot  
\left( 1+\sum_{\ell=2}^\infty \ell \cdot (1-\eps/3)^{\ell-2} \right) \\[2ex]
& = & \displaystyle 
\Sigma_n \cdot 10^{13} \cdot \left( e^{2w_2-w_1} + \vartheta \cdot e^{3w_2} 
+ \sum_{v=w_2}^\infty e^{3v} e^{-e^{v/2}} \right) \cdot  
\left( \frac{9+3\eps+\eps^2}{\eps^2} \right) \\[2ex]
& \leq & \displaystyle 
\Sigma_n \cdot 10^{15} \cdot \eps^{-2} \cdot \left( e^{2w_2-w_1} + \vartheta \cdot e^{3w_2} 
+ \sum_{v=w_2}^\infty e^{3v} e^{-e^{v/2}} \right) 
\end{array} $$

\noindent
For the sake of the presentation, we introduce an additional small constant $\delta>0$.
Without loss of generality we can assume chose $w_2>c$ large enough so that

$$ \sum_{v=w_2}^\infty e^{3v} \cdot e^{-e^{v/2}} < \delta. $$

We can also assume $w_1,\vartheta$ are such that $e^{2w_2-w_1}, \vartheta e^{3w_2} < \delta$.
Filling these bounds into the bound on $\Sigma_{n+1}$ gives 

$$ \Sigma_{n+1} \leq 
\Sigma_n \cdot 10^{15} \cdot \left(\frac{3\delta}{\eps^2}\right) 
\leq
\Sigma_n \cdot \left(\frac12\right), $$
  
\noindent
the last line holding because we have chosen $\delta$ appropriately.

To recap, we can choose $w_1,w_2,\vartheta$ in such a way that 
$\Sigma_n \leq \left(\frac12\right)^{n-1}\cdot \Sigma_1$ for all $n$; and $\Sigma_1<\infty$.
In particular $\Sigma_n \to 0$ as $n\to\infty$.
 
Let us denote by $S_n$ the total number of black, linked sequences of chunks starting from the origin and
consisting of precisely $n$ chunks. Obviously

$$ \Ee S_n = \sum_{k=1}^\infty \Ee S_{n,k}^{\text{\bf{II}}} + \sum_{k=1}^\infty \Ee S_{n,k}^{\text{\bf{III}}}
+ \sum_{k=1}^\infty\sum_{v=w_2}^\infty \Ee S_{n,k}^{\text{\bf{IV}},v} \leq \Sigma_n. $$

\noindent
But then we also have that 

$$ \Ee S_n \xrightarrow[n\to\infty]{} 0. $$

\noindent
It follows that, almost surely, there are is no infinite, black, linked sequence of chunks starting from $o$.
\end{proof}

Combining Proposition~\ref{prop:lbprep}, Corollary~\ref{cor:noinfgoodpaths} and Lemma~\ref{lem:noinflinkedseq}
immediately gives:

\begin{corollary}\label{cor:lbtypical}
For every $\eps>0$ there is a $\lambda_0=\lambda_0(\eps)$ such that 
for all $0<\lambda<\lambda_0$ and all $p\leq(1-\eps)\cdot(\pi/3)\cdot\lambda$, almost surely,
the black cluster of $o$ in the Voronoi tessellation for $\Zcal\cup\{o\}$ is finite.
\end{corollary}

For completeness we point out how Proposition~\ref{prop:lb} follows from Corollary~\ref{cor:lbtypical}.

\begin{proofof}{Proposition~\ref{prop:lb}}
This follows from Corollary~\ref{cor:lbtypical} by a near verbatim repeat of the proof of Corollary~\ref{cor:noinfgoodpaths}.
We only mention the changes that need to be made.
Now, we let $N_x$ be the number of $z\in\Zcal_b \cap \ballH(o,x)$ that are part of an infinite black component of 
the Voronoi tessellation, and we let $g$ be the indicator function that $z\in \ballH(o,x)$ and 
that $z$ lies in an infinite, black cluster in the Voronoi tessellation for $\Zcal\cup\{z\}$.
\end{proofof}

\section{Suggestions for further work\label{sec:discussion}}

Our Theorem~\ref{thm:main} states that $p_c(\lambda) = (\pi/3)\lambda + o(\lambda)$ as $\lambda\searrow 0$, answering a 
question of Benjamini and Schramm~\cite{benjamini2001percolation}.
A natural direction for research is to try and find more terms in the expansion.

\begin{problem}\label{prob:prob1}
Determine a constant $c$ such that $p_c(\lambda) = (\pi/3)\lambda + c\lambda^2 + o(\lambda^2)$ as
$\lambda\searrow 0$, or show no such constant exists.
\end{problem}

In our proofs we have either taken $p \geq (1-\eps)\cdot(\pi/3)\cdot\lambda$, for the upper bound on 
$p_c(\lambda)$, or $p\leq(1-\eps)\cdot(\pi/3)\cdot\lambda$, for the lower bound. We have always taken $\eps$ constant in 
our paper and have not made any effort to see how fast we can send $\eps$ to zero as a function of $\lambda$ 
before our proof technique breaks down.
Of course many of our bounds are rather crude and it seems likely that more fine-grained proof techniques will need to be 
developed.

As mentioned in the introduction, the critical value for percolation in $\Zed^d$ for high dimension $d$ is related, at least
on an intuitive level, to our result. And, as also mentioned, a trivial comparison 
to branching processes shows the critical value for bond percolation on $\Zed^d$ 
satisfies $p_c(\Zed^d) \geq \frac{1}{2d-1} > \frac{1}{2d}$, the reciprocal of the degree.
(It was in fact shown by Van der Hofstad and Slade~\cite{VdhofstadSlade06} that
$p_c(\Zed^d) = \frac{1}{2d} + \frac14 d^{-2} + \frac{7}{16} d^{-3} + O( d^{-4} )$. 
So we even have $p_c(\Zed^d) > \frac{1}{2d-1}$ for large enough $d$.)
Inspired by this, and the fact that the typical degree is actually strictly larger that $\frac{3}{\pi\lambda}$, we
offer the following conjecture:

\begin{conjecture}\label{conj:conj1}
There exists a $\lambda_0>0$ such that $p_c(\lambda) > (\pi/3)\cdot\lambda$ for all $0<\lambda<\lambda_0$.
\end{conjecture}

Of course the answer to Problem~\ref{prob:prob1} is likely to also tell us whether this conjecture holds or not.
But, perhaps Conjecture~\ref{conj:conj1} can be settled via a different route.

We reiterate a natural conjecture of Benjamini and Schramm.

\begin{conjecture}[\cite{benjamini2001percolation}]\label{conj:conj2}
$p_c(\lambda)$ is strictly increasing.
\end{conjecture}

Another natural conjecture in the same vein is:

\begin{conjecture}\label{conj:conj3}
 $p_c(\lambda)$ is differentiable.
\end{conjecture}

\noindent
We were tempted to write ``smooth'' in place of ``differentiable'', but opted against it to give whoever 
attempts to prove the conjecture the best chances.

Of course Poisson-Voronoi percolation can also be defined on $d$-dimensional hyperbolic space $\Haa^d$, and
we expect that the main result of the current paper will generalize.
In two dimensions, the typical degree is the same as the number of $1$-faces of the typical cell.
In $d$-dimensions the relevant corresponding quantity is the number of $(d-1)$-faces of the typical cell. 

\begin{conjecture}\label{conj:conj4}
For every $d$, as the intensity $\lambda\searrow 0$, the critical value for Poisson-Voronoi percolation on $\Haa^d$ 
is asymptotically equal to the reciprocal of the expected number of $(d-1)$-faces of the typical cell.
\end{conjecture}

\noindent
It might be possible to leverage some of the existing work on the expected $f$-vectors of the typical cell 
in~\cite{CalkaChapronEnriquez,Zakhar,Isokawa3d}. But, of course it might also be possible to prove or disprove 
the conjecture without knowing the expected number of $(d-1)$-faces precisely.

In a recent separate paper, we proved the conjecture of Benjamini and Schramm that $p_c(\lambda)\to1/2$ as $\lambda\to\infty$
for planar, hyperbolic Poisson-Voronoi percolation.
It seems natural to expect that this result generalizes to arbitrary dimensions.

\begin{conjecture}\label{conj:conj5}
For every $d$, as the intensity $\lambda\to\infty$, the critical value for Poisson-Voronoi percolation on $\Haa^d$ 
tends to the critical value for Poisson-Voronoi percolation on $\eR^d$.
\end{conjecture}

A complicating issue here is that for Poisson-Voronoi percolation on $\eR^d$ there is a large  
gap between the best known lower and upper bounds~\cite{BalisterBollobas10, BalisterBollobasQuas05}.
But, again, it may be possible to prove the conjecture without first determining the precise critical value in the 
Euclidean case.

\section*{Acknowledgements}

We thank the anonymous referee for many helpful comments that have greatly improved the paper.

 \bibliographystyle{plain}
 \bibliography{voronoibib}

\end{document}

%% file: euclprep1.pspdftex
\begin{picture}(0,0)%
\includegraphics{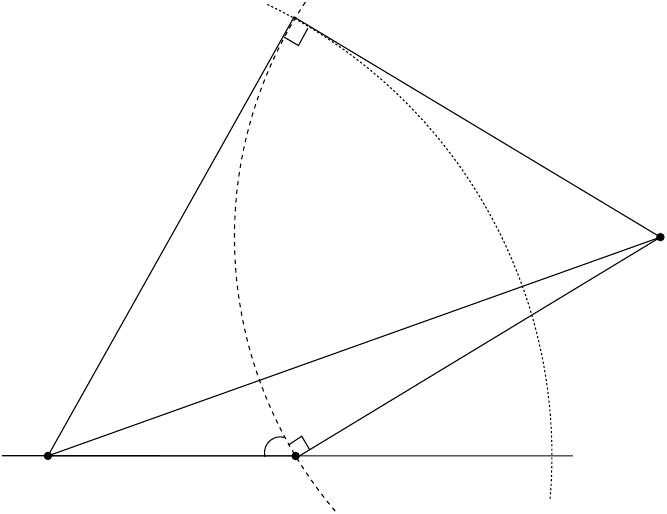}%
\end{picture}%
\setlength{\unitlength}{2763sp}%
\begingroup\makeatletter\ifx\SetFigFont\undefined%
\gdef\SetFigFont#1#2#3#4#5{%
  \reset@font\fontsize{#1}{#2pt}%
  \fontfamily{#3}\fontseries{#4}\fontshape{#5}%
  \selectfont}%
\fi\endgroup%
\begin{picture}(4561,3516)(4473,-3758)
\put(7530,-1782){\makebox(0,0)[lb]{\smash{{\SetFigFont{8}{9.6}{\rmdefault}{\mddefault}{\updefault}{\color[rgb]{0,0,0}$\partial\Dee$}%
}}}}
\put(6068,-3274){\makebox(0,0)[lb]{\smash{{\SetFigFont{8}{9.6}{\rmdefault}{\mddefault}{\updefault}{\color[rgb]{0,0,0}$\vartheta$}%
}}}}
\put(5458,-2102){\rotatebox{60.0}{\makebox(0,0)[lb]{\smash{{\SetFigFont{8}{9.6}{\rmdefault}{\mddefault}{\updefault}{\color[rgb]{0,0,0}$1$}%
}}}}}
\put(7655,-2868){\rotatebox{30.0}{\makebox(0,0)[lb]{\smash{{\SetFigFont{8}{9.6}{\rmdefault}{\mddefault}{\updefault}{\color[rgb]{0,0,0}$r$}%
}}}}}
\put(7576,-935){\rotatebox{330.0}{\makebox(0,0)[lb]{\smash{{\SetFigFont{8}{9.6}{\rmdefault}{\mddefault}{\updefault}{\color[rgb]{0,0,0}$r$}%
}}}}}
\put(6528,-2659){\rotatebox{20.0}{\makebox(0,0)[lb]{\smash{{\SetFigFont{8}{9.6}{\rmdefault}{\mddefault}{\updefault}{\color[rgb]{0,0,0}$\sqrt{1+r^2}$}%
}}}}}
\put(4732,-3553){\makebox(0,0)[lb]{\smash{{\SetFigFont{8}{9.6}{\rmdefault}{\mddefault}{\updefault}{\color[rgb]{0,0,0}$o$}%
}}}}
\put(6423,-3559){\makebox(0,0)[lb]{\smash{{\SetFigFont{8}{9.6}{\rmdefault}{\mddefault}{\updefault}{\color[rgb]{0,0,0}$u$}%
}}}}
\put(8976,-2071){\makebox(0,0)[lb]{\smash{{\SetFigFont{8}{9.6}{\rmdefault}{\mddefault}{\updefault}{\color[rgb]{0,0,0}$c$}%
}}}}
\put(8459,-3425){\makebox(0,0)[lb]{\smash{{\SetFigFont{8}{9.6}{\rmdefault}{\mddefault}{\updefault}{\color[rgb]{0,0,0}$x$-axis}%
}}}}
\put(6137,-1984){\makebox(0,0)[lb]{\smash{{\SetFigFont{8}{9.6}{\rmdefault}{\mddefault}{\updefault}{\color[rgb]{0,0,0}$C$}%
}}}}
\end{picture}%

%% file: BGab.pspdftex
\begin{picture}(0,0)%
\includegraphics{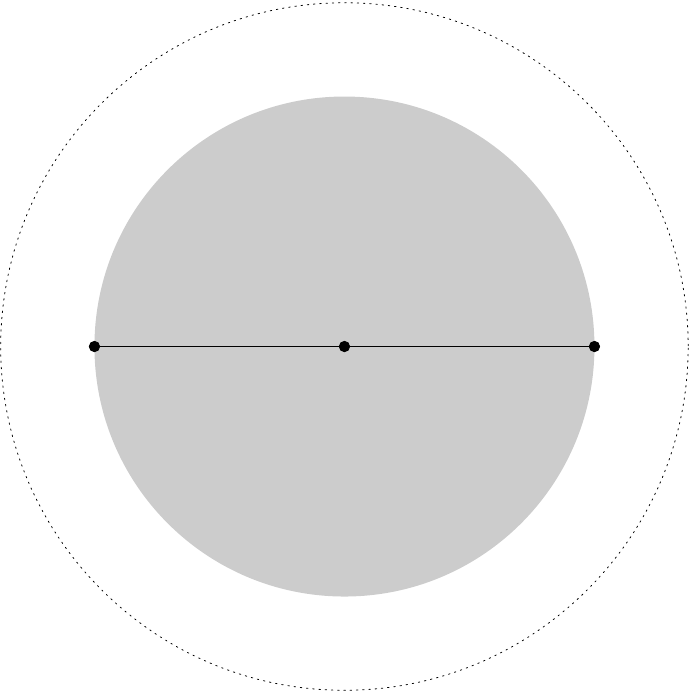}%
\end{picture}%
\setlength{\unitlength}{1973sp}%
\begingroup\makeatletter\ifx\SetFigFont\undefined%
\gdef\SetFigFont#1#2#3#4#5{%
  \reset@font\fontsize{#1}{#2pt}%
  \fontfamily{#3}\fontseries{#4}\fontshape{#5}%
  \selectfont}%
\fi\endgroup%
\begin{picture}(6616,6614)(1943,-5768)
\put(4651,-1261){\makebox(0,0)[lb]{\smash{{\SetFigFont{6}{7.2}{\rmdefault}{\mddefault}{\updefault}{\color[rgb]{0,0,0}$\BGab^-(z_1,z_2)$}%
}}}}
\put(8326,-886){\makebox(0,0)[lb]{\smash{{\SetFigFont{6}{7.2}{\rmdefault}{\mddefault}{\updefault}{\color[rgb]{0,0,0}$\partial\Dee$}%
}}}}
\put(4801,-3736){\makebox(0,0)[lb]{\smash{{\SetFigFont{6}{7.2}{\rmdefault}{\mddefault}{\updefault}{\color[rgb]{0,0,0}$\BGab^+(z_1,z_2)$}%
}}}}
\put(5401,-2311){\makebox(0,0)[lb]{\smash{{\SetFigFont{6}{7.2}{\rmdefault}{\mddefault}{\updefault}{\color[rgb]{0,0,0}$o$}%
}}}}
\put(3001,-2236){\makebox(0,0)[lb]{\smash{{\SetFigFont{6}{7.2}{\rmdefault}{\mddefault}{\updefault}{\color[rgb]{0,0,0}$z_1$}%
}}}}
\put(7801,-2236){\makebox(0,0)[lb]{\smash{{\SetFigFont{6}{7.2}{\rmdefault}{\mddefault}{\updefault}{\color[rgb]{0,0,0}$z_2$}%
}}}}
\end{picture}%

%% file: DD.pspdftex
\begin{picture}(0,0)%
\includegraphics{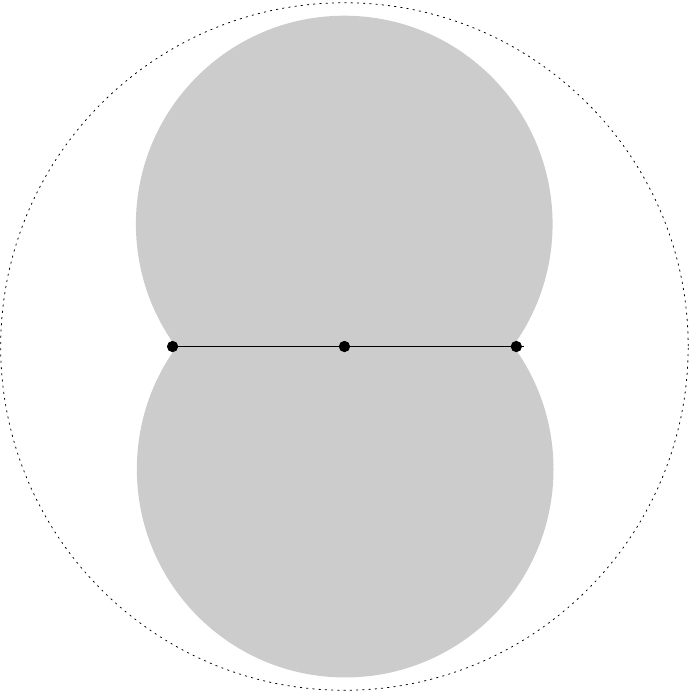}%
\end{picture}%
\setlength{\unitlength}{1973sp}%
\begingroup\makeatletter\ifx\SetFigFont\undefined%
\gdef\SetFigFont#1#2#3#4#5{%
  \reset@font\fontsize{#1}{#2pt}%
  \fontfamily{#3}\fontseries{#4}\fontshape{#5}%
  \selectfont}%
\fi\endgroup%
\begin{picture}(6616,6614)(1943,-5768)
\put(3601,-2236){\makebox(0,0)[lb]{\smash{{\SetFigFont{6}{7.2}{\rmdefault}{\mddefault}{\updefault}{\color[rgb]{0,0,0}$z_1$}%
}}}}
\put(4801,-3736){\makebox(0,0)[lb]{\smash{{\SetFigFont{6}{7.2}{\rmdefault}{\mddefault}{\updefault}{\color[rgb]{0,0,0}$\DD^+(z_1,z_2,\rho)$}%
}}}}
\put(5401,-2311){\makebox(0,0)[lb]{\smash{{\SetFigFont{6}{7.2}{\rmdefault}{\mddefault}{\updefault}{\color[rgb]{0,0,0}$o$}%
}}}}
\put(4651,-1261){\makebox(0,0)[lb]{\smash{{\SetFigFont{6}{7.2}{\rmdefault}{\mddefault}{\updefault}{\color[rgb]{0,0,0}$\DD^-(z_1,z_2,\rho)$}%
}}}}
\put(8326,-886){\makebox(0,0)[lb]{\smash{{\SetFigFont{6}{7.2}{\rmdefault}{\mddefault}{\updefault}{\color[rgb]{0,0,0}$\partial\Dee$}%
}}}}
\put(6901,-2236){\makebox(0,0)[lb]{\smash{{\SetFigFont{6}{7.2}{\rmdefault}{\mddefault}{\updefault}{\color[rgb]{0,0,0}$z_2$}%
}}}}
\end{picture}%